\documentclass[12pt, leqno]{amsart}
\usepackage{latexsym}
\usepackage{amssymb}
\usepackage{amsmath}
\usepackage[english]{babel}
\usepackage[colorlinks=true, pdfstartview=FitV, linkcolor=red, citecolor=red, backref=page]{hyperref} 

\usepackage{tikz, tikz-3dplot}
\usepackage{tkz-euclide}
\usetkzobj{all}
\usetikzlibrary{decorations.fractals}
\usetikzlibrary{decorations.footprints}
\usetikzlibrary{decorations.text}
\usepackage{sansmath}

\usetikzlibrary{shadings,intersections}
\usetikzlibrary{decorations.pathreplacing}
\usetikzlibrary{shapes.geometric}

\usepackage[top=1.25in, bottom=1.25in, left=1.25in, right=1.25in]{geometry}

\newtheorem{theorem}{Theorem}

\newtheorem{definition}[theorem]{Definition}

\newtheorem{lemma}[theorem]{Lemma}

\newtheorem{remark}[theorem]{Remark}

\newtheorem{mytheorem}{Theorem}

\newcommand{\rn}[1]{\mathbb{R}^{#1}}

\newcommand{\re}{ \mathbb{R}}

\newcommand{\beq}{\begin{equation}}
\newcommand{\bea}[1]{\begin{array}{#1} }
\newcommand{\eeq}{ \end{equation}}
\newcommand{\ea}{ \end{array}}
\newcommand{\ep}{\epsilon}
\newcommand{\es}{\emptyset}

\newcommand{\al}{\alpha}
\newcommand{\ga}{\gamma}

\newcommand{\de}{\delta}
\newcommand{\ds}{\displaystyle}
\newcommand{\ts}{\textstyle}
\newcommand{\rar }{\mbox{$\rightarrow$}}

\newcommand{\ran}{\rangle}
\newcommand{\lan}{\langle}
\newcommand{\Ga}{\Gamma}
\newcommand{\la}{\lambda}
\newcommand{\La}{\Lambda}
\newcommand{\ar}{\partial}
\newcommand{\si}{\sigma}
\newcommand{\Si}{\Sigma}
\newcommand{\om}{\omega}
\newcommand{\Om}{\Omega}
\newcommand{\be}{\beta}

\newcommand{\ph}{\phi}
\newcommand{\he}{\theta}
\newcommand{\He}{\Theta}
\newcommand{\Ph}{\Phi}
\newcommand{\hs}[1]{\mbox{$ \hspace{#1}$}}

\newcommand{\sem}{\setminus}
\newcommand{\ze}{\zeta}
\newcommand{\De}{\Delta}
\newcommand{\ti}{\tilde}

\renewcommand*{\backref}[1]{}
\renewcommand*{\backrefalt}[4]{%
 \ifcase #1 (Not cited.)%
   \or        (Cited on page~#2).%
    \else      (Cited on pages~#2).%
    \fi}

\begin{document} 

\title[Note on an eigenvalue problem with applications ]{Note on an eigenvalue problem with applications to  a   Minkowski  type  regularity problem in  $   \rn{n} $}

\author[M. Akman]{Murat Akman}
\address{{\bf Murat Akman}\\ Department of Mathematical Sciences, University of Essex\\ Wivenhoe Park, Colchester, Essex CO4 3SQ, UK} 
\email{murat.akman@essex.ac.uk}

\author[J. Lewis]{John Lewis}
\address{{\bf John Lewis} \\ Department of Mathematics \\ University of Kentucky \\ Lexington, Kentucky, 40506}
\email{johnl@uky.edu}

\author[A. Vogel]{Andrew Vogel}
\address{{\bf Andrew Vogel}\\ Department of Mathematics, Syracuse University \\  Syracuse, New York 13244}
\email{alvogel@syr.edu}


\begin{abstract}
We   consider   existence and uniqueness of  homogeneous  solutions  $ u > 0   $   to  certain    PDE  of   $p$-Laplace type,   
$ p $ fixed,   $ n - 1 <p< \infty,  n  \geq 2, $     when  $ u $ is  a solution  in  $K(\al)\subset\mathbb{R}^n$ where
\[  
K (\al)  :=    \{ x = (x_1, \dots, x_n ) : 
x_1   >  \cos \al  \,   | x|  \}  \quad \mbox{for fixed}\, \,   \al \in (0,  \pi ],     
\]      
 with continuous boundary value zero  on  $ \partial   K ( \al ) \sem \{0\}$.       
 In our main result  we show that  if    $ u $  has  continuous boundary value $0$  on  $  \ar  K  ( \pi )$  
 then  $u$ is homogeneous of   degree   $  1   -   (n-1)/p  $  when $ p  >  n - 1. $      
 Applications of  this result  are given  to  a   Minkowski  type  regularity problem  in $  \rn{n}$ when $n=2,3$.  
\end{abstract}

\maketitle
\setcounter{tocdepth}{1}
\tableofcontents

\section{Introduction}
Let $ u > 0  $  be  a homogeneous  $p$-harmonic  function 
 in  the cone $K(\alpha)\subset \mathbb{R}^n$, $n\geq 2$, with continuous boundary value 0 on $ \ar K ( \al ) \sem \{0\}$ where
  \[  
  K(\alpha)  :=    \{ x = (x_1, \dots, x_n ) : x_1   >  \cos  \alpha \, | x|  \} \quad \mbox{for}\, \, \al \in (0, \pi].
  \]  
More    specifically,   for fixed $ p, 1 <  p < \infty$, $u$ is  
a  weak solution  to   $ \nabla \cdot \left( |\nabla u |^{(p-2)} \,   \nabla u \right)  =  0 $ 
 in  $  K ( \al ) $  and 
\begin{align} 
\label{1.1} 
  u ( t x ) =  t^{\lambda} u ( x )\, \,   \mbox{for some  real $ \la $   whenever $ t > 0$ and  $x \in K ( \al )$.}  
\end{align} 
 Given $ x \in    \mathbb{R}^{n}\sem \{0\}$,  introduce spherical  coordinates  $ r =   |x|$ and  $x_1 =   r  \cos \he$ for $ 0 \leq  \he  \leq \pi$.   
  If $u$ as in \eqref{1.1}   is  $p$-harmonic in  $ K ( \al )$  and   $ u ( 1, 0, \dots, 0) = 1$   
   then using rotational  invariance of  the  $p$-Laplace equation, it turns out that   $ u  $  has additionally the following form            
   \begin{align} 
    \label{1.2} 
     u ( x ) = u(r, \theta ) =  r^{\lambda}  \, \phi ( \theta )\quad \mbox{for}\, \,  0 \leq   \theta  <  \alpha  \, \, \mbox{and} \, \, r>0 
    \end{align}    
    with $\phi (0) = 1$ and $\phi  ( \alpha )  = 0$     for  some    $\la(\al)= \la \in ( - \infty, \infty ) $  and $ \ph \in C^\infty ( [0, \al ])$.    
             
   It  was first shown by    
   Krol' and Maz'ya in \cite{KM}  that if   $ 1 < p  \leq n - 1 $  and   $  \al \in (0, \pi )$, $\alpha$ is near enough $ \pi$,  
   then there exists a unique  solution  to  \eqref{1.1} in $ K ( \al ) $   of the special form  
   \eqref{1.2}  with $ \la ( \al ) > 0$.         Tolksdorf in   \cite{To}   showed   that  given  $ \al \in  (0, \pi)$, for $i=1,2$,
    there  exist  unique  $ \la_i $ with $ \la_2 <  0  <  \la_1$ and $\ph_i$   where  $ \phi_i$ is 
    infinitely differentiable on  $ [0, \al] $  satisfying $    \phi_i ( \al ) = 0$ and $\phi_i ( 0  ) = 1$  
    and   $ u_i (r, \he )   =  r^{\la_i}  \phi_i (\he)$  are  solutions to the $p$- Laplace equation in $ K (\al)$.       
    Also  Porretta  and  V{\'e}ron     gave  another proof  of  Tolksdorf's result in  \cite{PV}.   A  similar  study was made in   more general  Lipschitz cones  by   Gkikas   and   V{\'e}ron   in  \cite{GV}.    
  
Next  we discuss   what is  known about  ``eigenvalues'' $   \la $ in \eqref{1.2}  for various 
 $ \al$ and $n$.       Krol'   in \cite{ K}    (see also \cite{A})  used \eqref{1.2} and  separation of variables  to show  for  $ u $  as in  \eqref{1.2}   that            
\begin{align*}
0 =&  \frac{d}{d\he}  \left\{ [ \la^2 \phi^2 (\he) + (\ph')^{2} (\he) ]^{(p-2)/2}\, \ph' ( \he ) \, (\sin \he )^{n-2} \right\}   \\
 &+  \la [ \la (p-1) + (n-p)  ]   [ \la^2 \phi^2 (\he) + (\ph')^{2} (\he)  ]^{(p-2)/2}  \ph ( \he )  (\sin \he)^{n-2}.   
\end{align*}  
 Letting  $ \psi = \ph' / \ph $ in the above equation he obtained, the first order DE  
\begin{align} 
\label{1.3}     
\begin{split}
  0 = &       ( (p-1) \psi^2 + \la^2 )  \,  \psi'      \\
           & +  ( \la^2  + \psi^2 )   [(p-1)  \, \psi^2   +   (n-2) \cot  \he \,  \psi      +  \la^2   (p-1) + \la (n-p) ].
\end{split}
\end{align}
  If  $ n = 2 $ the cotangent  term in the above DE  goes out and  variables can be  separated in \eqref{1.3}  to  get    
\[
  \frac{  \la  d \psi}{ \la^2 + \psi^2 }     -  \frac{ (\la - 1 ) \, d \psi }{ \la^2  + \psi^2  + 
\la (2-p)/(p-1) } +  d \he = 0.
\]
 The  boundary conditions imply  that  $ \ph $  is decreasing on $ ( 0, \al ) $  so $ \psi (\al ) = - \infty$ and $\psi (0) = 0$.  Using this fact and integrating it follows that  
\begin{equation}  
\label{1.4} 
\pm 1  -     \frac{ \la - 1}{\sqrt{ \la^2 + \la (2 - p)/(p-1) }} =\frac{2 \al}{\pi}    
\end{equation}  
where $ + 1 $ is taken  if  $ \la > 0 $ and  $ - 1 $ if  $ \la < 0. $  For later discussion we note that if  $ \al = \pi/2$, i.e., $ K (\pi/2)$  is a half-space, then  \eqref{1.4}  gives 
\[    
 \la_1 = 1 \quad \mbox{and}\quad \la_2 = \frac{p - 3  - 2  \sqrt{ p^2 - 3p + 3 }} {3(p - 1)}.
 \]
We remark that $\lambda_1= \la_1 ( \pi/2) = 1 $ for $ n \geq  2 $ since $ x_1 = r \cos \he $ is  $p$-harmonic  for  $ 1 < p < \infty$.     Also if $ \al = \pi$ and $n = 2, $  i.e., 
$ K ( \pi ) = \mathbb{R}^2 \sem (-\infty, 0]$, then  \eqref{1.4}  yields 
\[
 \la_1 =  1  -  1/p \quad \mbox{and} \quad    \la_2 =  (1/16) \left(  7p - 16  -  \sqrt{ 81 p^2  - 288 p  + 288} \, \right)/ (p-1).
 \]
 For other values of   $\la_2= \la_2 ( \al )  $ when $ n = 2$,  see  \cite{LuV}.  For  $ n  \geq 3$, $\al =  \pi/2$,   and $ p = 2$,  one can use the Kelvin transformation to get $ \la_2 (\pi/2)  =  1 - n   $ 
 while  if $p = n, $  it follows from conformal invariance of  the  $n$-Laplacian that  $  \la_2 ( \pi/2) = - 1$.     Also if  $ p   =   (4 n - 2)/3$ then
 \[  - 2 \, \la_2 ( \pi  )  =   \frac{p +  1  - n}{p - 1} =  \be  =       \frac{n+1}{4n - 5}
 \]   
since $u ( r,  \theta)     =   r^{- \beta/2}    (\cos ( \theta/2) )^{ \be}$ in   \eqref{1.2}  for $ \alpha  = \pi. $  DeBlassie and  Smits  in   \cite{DS}    obtained estimates on
  $-  \la_2 ( \pi/2)$, $1 < p\neq 2  < \infty$,   by leaving out  the  cotangent term in  
   \eqref{1.3}.    
  In fact  their  solution to  the    DE  in   \eqref{1.3}   with   the  cotangent term omitted leads  to   a  supersolution   of the form \eqref{1.2}  for the  $p$-Laplace equation, so leads to a lower   estimate for  $ - \la_2 (\pi/2) $    in   \eqref{1.3}.  
   Upper and   lower  estimates  for  $  \la_2 ( \al )$ for $\al \in  (0,\pi/2]$   were  also  obtained by these authors  in    \cite{DS1},   by finding  $p$-harmonic  subsolution and supersolution of  the  form   $ r^k  \ti \ph ( \he )$ where $ k   < 0 $  and  $ \ti \ph $  is the solution to    
   \eqref{1.2}  when $ p = 2$  in $ K ( \al )$.     Sub and super $p$-harmonic solutions  of the form  
 $ r^k  \cos \he $   were  also found  in  $ K ( \pi/2) $  by  Llorente, Manfredi, Troy, and  Wu in  \cite{LMTW}. These estimates  were then used  to  find  upper and lower bounds for  $  \la_2 
(\pi/2)$ in  $ K (\pi/2). $ In  \cite{LMTW}, the authors also     
  use    shooting  methods  to   give a  strictly ODE  proof  for  existence of  a  solution   to 
\eqref{1.3} on $ [0, \pi/2] $   satisfying $ \psi (0) = 0$ and $\lim_{\he \to \pi/2} \psi ( \he ) = - \infty$  when  $p$ and $n $ are fixed with  $ 1 < p < \infty$ and $n  \geq 2$.   

 In this paper we   consider  problems  similar to the  above  for  certain  PDEs   of  $p$-Laplace  type.    Our  results,   when  specialized   to  the  $p$-Laplace  equation   for fixed $ p  >  n - 1, $  give a unique solution   $ u   $  to    \eqref{1.2}  in $ K ( \pi ) $  with  continuous boundary value $0$  on   $  \ar  K ( \pi )   $   and     $  \la =  \la_1 ( \pi  )  = 1 -  (n-1)/p $  when $ n \geq 3 $ 
(compare with  Krol's  $ n  = 2$ and $\al = \pi$ result).   To be more specific we need some notation.  
 Put 
  \[ 
  B (z, r ) = \{ y \in \mathbb R^{n} : | z  -  y | < r \}\quad  \mbox{whenever}\, \, z\in 
\mathbb R^{n} \, \, \mbox{and}\, \, r>0.  
\]  
Let  $ \lan \cdot,  \cdot  \ran $  denote  the standard inner
product on $ \mathbb R^{n} $ and  let  $  | y | = \lan y, y \ran^{1/2} $ be
the  Euclidean norm of $ y. $   Let  $dy$ denote  $n$-dimensional Lebesgue measure on    $ \mathbb R^{n} $ and   let  $\mathcal{H}^{\ga}$,  $0 <  \ga  \leq n, $  denote  $\ga$-dimensional  \textit{Hausdorff measure} on $ \rn{n}$ defined by 
\[
\mathcal{H}^{\ga }(E)=\lim_{\delta\to 0} \inf\left\{\sum_{j} r_j^{\ga}; \, \, E\subset\bigcup\limits_{j} B(x_j, r_j), \, \, r_j\leq \delta\right\}
\] 
where the  infimum is taken over all possible  $  \de$-covering  $\{B(x_j, r_j)\} $ of  $E$.  
If $ O  \subset \mathbb R^{n} $ is open and $ 1  \leq  q  \leq  \infty, $ then by   $
W^{1 ,q} ( O ) $ we denote the space of equivalence classes of functions
$ h $ with distributional gradient $ \nabla h= ( h_{y_1},
 \dots, h_{y_n} ), $ both of which are $q$-th power integrable on $ O. $  Let  
 \[
 \| h \|_{1,q} = \| h \|_q +  \| \, | \nabla h | \, \|_{q}
 \]
be the  norm in $ W^{1,q} ( O ) $ where $ \| \cdot \|_q $ is
the usual  Lebesgue $ q $ norm  of functions in the Lebesgue space $ L^q(O).$  Next let $ C^\infty_0 (O )$ be
 the set of infinitely differentiable functions with compact support in $
O $ and let  $ W^{1,q}_0 ( O ) $ be the closure of $ C^\infty_0 ( O ) $
in the norm of $ W^{1,q} ( O  ). $  
 Given $  p,  1 < p < \infty, $    suppose  $ f : \rn{n}  \rar  [0, \infty) $  satisfies:
\begin{align}
\label{1.5}
\begin{split}
(a)&\, \,   f  (  t   \eta )  =  t^p   f ( \eta ) \quad \mbox{when}\, \,   t  > 0 \, \, \mbox{and}\, \,  \eta \in  \rn{n}.
\\
(b)& \, \,   \mbox{There exists $  \ti a_1   \geq 1 $  such that if  $ \eta, \xi   \in  \rn{n} \sem \{0\},$  then } \\
&  \hs{.44in}      \ti a_1^{-1}  \,   | \xi  |^2     | \eta |^{p-2}  \,  \leq   \,  
\sum_{i, j = 1}^n  \frac{ \ar^2  f }{ \ar \eta_i \ar \eta_j}  ( \eta )   \,  \xi_i   \,  \xi_j     \leq             \ti a_1  \,  | \xi  |^2     | \eta |^{p-2}.   \\
(c)&\,\,   \mbox{There exists $  \ti a_2   \geq 1 $  such that  for  $ \mathcal{H}^{n}$-almost every   $ \eta   \in  B ( 0, 2 )  \sem  B( 0, 1/2),$  }  \\
& \hs{1.47in}   
     {\ds    
\sum_{i, j,k = 1}^n }  \left|  \frac{ \ar^3  f }{ \ar \eta_i \ar \eta_j \ar \eta_k}  ( \eta )   \right|    \leq             \ti a_2.  
\end{split}
\end{align}
Note  that  our assumptions  in  \eqref{1.5}  imply that   second  derivatives of  $  f $  are  Lipschitz and  homogeneous  of  degree  $ p -   2 $  on $  \rn{n}\sem\{0\}. $
 To  conform with the notation  in   \cite{AGHLV} and  \cite{ALSV}  we put  $  \mathcal{A} = \nabla f  $   for fixed $p$, $1<p<\infty$,   and  given an  open set 
 $ O  $ we say that $v  $ is $  \mathcal{A}$-harmonic in $ O $ provided $ v \in W^ {1,p} ( G ) $ for each open $ G $ with  $ \bar G \subset O $ and   
	\begin{align}
	\label{1.6}
		\int \lan    \mathcal{A}
(\nabla v(y)), \nabla \he ( y ) \ran \, dy = 0 \quad \mbox{whenever} \, \,\he \in W^{1, p}_0 ( G ).
			\end{align} 


As a short notation for \eqref{1.6} we write $\nabla \cdot \mathcal{A}(\nabla v)=0$ in $O$.  
Note that if   $ f (\eta  )  = p^{-1}  | \eta |^p$  then $ v  $  as in \eqref{1.6} is $p$-harmonic in  $ O.$ The  definition of    $ \mathcal{A}$-capacity,  a  $ \mathcal{A}$-capacitary  function,      and  of  the  
$ \mathcal{A}$-harmonic  Green's  function with pole at  $\infty$  are given in  section \ref{section2}.      
  



  In   this article,   we  first  prove 

\begin{mytheorem} 
\label{thmA} 
Fix $ f $ as  in   \eqref{1.5},  $ n \geq 2, \al  \in (0, \pi], $  and suppose   $  1 < p < \infty $  when $  \al  \in (0, \pi), $   while $ p  > n - 1 $  when  $ \al  = 
 \pi. $  For $ i = 1, 2, $   there exists  a unique  $\mathcal{A}$-harmonic function $ u_i  > 0  $ in  $ K  ( \al  )$ with  $ u_i ( 1, 0, \dots, 0)  = 1$ satisfying
\begin{align*}
 &(+) \, \,\, \, \, \, \,  \, \quad  \mbox{$ u_1 $  has continuous boundary value 0 on  $  \ar  K  ( \al )$.} \\
&(++) \, \, \quad  \lim_{ |x| \to \infty}  u_2 (x) = 0 \quad \mbox{and $u_2$  has  continuous boundary value 0  on  $ \ar K ( \al )  \sem \{0\}$}.
\end{align*}
Moreover,   \eqref{1.1}  holds with  $ \la  =  \la_i (\al)$, for $ i = 1, 2$,  where  $ \la_2 (\al)  <  0  <   \la_1 (\al )$ with the property that $|\lambda_i(\alpha)|$ is decreasing on $(0,\pi)$.  Finally,  $ \la_1 ( \pi ) = 1 - (n - 1)/p$ for $ p  >  n - 1 $    
and    
\begin{align} 
 \label{1.7}      
\la_1 (\al) -   1  +  \frac{n-1}{p}      \approx   (\pi -  \al)^{\frac{ p  + 1 - n  }{p - 1}} \quad  \mbox{as}\, \,    \al \to \pi.    
\end{align}  
 \end{mytheorem}    
\begin{remark}  
\label{rmk1}   
 We remark  that if  $ 1 < p \leq n - 1$ then a  slit has $p$-capacity zero in 
    $  \rn{n}$ for $n \geq 3$  and so  one can show  (see  \cite[chapter 2]{HKM}) that   there are no  solutions to   \eqref{1.3}. 
    In fact,  Krol' and Maz'ya  in the paper mentioned earlier obtained   that
\[
    \la_1 (  \al)  \approx 
\begin{cases}    
     ( \pi - \al )^{ \frac{ n -  1 - p}{p-1} } & \mbox {for}\, \, \,  1 < p < n -1 \\  
      - \frac{ 1}{  \log( \pi - \al) } & \mbox{for}\, \, \,  p =  n - 1
      \end{cases}
      \qquad \mbox{as}\, \,  \al \to \pi.
\] 
Here  and in \eqref{1.7},     $ \approx $ means the  ratio of the two functions is bounded above and below by positive constants depending only  on  $ p, n,  $  and 
possibly $ \ti a_1$, $\ti a_2$ in \eqref{1.5}.    We  regard   \eqref{1.7}   as  our main  contribution in  Theorem \ref{thm1}.  For an  outline of  our  efforts  in  trying to  prove this equality we refer  the reader to   \cite{ALV}.  As mentioned  above,  our  proof  of  existence and  uniqueness  in  Theorem  \ref{thmA}  for  $p$-harmonic functions   when $ 0 < \al  <  \pi$   is  considerably less general than the proof  in  \cite{PV} given  for  ``Lipschitz cones''.   Our  proof, however,  differs  somewhat  from the proof  of  these  authors 
(even for $p$-harmonic  functions).      We  include  a    proof  in  our setting  mainly    to facilitate the proof of  \eqref{1.7}  but also  for  completeness.  
\end{remark} 
In order  to give  an  application of  Theorem \ref{thmA}  we  need some background material.  Let $ E \subset \rn{n} $  be  a  convex  set  with nonempty interior.    
Then for  $ \mathcal{H}^{n-1} $  almost every $ x \in \ar E,  $   there is  a  well defined  
outer unit normal, $ \mathbf{g} ( x, E) $  to $ \ar E.  $   The function $ \mathbf{g}(\cdot, E): \ar E  \mapsto  \mathbb{S}^{n-1}$  
(whenever  defined) is called the  Gauss map for $ \ar E.$     Let  $ \mu $  be  a  finite   positive Borel  measure on  $ \mathbb{S}^{n-1}  = \{x \in \rn{n} :  |x| = 1\} $  satisfying
\begin{align}  
\label{1.8} 
\begin{split}
(i)&\, \,   {\ds \int_{ \mathbb{S}^{n-1}} } | \lan \he, \ze  \ran | \, d \mu ( \ze )  >  0  \quad \mbox{for all} \, \,  \he \in \mathbb{S}^{n-1},\\
(ii)& {\ds  \int_{ \mathbb{S}^{n-1}} } \ze   \, d \mu ( \ze )  = 0. 
\end{split}
\end{align}  
  Then in  \cite{ALSV}, it was shown that
\begin{theorem}
\label{thm1}  
Let   $  \mu $ be   as in    \eqref{1.8}, $f$  as in  \eqref{1.5},  and  $p$   fixed,  $ n \leq   p < \infty$.   
Then there exists a  compact              convex set $ E $ with non-empty interior   
and an $  \mathcal{A}$-harmonic Green's function  $U$ for $\mathbb{R}^{n}\setminus E$ with  pole at infinity satisfying  
\begin{align*}
 &     (a) \hs{.2in}  {\ds  \lim_{y\to x} } \nabla U (y)  =  \nabla U (x)  \, \, \mbox{exists for $\mathcal{H}^{n-1}$-almost every  $ x  \in   \ar E$}  \\ 
 &\hs{.43in}  \mbox{ as $ y \in \rn{n} \sem E $   approaches $ x $ non-tangentially.} \\  
 &(b)  \hs{.2in}  {\ds \int_{\ar E}  f ( \nabla U(x) ) \, d\mathcal{H}^{n-1}  <  \infty}.      \\ 
   &(c) \hs{.2in}     {\ds \int_{\mathbf{g}^{-1} ( K, E )  }  f ( \nabla  U(x)  )  \, d \mathcal{H} ^{n-1}} =  \mu  (K)  \quad  \mbox{whenever }      K \subset \mathbb{S}^{n-1}\, \,\mbox{is a Borel set}. \\
     &    (d) \hs{.2in}     \mbox{$E$  is  the unique  set up to translation for which  $ (c) $  holds.} 
\end{align*}
\end{theorem}  
Also in   \cite{AGHLV}  the authors  proved 
\begin{theorem}  
\label{thm2}  
Let   $  \mu $ be   as in    \eqref{1.8} and $f$ be  as in  \eqref{1.5}.
Then  for  fixed $p$ with $1<p   \not =  n - 1<n$, there exists a  compact  convex set $ E $ with non-empty interior   
and an $  \mathcal{A}$-capacitary function,   $\ti U$ for $E$,   satisfying   $ (a)-(d)$ of   Theorem  \ref{thm1} with  $  U  = \ti U$.   If    $  p  =  n - 1,  $ then  there exists a  compact              convex set $ E $ with non-empty interior having  $ \mathcal{A}$-capacity $1$, 
and a corresponding $  \mathcal{A}$-capacitary function $ \ti U $ for $E$  satisfying   $(a)$ and $(b)$  of  Theorem \ref{thm1} with  $ U  = \ti U, $ as well as,    
\begin{align*}
&(c') \hs{.2in}   \mbox{ There exists $  \ti b, 0 < \ti b < \infty, $  with } \\   
&\hs{.49in}  {\ds   \ti b  \int_{\mathbf{g}^{-1} ( K, E )  }  f ( \nabla \ti  U  )  \, d \mathcal{H}^{n-1}} =  \, \mu  (K)  \mbox{ whenever  $ K \subset \mathbb{S}^{n-1}$ is a Borel set}. \\
& (d') \hs{.2in}     \mbox{$ E $  is  the unique  set up to translation  satisfying  $(c')$  with  $ \mathcal{A}$-capacity   1.  }  
\end{align*}
\end{theorem} 
As an application of  Theorem \ref{thmA}   when $ n = 2, 3, $ we prove the regularity of the Minkowski problem.     
  \begin{mytheorem}  
  \label{thmB}    
  Let   $  \mu $ be   as in    \eqref{1.8} and  $f$  as in  \eqref{1.5}. Suppose  also  that  $ \hat \al \in (0, 1),  k $ is a non-negative  integer,  and    $  d\mu  =   \He   \,  d \mathcal{H}^{n-1} $   on  
$ \mathbb{S}^{n-1} $   for some $ 0 < \He  \in  C^{k, \hat \al } ( \mathbb{S}^{n-1} ). $  If  $ k \geq 1, $   assume  $ f  \in  C^{k + 2,  \hat \al} (  \rn{n} \sem \{0\})  . $    Let $ E $ be the compact convex set with non-empty interior in  Theorem \ref{thm1} or Theorem \ref{thm2} corresponding to  $\mu.$    
  If  either  $ n = 2, 3, $  and  $ 1 <  p  <  \infty, $ or  $ n  \geq  4 $  and $ 1 < p \leq 2, $  then  $ \ar E  $  is  locally the graph  of  a   $  C^{k + 2, \hat  \al} ( \rn{n-1} ) $  function. 
  \end{mytheorem} 
\begin{remark} 
\label{rmk2} 
Theorems \ref{thm1}, \ref{thm2}, and  \ref{thmB}  are generalizations of  existence, uniqueness, and  regularity for the classical  Minkowski Problem.  To give a little history,  the classical Minkowski existence and  uniqueness  theorem  states that if   $   \mu $  is as in   (1.8) , then there 
exists a  unique  compact convex set $  E $ (up to  translation)  with non-empty interior such that   
\begin{align}  
\label{eqn7.2}    
\mathcal{H}^{n - 1} ( \mathbf{g}^{-1} (K,E ) ) =  \mu ( K ) \mbox{ whenever  $ K \subset \mathbb{S}^{n-1} $ is a Borel set.} 
\end{align} 
 When $E$ is a polyhedron, the measure is a sum of point masses at the normals to each of the faces, and the coefficient at a normal is the surface area of that face.

The  analogue of  Theorem  \ref{thmB} concerning regularity  in the Minkowski  problem was studied by    Pogorelov in \cite{Pog}, Nirenberg in \cite{Nir}, Cheng and Yau in \cite{CY}, and     
Caffarelli in  \cite{C1, C2, C3, C}.  See also  recent work of Savin in \cite{Savin} and De Philippis and Figalli in \cite{DeF}.  In  all  papers regularity of $ \ar E $  reduces to a  corresponding regularity problem for  the graph of  a  convex solution to  a  certain  
Monge-Amp\`ere  equation with $0$  boundary values.  A more thorough  discussion of this reduction is given in  section \ref{section5}.

 Theorems  \ref{thm2} and  \ref{thmB} were first proved  by Jerison     in  \cite{J} for  Laplace's equation (i.e., when $ f ( \eta ) = |\eta |^2/2 $)   and  after that  generalized  to  $p$-harmonic functions when  $ 1 <  p  < 2$  in \cite{CNSXYZ} for $ n  >  2 $.  It  will turn out that  it  suffices  to  assume   that   $ \He  $ is bounded above  and below on  $  \mathbb{S}^{n-1}  $  in order   to   conclude  $  \ar  E $  is strictly convex  and  locally the graph of  a  
 $ C^{1, \ep}  $ function where  $ \ep > 0 $  depends  on $ \ti a_1, \ti a_2, p, n, $ the  eccentricity of $ E, $  and the bounds  for  $ \He $. 
 \end{remark}  
\subsection{Outline of the proof of  Theorems  \ref{thmA} and \ref{thmB}} 
Existence  in Theorem \ref{thmA}  for   $  \al \in ( 0, \pi ) $   follows  easily  from interior  regularity results   and  Wiener type estimates    for  $ \mathcal{A}$-harmonic functions 
   listed   in   section   \ref{section2}. 
  Uniqueness in  Theorem  \ref{thmA}  for   $ \al   \in ( 0,  \pi)$  follows  from    boundary Harnack  inequalities,  originally  proved   for   positive $p$-harmonic functions vanishing on a portion of a Lipschitz  domain in    \cite{LN,LN1}.  These  inequalities  were updated  to  $ \mathcal{A}$-harmonic functions for fixed $ p$ with  $1 < p < n$ in  \cite{AGHLV} and for $ p \geq n $  in \cite{ALSV}.   Uniqueness in  the case  $  \al  = \pi $  
is   somewhat  more involved (since  $ K ( \pi ) \cap  B ( 0, \rho)  $ is not a Lipschitz domain),  using not only the  above  boundary Harnack inequalities but also arguments from   
\cite{LLN} and  \cite{LN2}.   To  outline the  
proof  of \eqref{1.7}  we now write $ u (  \cdot, \al)$ and $\la ( \al )$ for $ u_1$ and $\la_1$   in  Theorem \ref{thmA}  relative to  $ K ( \al ). $  First  it  follows easily from our existence and uniqueness results that $ \la $  is continuous  and decreasing  as a function of $ \al $  on  $ (0, \pi )$  with    $ {\ds  \lim_{\al \to \pi}} \la ( \al ) =  \la (\pi). $    From     boundary  Harnack inequalities for $ \mathcal{A}$-harmonic functions,  as well as  an integral  identity proved in   \cite{AGHLV} for  $ n - 1 < p  < n $  and in \cite{ALSV} for $ p \geq n,  $ we  eventually obtain     
 \begin{align}    
  \label{1.10}  
  \bar c (\de)^{-1}    \leq    \int_{ \ar K ( \al) \cap  \{ x : \, x_1  \geq  
-  1  + 4 \de    \} }  \sin (\pi - \al ) \,  
     f ( \nabla u  ( y, \al ))     d \mathcal{ H }^{n - 1}   
   \leq     \bar c  (  \de ).  
\end{align}
  in \eqref{4.11} where 
\[  0 <    \pi -  \al < <  \de <<  1   \mbox{  and   $ \de$  is fixed.}   \]   Also   $ c ( \de ) \geq 1 $  is a positive constant depending only on 
$ p, n, $ and $ \ti a_1, \ti a_2  $ in \eqref{1.5}.    To estimate the integral  in  \eqref{1.10} we use   a boundary Harnack inequality for $ \mathcal{A}$-harmonic functions on  lower dimensional  sets  from \cite{LN2}   to  essentially obtain       
\begin{align}
\label{1.11}      
|  \nabla u ( \cdot,  \al )  |  \leq   c '  \, (\pi -  \al)^{\frac{ 2  - n   }{ p - 1 }} \quad \mbox{on}\, \,  \ar K ( \al ) \cap  [ B ( 0, 2 )  \sem B ( 0, 1/2)]
\end{align} 
 where $ c'$  depends on  $ p, n, $ and $ \ti a_1, \ti a_2  $ in \eqref{1.5}.  
 From  \eqref{1.10}, \eqref{1.11}, and homogeneity of  $ u (  \cdot, \al ) $   we finally get 
\begin{align}
   \label{1.12}   
   \begin{split}  
   c(\de)^{-1}       & \leq \left( { \ds \int_0^1  \,   r^{ (\la (\al) - 1) p + n - 2 }  dr  }\right) \, (\pi -  \al)^{\frac{ p  - n + 1  }{ p - 1 }}     \\
  &   \leq   \,  {\ds \frac{c(\de)}{ (\la (\al)  - 1) p + n - 1} \,(\pi -  \al)^{\frac{ p  - n + 1  }{ p - 1 }} }
\end{split}  
        \end{align} 
        where  $ c ( \de ) $   has the same  dependence  as  $ \bar c  (\de ) $   above and  we  have also  used the fact that an element of  surface area on   $ \ar K(\al) $  is  of the form $[\sin(\pi -  \al)]^{n-2}   r^{n-2}  dr$.   
 From \eqref{1.12}  and some arithmetic we conclude        
      \begin{align}  
      \label{1.13}  
      \lambda (\alpha)  \leq   1  -  \frac{n-1}{p}  +   c^* \,  ( \pi - \alpha)^{\frac{ p  - n + 1  }{ p - 1 }} \quad \mbox{as}\, \,  \alpha \to \pi
      \end{align} 
      for some $ c^* = c^* (p,n, \ti a_1, \ti a_2 )  \geq  1$    and so  get  the  desired  upper  estimate  for  $  \la_1 ( \alpha ) $  in  Theorem \ref{thmA}.  The lower estimate is similar.  We note that a slightly different  proof  of   Theorem  \ref{thmA}  for  $p$-harmonic functions when $ n - 1 < p < n $   (with more details)  is outlined  in   \cite{ALV}.         

As for the proof of  Theorem \ref{thmB},  armed   with  Theorems \ref{thmA}, \ref{thm1}, and \ref{thm2}.,  we  can    follow closely the proof in  \cite{CNSXYZ},  who in turn  followed closely the proof in  \cite{J}.    Indeed,  Jerison    in \cite{J},   first converts Theorem \ref{thmB}  into a  regularity statement for the solution, say  $ \hat u $ to a  Monge Amp{\`e}re equation whose right-hand side    corresponds to   a  measure $ \hat \mu $ on $ \mathbb{S}^{n-1}$.   To show regularity of $ \hat u,  $  he  first generalized  the  Alexandrov-Bakelman  inequality  (see \cite[Lemma 7.3]{J})  and  then  used  this  generalization to  prove  a  certain integral inequality for $ \hat \mu $  in  Theorem 6.5 of  \cite{J}.   This inequality was then used to show that arguments in  \cite{C,C1,C2,C3} could  be used to eventually obtain  Theorem  \ref{thmB} (see also \cite{GH}). Theorem  \ref{thmA} is  used in Theorem \ref{thmB}  to prove the  analogue of  Theorem 6.5 in \cite{J}  when   $ n  =  2,  3$ and $p  > 2$.  In fact,  Theorem  \ref{thmA}  is  used  only in the proof of  Lemma \ref{lem5.8}.    Unfortunately  this lemma  is   not   strong enough to be used in the rest of  Jerison's  proof when  $ p > 2, $  unless   $ n =2, 3. $

    As for the plan of this paper,  in  section \ref{section2},     we  state   some  basic  properties of  $  \mathcal{A}$-harmonic functions, give the definitions mentioned after Theorem  \ref{thm2},  and prove existence in  Theorem \ref{thmA}.   In  section \ref{section3},   we  state several   boundary  Harnack inequalities and then apply  these inequalities  to prove uniqueness  in  Theorem   \ref{thmA}.    In   section  \ref{section4} we state    integral   identities from  \cite{AGHLV,ALSV} and then use these  identities to prove Theorem  \ref{thmA}.  Theorem  \ref{thmB}  is proved in  
section  \ref{section5}.        In  section \ref{section6} we  make  closing remarks concerning generalizations of  Theorems \ref{thmA}  and \ref{thmB}.   
\setcounter{equation}{0} 
 \setcounter{theorem}{0}


\section{Basic estimates and definitions for $\mathcal{A}$-harmonic functions}
\label{section2}  
 In this section we  first introduce some notation and then state  some fundamental estimates for 
   $\mathcal{\ti A} =  \nabla \ti f$-harmonic functions when $ p $ is fixed, $ 1 < p < \infty, $ and  $ \ti f $  satisfies  \eqref{1.5}  with  $ f = \ti f$. Second, we define the  $ \mathcal{A}$-capacitary function when $ 1 < p < n $ and  $ \mathcal{A}$-harmonic  Green's function with pole at $ \infty $ when $   p  \geq n $  of  a  compact convex set $ E. $   Third, we show  existence of $ u_i$ for $i = 1, 2$, in Theorem \ref{thmA} relative to  $ K ( \al ) $  when  $ \al \in (0,  \pi). $ 
Concerning constants, unless otherwise stated, in this section, and throughout the paper,
$ c $ will denote a  positive constant  $ \geq 1$, not
necessarily the same at each occurrence, depending at most on
 $ p, n,  \ti a_1, \ti a_2,  $ 
  which sometimes we refer to as depending on the data.
In general,
$ c ( t_1, \dots, t_m ) $  denotes  a positive constant
$ \geq 1, $  which may depend at most on $p, n, \ti a_1, \ti a_2 $   and   $ t_1, \dots, t_m, $
not necessarily the same at each occurrence. Also, as in  the  introduction,  if $B\approx C$ then $B/C$ is bounded from above and below by
constants which, unless otherwise stated, depend at most on  the data.
 Let  $ e_k $ be the $n$  tuple  with   one in the $k$th position and zeros elsewhere.  Let  $ d ( E_1, E_2 )$  denote the distance between the sets $ E_1$ and $E_2. $ For short we write  $ d ( x,  E_2  ) $ for  $ d  ( \{x\}, E_2). $  
Also put  $ E_1 +  E_2  =  \{ x + y : x \in E_1, y \in E_2\}$ and $\la E =  \{ \la x : x \in E \}$ for   $ \la > 0 $. Let $\mbox{diam}(E)$, $\bar E$, and $\ar E $  denote  the diameter, closure, and boundary  of  $ E$ respectively. We write
$ { \ds \max_{E}   \ti u, \,  \min_{E} \ti  u } $ to denote  the
  essential supremum and infimum  of $  \ti u $ on $ E$
whenever $ E \subset  \mathbb R^{n} $ and  $ \ti u$ is defined on $ E$.
 \begin{lemma}
\label{lem2.1}
Given $ p, 1 < p < \infty,  n \geq 2, $  and  $ \ti f $ as in  \eqref{1.5},  let
 $ \ti u $  be  a
positive $\ti {\mathcal{A}} =  \nabla   \ti f $-harmonic function in $B (w,4r)$ for $r>0$.Then
	\begin{align}
	\label{2.1}
	\begin{split}
		(i)&\, \, r^{ p - n} \,\int_{B ( w, r/2)} \, | \nabla \ti  u |^{ p } \, dy \, \leq \, c \, (\max_{B ( w,
r)} \ti  u)^p, \\
		(ii)&\, \, \max_{B ( w, r ) } \,\ti   u \, \leq c \min_{ B ( w, r )} \ti u.
		\end{split}
\end{align}	
	Furthermore, there exists $\ti \sigma=\ti \sigma(p,n, \ti a_1, \ti a_2)\in(0,1)$ such that if $x, y \in B ( w, r )$, then
	\begin{align*}
		(iii)&\ \ | \ti  u ( x ) - \ti  u ( y ) | \leq c \left( \frac{ | x - y |}{r} \right)^{\ti \sigma} \, \max_{B (
w, 2 r )}  \, \ti  u.
	\end{align*}
   \end{lemma}
\begin{proof}
A proof of this lemma can be found in  \cite{S}.
\end{proof}
\begin{lemma}  
\label{lem2.2}	
Let $p,n,  \ti f , \ti {\mathcal{A}},  \ti u, w,r,$ be as in Lemma \ref{lem2.1}.
Then	 $ \ti u$ has a  representative locally in  $ W^{1, p} (B(w, 4r)),$    with H\"{o}lder
continuous partial derivatives in $B(w,4r) $  (also denoted $\ti u$), and there exist $ \ti \be \in (0,1]$ and $c \geq 1 $,
depending only on $ p, n, \ti a_1, \ti a_2, $ such that if $ x, y \in B (  w,  r ), $    then  
	\begin{align}
	\label{2.2}
	\begin{split}
&	(\hat a) \hs{.1in} \, \,  c^{ - 1} \, | \nabla \ti u ( x ) - \nabla \ti u ( y ) | \, \leq \,
 ( | x - y |/ r)^{\ti \be }\, \max_{B (  w , r )} \, | \nabla \ti u | \leq \, c \,  r^{ - 1} \, ( | x - y |/ r )^{\ti \be }\, \ti u (w). \\
 & (\hat b)  \, \,  \, \int_{B(w, r) }  \, \sum_{i,j = 1}^n  \,  |\nabla \ti u |^{p-2} \,  |\ti  u_{x_i x_j}|^2  dy 
\leq   c r^{(n-p-2)} \ti{u}(w).
  \end{split}
\end{align}
   \end{lemma}    
   \begin{proof}
   A proof of Lemma \ref{lem2.2} can be found  in \cite{T}.  
\end{proof}
\begin{definition}  \label{def2.3} 
 Fix $ p,  1 < p <  \infty $   and  let  $ \ti f $  be as  in  \eqref{1.5} with  $ f  = \ti f. $  If 
  $ \ti K  $   is a  compact subset of  the connected open set  $ D, $    define   the    $ \ti{\mathcal{A}}  =  \nabla \ti f$-capacity of  $ \ti K $  relative to  $  D  $  by  
\[   
\mbox{Cap}_{ \mathcal{\ti A}}( \ti K, D) =\inf\left\{\, \, \int_{D} f (\nabla w(x) ) dx:\, \, w \in C^{\infty}_{0}(D) \, \, \mbox{and}\, \, w (x)\geq  1\, \, \mbox{for}\, \,  x\in \ti K  \right\}.
\] 
\end{definition}   
In  case    $ \ti f ( \eta )  =  p^{-1}  | \eta |^p$ for $\eta  \in \rn{n}, $  we  write   
$ \mbox{Cap}_{p}(\ti K, D) $   instead of 
$ \mbox{Cap}_{\mathcal{\ti A}}( \ti K, D). $ If   $D=\mathbb{R}^{n}$    we  also write   $ \mbox{Cap}_{\mathcal{\ti A}}(\ti K) $
and   $ \mbox{Cap}_{p}(\ti K)  $  for short.   We note  from  \eqref{1.5}   that    
\begin{align}
    \label{2.4}   
     \mbox{Cap}_{p}( \ti K, D)   \approx  \mbox{Cap}_{\mathcal{\ti A}}( 
\ti K, D) \quad \mbox{and} \quad  \mbox{Cap}_{\mathcal{\ti A}}(\tau \ti K + \{x_0\} )   = \tau^{n-p}  \mbox{Cap}_{\mathcal{\ti A }}( \ti K)
\end{align}
for  $\tau > 0$ and $x_0 \in \rn{n}. $ Ratio constants depend only on the data.   If $ n \leq p < \infty$   then  
   $    \mbox{Cap}_{\mathcal{\ti A}}( \ti K)  \equiv  0$ (see \cite[Chapter 2]{HKM}).

\begin{definition}   
\label{def2.4} 
Let $ p, \ti f,  \ti A,   $   be as in  Definition \ref{def2.3}.   A  closed set $ \ti K\subset\mathbb{R}^{n}$ is called uniformly $(r_0, p)$-fat if there exists $\hat c\geq 1$ such that
\[
 \frac{\mbox{Cap}_{p}(\ti K\cap\bar{B}(w,r), B(w,2r))}{\mbox{Cap}_{p}(\bar{B}(w,r), B(w,2r))}\geq  \hat c^{- 1} 
\]
for all $0 < r \leq r_0$  and $w\in \ti K.$  The largest such $\hat c^{- 1} $ is called the uniform $(r_0, p)$-fatness constant of $\ti K$. 
\end{definition} 
\begin{lemma} 
\label{lem2.5}
Let   $ p, \ti f,   \mathcal{\ti  A} ,  $  be as in Definition  \ref{def2.4}     and  suppose  
 that     $ \ti  K $   is a    uniformly  $  (r_0, p )$-fat  compact set with $ \ti K \cap  B (z, 3\rho) \not  =  \es, $   where $ r_0 = \mbox{diam}(\ti K)$.  
Let   $\zeta \in C_0^\infty ( B (z, 4 \rho )   ) $ with  $ \ze  \equiv 1 $  on  
$  B (z, 3 \rho ). $  If $  0  \leq  \ti u $ is   $ \ti {\mathcal{A}}$-harmonic in
  $ B (z, 4 \rho ) \sem  \ti K,$    and  $  \ti u \ze \in  W_0^{1,p} (B (z,4 \rho) \sem  \ti K ),$    
  then  $  \ti u $  has a continuous  extension to  $ B (z, 3 \rho) $  obtained 
  by putting  $ \ti  u  \equiv 0 $ on  $    \ti K \cap  B (z, 3 \rho) $.   Moreover, if   $  0  < r  <   \min\{r_0, \rho\} $ 
  and   $w \in    \ti K \cap  B (z,  2 \rho) , $ then    
	 \begin{align}
	 \label{2.5} 		
	 (i) \quad  r^{ p - n} \, \int\limits_{B ( w, r/2)} \, | \nabla \ti  u |^{ p } \, dy \, \leq \, c_1 \, \left( \max_{B ( w,  r)}  \ti u \right)^p.   
	\end{align}  
	where $ c_1  $  depends   only  on   $ p,n,  \ti a_1, \ti a_2, $ and the uniform 
$ (r_0, p)$-fatness  constant for $  \ti K $. 	Furthermore,  there exist  
$\hat \sigma \in (0,1)$ and $c_2 \geq 1$,   having the  same dependence as $ c_1 $,   such that 
	\[ 
		(ii)\quad  | \ti  u ( x ) - \ti  u ( y ) | \leq c_2 \left(  \frac{ | x - y |}{r} \right)^{ \hat \sigma} \,  \max_{ B (w,  r )}\, \ti u
		\] 
		whenever $ x, y \in B ( w, r/2)$ and $ 0 < r < \min \{ r_0, \rho \} $. 
		\end{lemma}  
\begin{proof} 
Here  $(i) $ in \eqref{2.5} is a standard  Caccioppoli  inequality and
$(ii) $  for  $ y \in   \ti K  $    follows from uniform $ (r_0, p)$-fatness of  $  \ti K  $ 
 and  essentially Theorem 6.18  in  \cite{HKM}.  Combining this fact with  
  \eqref{2.1}  $ (iii)$ we  obtain  $ (ii). $   	
  \end{proof}

\begin{lemma} 
\label{lem2.6}
Let  $    {\mathcal{\ti A}}, p,  \ti f,  \ti K, r_0, z,  \rho, \ti u   $ be  as in  Lemma \ref{lem2.5}. 
	 Then there exists a unique finite positive Borel measure $ \ti \nu $ with support contained in $ \ti K \cap B ( z, 3 \rho) $ such that
\begin{align}  
 \label{2.6}
		 \int  \langle \ti {\mathcal{A}} (\nabla \ti u(y)),\nabla\ph(y)\rangle \, dy \, = \, -  \int \, \ph \, d\ti{\nu} \quad \mbox{whenever}\, \, \ph  \in C_0^\infty ( B(z,2\rho) ).
\end{align}
Moreover,  there exists $ \bar c \geq 1, $ with the same dependence as $ c_1 $  in  Lemma \ref{lem2.5},  for which    
\begin{align}
  \label{2.7}    
   \bar c^{-1}\, r^{ p - n} \ti \nu (B( w,  r/2 ))\leq   \max_{B(w, r)} \ti u^{ p - 1}\leq \bar c r^{p-n}\ti \nu(B(w, 2 r))
\end{align}  
	whenever $ 0 < r  <  \min\{ r_0, \rho \}$ and $ w  \in  \ti K \cap  B ( z, \rho)$. 
  Furthermore, suppose  for some constant  $ \La   \geq  1 $   that  if   $ w   \in   \ti K  \cap B ( z, \rho ), $    
       and  $ 0 < s  < r, $   there  exists     $ a_s ( w )  \in  B ( w, r ) \sem \ti K $  with     
\[
 \La  \,  d ( a_s ( w ) ,   \ar [ B ( z, 2 \rho )  \sem  \ti K ] )  \geq   s. 
 \]
 Suppose also     that  whenever   $ w_1, w_2 \in 
B ( z, 2 r )  \sem \ti K  $ and   $ 0 < r   \leq \rho/ \La, $  
    there exists a  rectifiable curve
  $ \tau : [0, 1] \rar  B ( z, 2 \rho ) \sem \ti K $ with $ \tau ( 0 ) = w_1$ and $\tau ( 1 ) = w_2, $
and such that 
\begin{align} 
\label{2.7a}  
\begin{split}
&(a)\hs{.2in}  \mathcal{H}^1 ( \tau ) \, \leq \,  \La  \, | w_1 - w_2 |,  \\ 
&(b)\hs{.2in}  \min \{  \mathcal{H}^1 ( \tau ( [ 0, t] ) ) , \, \mathcal{H}^1 ( \tau (
[ t, 1 ] ) )  \, \} \,   \leq  \,  \La  \, d ( \tau ( t ),   \ar  [ B ( z, 2 \rho ) \sem \ti K] ), \, t \in (0, 1).
\end{split}
\end{align}
If   $ w \in  B ( z,  r/2) \cap \ti K$  then 
\begin{align}
 \label{2.8}  
 [r^{p-n}  \ti \nu ( B ( w, 2 r ) ) ]^{1/(p-1)}  \approx   \ti u ( a_{r} (w) )   \approx    \max_{B(w,  r)} \ti u     \approx   [r^{p-n}  \ti \nu ( B ( w, r/2) ) ]^{1/(p-1)}.
 \end{align}
 Ratio constants depend only  on   the data,  the uniform  $(r_0, p)$-fatness constant for $ \ti  K, $   and   $ \La. $  
\end{lemma}

\begin{proof}
 For the proof  of  \eqref{2.6},  see  \cite[Theorem 21.2]{HKM}   
 The left-hand inequality in  \eqref{2.7}   follows from   \eqref{2.6},    \eqref{1.5},    and  H\"{o}lder's inequality,   
 using a  test function, $ \ph, $ with  $ \ph \equiv 1 $ on $ \bar B ( w, r/2 ).  $   
  The  proof of the right-hand inequality in  \eqref{2.7}  follows from \cite{KZ} (see also \cite{EL}).   
  Here  \eqref{2.7a} is  equivalent  to  a   Harnack chain condition used  in the definition of an  
  non-tangentially accessible domain  (see  \cite{JK}). The proof    of  the middle   inequality
    in  \eqref{2.8} follows from an argument  often attributed to   Carleson  (see \cite{AS})     
    and  just uses  \eqref{2.5} $(ii), $     \eqref{2.1} $(ii),$   and    \eqref{2.7}. 
     The first and  last  inequalities in \eqref{2.8}  give   the     ``doubling  property''  of  $ \nu  $  measure.    
\end{proof} 
\begin{remark}
 \label{rmk2.7}  
Uniform $(r_0, p)$-fatness of $ \rn{n} \sem  D $  for some $ r_0 > 0 $  is a sufficient condition for solvability  
 of the Dirichlet problem for $\mathcal{\ti A}$-harmonic PDEs in a  bounded domain $ D $   
 in  the sense that if   $ \ph $ is  a  continuous function on  $ \ar  D, $ then there exists  an $ \mathcal{\ti A}$-harmonic function $ \Ph $ in $D$  with continuous boundary values equal to $ \ph $ on $ \ar D. $  In fact,  if  $  \mathbb{R}^{n}\setminus D$ is uniformly $(r_0, p)$-fat
  then   for every $w\in \mathbb{R}^{n}\sem D$ and $0<r<r_0$
  \[
\int_{0}^{r_0}\left[\frac{\mbox{Cap}_p((\mathbb{R}^{n}\setminus  {D})\cap \bar B(w,r), B(w,2r))}{\mbox{Cap}_p(\bar{B}(w,r), B(w,2r))}\right]^{\frac{1}{(p-1)}}\frac{dr}{r}=\infty. 
\]  
That is, uniform $(r_0, p)$-fatness implies Wiener regularity  (see \cite[Theorem 6.33]{HKM}). 
We also remark that if   $ E \subset B ( 0,  \rho  )  $ is  a  closed convex set with  $\mbox{diam}(E)=1$ and  $ \mathcal{H}^{k} ( E )   > 0   $ for some   positive integer $ k  >    n  - p$  then  $ E $ is  $ (1, p)$-uniformly fat   and  $ \mbox{Cap}_{\mathcal{\ti A}}(E, B ( 0, 2 \rho ) ) \approx 1 $  with ratio constants depending only on the data when $ 1 < p < n $ while for $ p \geq n $  these constants depend on the data and also  $ \rho. $   On the other hand,  if $\mathcal{H}^{k} ( E ) < \infty $  for some positive integer $ k \leq  n - p$  then  $ \mbox{Cap}_{\mathcal{\ti A}}( E) = 0 $ (see \cite[Chapter 2]{HKM}).
 \end{remark}   
\subsection{Definition of  $ \mathcal{A}$-capacitary and $\mathcal{A}$-harmonic  Green's  functions} 
\begin{definition} 
\label{def2.8} 
Let   $ 1 < p  <  n$ and  $f$ be  as in  \eqref{1.5} and  let $  E $  be a  compact convex set with      $ \mbox{Cap}_{\mathcal{A}}( E) > 0$.  Then the  $ \mathcal{A}$-capacitary function 
 of $ E$, say $  \ti U,  $ is  the unique  continuous function $ \ti U \not  \equiv 1,  0 < \ti  U \leq  1,  $ on $ \rn{n} $    satisfying 
\begin{align}
\label{2.9}
\begin{split}
&(a) \hs{.2in}  \, \ti U\, \, \mbox{is}\, \,  \mathcal{A}\mbox{-harmonic in}\, \,  \rn{n} \sem E.  \\
&(b) \hs{.2in}  \,  \ti U \equiv 1\, \, \mbox{on}\, \,  E  \, \, \mbox{and}\, \,  \ti U (x) \to  0   \, \, \mbox{uniformly  as}\,\, |x|  \to \infty.       \\
& (c) \hs{.2in} \, \,  | \nabla \ti U | \in   L^p ( \rn{n} )\, \,  \mbox{and}\, \, \ti U \in L^{p^*}  ( \rn{n}  )\quad \mbox{for}\, \,  p^* =  \frac{np}{n-p}. \\
 & (d) \hs{.2in}   \,  \,   \mbox{Cap}_\mathcal{A}  (E)  = \int_{ \rn{n}} \lan   \mathcal{A}  ( \nabla \ti U ) , \nabla \ti U  \ran \,  dy.
\end{split}
\end{align}
\end{definition}  
 For existence and uniqueness of  $\ti U $ see Lemma  4.1  in  \cite{AGHLV}.   We  note that if  $  \ti \nu $  denotes the measure associated with  $  \ti  U $  as in  Lemma  \ref{lem2.6} then    $     \ti  \nu  ( E )  =   \mbox{Cap}_{\mathcal{A}} ( E )$ (see   \cite[Lemma 4.2]{AGHLV}). Therefore, if  $  E  \subset  B ( 0, 1 ) $  with $ \mbox{diam}(E) \geq 1/2 $  and $ n - 1 < p  < n$  then  
from  \eqref{2.8} and  Remark \ref{rmk2.7}  we have 
\begin{equation}   
\label{2.9a} 
c^{-1}   \leq   \mbox{Cap}_{\mathcal{A}} ( E )     \leq   c  \max_{B ( 0, 2)} ( 1  - \ti  U )      
\end{equation}    
where  $ c $  depends only on the data.  
 
In order to define  an $  \mathcal{A}$-harmonic Green's function with pole at $ \infty $ when $p \geq n, $   we first have  to  define a fundamental   solution, say  $  F, $   with pole at $0$  in $  \rn{n} $  when $ p \geq n. $  Definitions for $p=n$ and $n<p<\infty$ are different and we start with $p=n$. 
\begin{definition}
\label{def2.9}
If  $ p  = n $  we say that $  F$  is a fundamental solution to  
$\nabla\cdot \mathcal{A}(\nabla  F)=0$ in $\mathbb R^{n}$ 
with pole at $0$ if  
\begin{align}
\label{2.10} 
\begin{split}
& (i)\hs{.2in}  \mbox{$  F $ is  $  \mathcal{A}$-harmonic in  $ \mathbb{R}^n  \sem \{0\}$}.  \\
& (ii)\hs{.2in}  F\in W_{\mbox{\tiny loc}}^{1,l}(\mathbb R^{n}) \, \, \mbox{for}\, \,  1 < l < n, \, \, F ( e_1 )  = 1, \mbox{and}\\
 & \hs{.6in} |F ( x )|  =  O ( \log |x| ) \, \, \mbox{in a neighborhood  of  $ \infty$}. \\ 
&(iii) \hs{.2in} \int \lan \mathcal{A} (\nabla  F(z)), \nabla \he ( z ) \ran\, dz = -\theta(0) \quad \mbox{whenever }\,\, \theta\in C_0^\infty(\mathbb R^{n}).
\end{split}
\end{align}
If  $ p  > n $  we say that $ F$  is a fundamental solution to  
$\nabla\cdot \mathcal{A}(\nabla F)=0$ in $\mathbb R^{n}$ 
with pole at $0$ if \begin{align}
\label{2.11} 
\begin{split}
 (i)\, \,\,  & \mbox{ $F $ is  $  \mathcal{A}$-harmonic in  $ \mathbb{R}^n  \sem \{0\}$.}  \\    
(ii)\, \, & F\in W_{{\rm loc}}^{1,p}(\mathbb R^{n}),\, \, \mbox{$F$ is continuous in $\mathbb R^{n}$},\, \,  F (0)=0,\, \,  F >0 \, \, \mbox{in}\, \,  \mathbb R^{n}\setminus\{0\}.  \\
(iii)\, & \int \lan \mathcal{A} (\nabla  F(z)), \nabla \he ( z ) \ran\, dz = -\theta(0) \quad \mbox{whenever} \, \, \theta\in C_0^\infty(\mathbb R^{n}).
\end{split}
\end{align}
\end{definition}  Existence   and  uniqueness of  $  F $  in \eqref{2.10} and  \eqref{2.11}  are  proved in    Lemma  4.4     
and   Lemma 4.6     of  \cite{ALSV},  respectively.   
   \begin{definition} 
  \label{def2.10} 
Let $p\geq n$ and for a given  compact, convex set $ E \subset \rn{n}$  with $0\in E$
  we say that $U$ is the $ \mathcal{A}$-harmonic Green's function for  $ \rn{n} \sem E $ 
  with pole at $ \infty, $  if   
  $U : \rn{n} \sem E \to ( 0, \infty ) $ has continuous 
  boundary value $0$ on $ \ar E$,    $U $ is  $ \mathcal{A}$-harmonic in  $ \rn{n} \sem E$,  and     $ U(x) =  F(x)  +  k(x)  $   where $ k(x) $ is  
  a bounded function  in a neighbourhood of $ \infty$ and $F$ is the fundamental solution as in Definition \ref{def2.9}.  
  \end{definition}  
  \begin{remark}  
  \label{rmk2.11} 
  In   \cite{ALSV} the authors show that   $ U  $ exists  and is  unique  if and only if  the convex compact set $ E $ is  either  $ (a) $ non-empty when $ p  > n$ or  $(b)$  contains at least two points when $ p   = n. $   If $ U $  exists  then it was  also shown that  $     k  \leq 0 $ in $ \rn{n}   \sem E $  and   $ k $  is  H\"{o}lder  continuous in  a  neighbourhood of  $ \infty $  with  $  {\ds  \lim_{ x \to \infty}}  k  (x) = k ( \infty).$    They    then   define   
\[ 
\mathcal{C}_{\mathcal{A}}( E ) :=
\left\{
\begin{array}{ll}
  e^{- k ( \infty )/\ga} &  \mbox{when}\, \, p = n, \\ 
 (- k ( \infty ))^{p-1} & \mbox{when}\, \,  p > n. 
\end{array}
\right.
\]  
If  $ E $  is  a  single point  and $ p = n $  (so  $U$ does not exist), set  
$ \mathcal{C}_{\mathcal{A}}( E ) := 0. $         
Here $\gamma$ is a  constant  depending only on the data which occurs in the asymptotic expansion of  $ F ( x ) $ as  $ x \to  \infty. $ 
From the definition of  $\mathcal{C}_{\mathcal{A}}( E )$  and
     translation, dilation invariance of  $\mathcal{A}$-harmonic functions  it follows  as  in   
   \eqref{2.4} that  if      $ x_0  \in \rn{n}$, $r > 0$, and  $ E $ is  a  convex compact set  then   
      \begin{align} 
 \label{2.12}
 \begin{split}
\mathcal{C}_{\mathcal{A}} (r  E + \{x_0\} )=
  \left\{
  \begin{array}{ll}
      r \mathcal{C}_{\mathcal{A}} (E)  & \mbox{when}\, \,  p =n,\\
      r^{p-n}  \mathcal{C}_{\mathcal{A}} (E) & \mbox{when} \, \, p > n.
\end{array}    
           \right.
\end{split}
\end{align}   Also if  $  \nu  $   is  the  measure  associated with  $ U $  as  in  Lemma   \ref{lem2.6}   then  (see Lemmas 5.2, 5.3 in  \cite{ALSV}),   
    $   \nu ( E ) =  1$. Hence if  $  E  \subset B ( 0, 1 ) $  with  $  1/2  \leq  \mbox{diam}(E)$   it follows  from  \eqref{2.8}  that
  
\begin{equation}   
\label{2.12a}   
 \max_{B ( 0, 2 )}  U  \approx 1   \mbox{ where the proportional constants depend only on the data}. 
 \end{equation} 
Finally, if  $ E_1 \subset E_2 $  are compact convex  sets and  $ U_1$ and $U_2$  the  corresponding  $\mathcal{A}$-harmonic Green's functions with pole at  $ \infty$   then  
\begin{equation} 
\label{2.12b}   
U_1 \geq U_2   \, \, \mbox{in}\, \,  \rn{n} \quad \mbox{so}\quad   \mathcal{C}_{\mathcal{A}} (E_1)  \leq  \mathcal{C}_{\mathcal{A}} (E_2). 
 \end{equation}   
  \end{remark}     

\subsection{Existence in Theorem \ref{thmA}} 
To show existence and uniqueness for   $ u_1$ and $u_2 $  in Theorem  \ref{thmA}   we shall  also need the following lemma.
\begin{lemma} 
 \label{lem2.12}  
 Fix $ p$ with $1 < p < \infty$ and $\al \in (0, \pi)$ and suppose  $0< r  \leq   R / 10$.   Let  $ v $ be  
 the  $ \mathcal{A}$-harmonic function in  
 $ D = [ K ( \al ) \sem   \bar B ( r e_1, \frac{r\al}{100})] \cap B ( 0, R ) $  with continuous boundary 
 values  $  v  \equiv  1$  on  $ \ar  B ( r e_1, \frac{r\al}{100})$ and   $v \equiv 0  $  on 
 $ [ \ar B (0, R ) \cap K ( \al ) ] \cup [\ar K ( \al ) \cap B (0, R)]. $  Then there exists  $ c \geq 1 $  such that 
\begin{align}    
\label{2.13}     
- c  \, \lan  \nabla   v  ( x ), {\ts \frac{x - r e_1}{|x - r e_1|}} \ran  \geq    v (  x   )  \quad \mbox{whenever} \, \, x\in D. 
\end{align}   
Here $ c $  depends   on the data and $ \al $ if  $ 1 < p \leq n - 1,   $  while  $ c $ depends only on the data if  $ p  > n - 1$.  
 \end{lemma}  

 \begin{proof}   
 Let  $  \hat D   =  \{ y :   y   +  r e_1 \in D \} $ and define $ \hat v $  on  
$ \hat D $   by  $ \hat v ( y )  = v ( y +  r e_1)$ for $y \in  \hat D$ (see Figure \ref{DhDhDl}).     Given  $  \la$ with $1 <  \la  <  1001/1000$, set    
 $  \hat D  ( \la )  = \{  y  \in  \hat D  :  \la y   \in  \hat  D \}.$

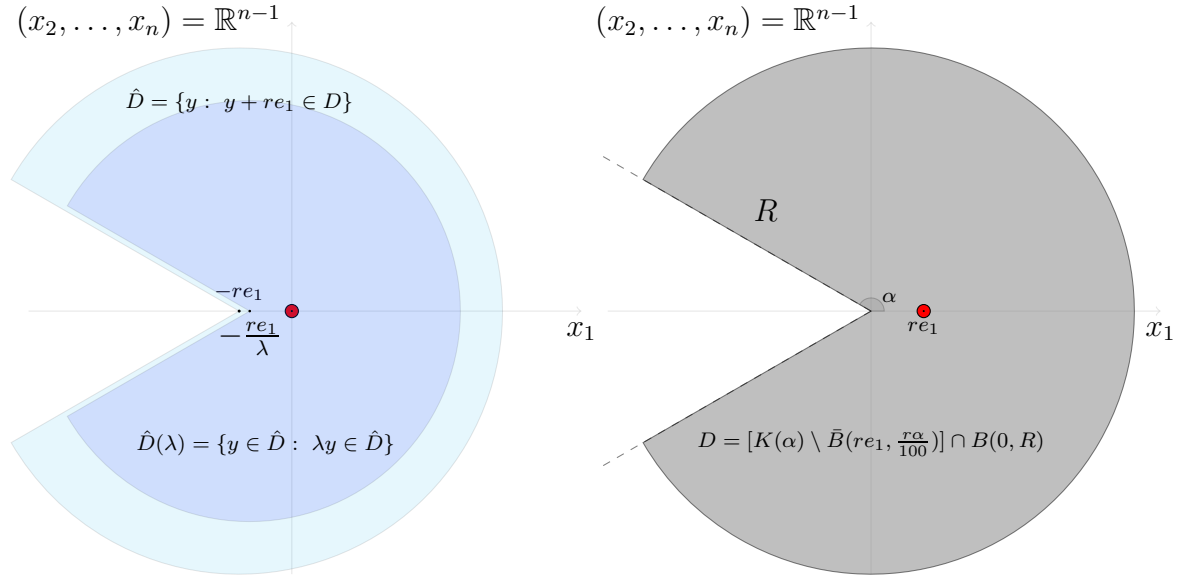
\begin{figure}[!ht]
  \begin{center}
\begin{tikzpicture}[scale=.35]
\begin{scope}[shift={(-10,0)}]
\coordinate (A) at (2,0);
\coordinate (A-) at (-2,0);
\coordinate (Alambda) at (-1.6,0);

\draw[->, gray, opacity=.2] (-10,0)--(11,0);
\node[below] at (11,0) {$x_1$};
\draw[->, gray, opacity=.2] (0,-10)--(0,11);
\node[left] at (0,11) {$(x_2,\ldots, x_n)=\mathbb{R}^{n-1}$};



\draw[fill=red] (0,0) circle (.25cm);
\draw[fill] (0,0) circle (.5pt);

\draw[fill=cyan, opacity=.1] (-2,0) -- ++(150:10) arc (150:-150:10)--cycle;
\draw[fill] (A-) circle (1pt);
\node[above] at (A-) {\tiny $-re_1$};

\begin{scope}[scale=.8]
\draw[fill=blue, opacity=.1] (-2,0) -- ++(150:10) arc (150:-150:10)--cycle;
\node[below] at (Alambda) {$-\frac{re_1}{\lambda}$};
\draw[fill] (Alambda) circle (1pt);
\end{scope}

\node at (-1,-5) {\tiny $\hat{D}(\lambda)=\{y\in \hat{D}: \, \, \lambda y\in \hat{D}\}$};
\node at (-2,8) {\tiny $\hat{D} = \{y: \, \, y+re_1\in D\}$};
\end{scope}

\begin{scope}[shift={(12,0)}]

\draw[->, gray, opacity=.2] (-10,0)--(11,0);
\node[below] at (11,0) {$x_1$};
\draw[->, gray, opacity=.2] (0,-10)--(0,11);
\node[left] at (0,11) {$(x_2,\ldots, x_n)=\mathbb{R}^{n-1}$};

\draw[opacity=.5, fill=gray] (0,0) -- ++(150:10) arc (150:-150:10)--cycle;
\draw[dashed, opacity=.5]  (0,0) --++(150:12) ;
\draw[dashed, opacity=.5] (0,0) --++(-150:12);

\draw[fill=gray, opacity=.2] (0,0) --++(150:.5) arc (150:0:.5) --cycle;
\node at (0.7, 0.5) {\tiny $\alpha$};

\draw[fill=red] (2,0) circle (.25cm);
\draw[fill] (2,0) circle (.5pt);
\node[below] at (2,0) {\tiny $re_1$};

\node at (0,-5) {\tiny $D=[K(\alpha)\setminus \bar{B}(re_1, \frac{r\alpha}{100})]\cap B(0,R)$};
\node[above] at (-4,3) {$R$};
\end{scope}

\end{tikzpicture}
  \end{center}
  \caption{The sets $D$, $\hat{D}$, $ \hat  D (\la)$. }
  \label{DhDhDl}
\end{figure} 
 
 From the  definition of  $ D,  \hat D,  v,  \hat v, $ and translation and dilation invariance of  $  \mathcal{A}$-harmonic functions  
 we see that  $ y   \mapsto \hat v (y)$ and $y \mapsto  \hat v ( \la y )$ are both  $  \mathcal{A}$-harmonic  in   $ \hat  D (\la)$.   If  
 \[
 h ( y )  :=  \frac{ \hat  v ( y )  - \hat  v ( \la y )  }{ \la - 1}\quad \mbox{for}\, \, y  \in  \hat  D ( \la )
 \]
 we claim  that    
\begin{align}
 \label{2.14}   
 \breve{c} \,  h ( y )  \geq  \hat v ( y ) \quad \mbox{for} \, \, y\in  \hat D ( \la )
 \end{align} 
 where   $ \breve{c} \geq 1 $ has the same dependence as  $ c  $ in  Lemma \ref{lem2.12}.   
 Using the boundary maximum principle for $ \mathcal{A}$-harmonic  functions and  continuity of  $  h$ and $\hat v  $  we   see that it
   suffices to prove   \eqref{2.14}  when   $ y  \in  \ar  \hat D ( \la ). $  To do this we note 
from the definition of  $ \hat  D  ( \la ) $  that  if $ y  \in  \ar  \hat D ( \la ), $  then either  
$   y =  z/\la  $  for  some  $ z   \in  \ar \hat D $  with $ \hat v (z) = 0 $   or   
$  y  \in  \ar \hat D $     and  $ \hat v (y) = 1.$   In the  first case we see that   
$ \hat v  ( \la y  )  = 0 $    so    \eqref{2.14} is trivially true.   In the   second case   
let  $ \ti  f ( \eta )  = f (  -  \eta ), $  and    note that  $ 1 - \hat v $ is  $ \nabla \ti f =\mathcal{\ti  A}$-harmonic in   $ \hat  D. $ 
Using this note,   uniform fatness of   $ K ( \al) \cap  B ( 0,  R ),  $  the definition of  $ \hat D, $     \eqref{2.5} $(ii)$ for $\hat v,$ and   
Harnack's  inequality  we deduce  that  if  $  r' =  \frac{  2 \al r }{100  }, $   then      $  1  -  \hat v      \geq   c_*^{-1} $
on  $  \ar   B  ( 0,    r'  ) $    for some  $ c_*  \geq  1 $  with the same  dependence as  $c$  in  the statement of Lemma   \ref{lem2.12}.   Thus    
$ c_*  ( 1 -  \hat v )  \geq   1 $  on   $  \ar  B ( 0, r')  $ and  $(1-\hat v)  \equiv 0 $  on  $  \ar  B ( 0,  r'/2). $  Also this function  is  $ \nabla \ti f = \mathcal{\ti A}$-harmonic  in   $ T = B ( 0, r' )  \sem \bar B ( 0,  r'/2). $    

 Using these  facts   and   a  barrier type argument as in \cite[section 7]{AGHLV}  or   \cite[(4.6)-(4.9)]{ALSV},       it follows (since  $ |y|  =  r'/2$)  that 
\begin{align}
  \label{2.15} 
  \hat v ( y )  -  \hat v ( \la y )  =  1 -   \hat v ( \la y) \geq     ( \la - 1) /( \bar c \, c_*) 
\end{align}
where $ \bar c  \geq 1 $ depends only on the data.  
From  \eqref{2.15} we   conclude   that  \eqref{2.14}  also holds in the  second case when 
$ y \in   \ar  \hat D (\la). $        Thus      \eqref{2.14}  holds on  $  \ar  \hat D ( \la ) $ so 
by the  above maximum principle  is   valid  in  $ \hat D (\la)  .$  Letting  $  \la  \to 1 $ in    
\eqref{2.14}   and   using   
\eqref{2.2}  $ (\hat a),$  as well  as the chain rule,      we get       
\[  
 - c \lan  \nabla \hat v(y), y/|y| \ran  \geq  \hat v (y)   
 \] 
 for  $ y \in  \hat D.$   Clearly this inequality implies \eqref{2.13}.   
\end{proof}  

To begin the  proof  of   existence  in  Theorem  \ref{thmA} for  $ u_1,$  let $ v$ and $D $  be as in 
Lemma  \ref{lem2.12} and put  $ R=l, r = l/10$, for sufficiently large positive integer $l$ (say $l\geq 100$).   Set   $ v_l  = M_l \, v $  where  $ M_l > 0 $ is chosen so that  $ v_l  (e_1) = 1.$  Extend $ v_l $ to  a continuous function in $ \bar B (0, l) $ 
by  defining  $v_l  \equiv 0 $  on  $   [\bar B (0, l ) \sem K ( \al)] \cup  \ar B (0,l)   $  while  $ v_l  \equiv  M_l $ 
on  $ \bar B ( \frac{l e_1}{10},  \frac{l\al}{1000}).$    
  Using  Lemmas  2.1,  2.2,  2.5    and  letting  $ l \rar \infty $ it follows  from  Ascoli's theorem that  a  subsequence  of   $ (v_l), $ 
  also denoted  $ (v_l) $,  converges  uniformly to   $ u_1,  $   an  $ \mathcal{A}$-harmonic  function    in  $  K  (\al) $  that  is  also   
H\"{o}lder  continuous in   $  \rn{n} $  with   $  u_1   \equiv 0 $  on  $ \rn{n} \sem  K  ( \al ). $   

   To   construct  $ u_2 $, we   let   $ r  =  1/l$, $R = l $,   and  let   $ v_l  = \hat M_l \,  v$  for $l=2,3,\ldots,$  
   where  $ \hat M_l $ is chosen  so  that   $ v_l (e_1)  = 1. $    Extend $ v_l $  to  a  continuous function 
   on  $ \bar B ( 0, l )  $  by putting $ v_l  \equiv  0 $  on  $ [B ( 0, l ) \sem  K (  \al )] 
  \cup   \ar B ( 0, l )    $   and $ v_l   \equiv   \hat M_l $  on  $ \bar B ( e_1 /l,   \frac{\al}{100l}).  $ 
 Also    from   Lemmas  2.1,  2.2,  2.5  and \eqref{2.8}  we deduce    for    $ l  >  \rho   >   2/ l ,  $ that  
 there exists   $ c \geq 1$ and $\breve \be \in (0, 1)$ such that  
\begin{align}
 \label{2.16}  
  \max_{B ( 0, l ) \sem B ( 0, \rho)} v_l    \leq  c  v_l ( \rho e_1 )  \leq   c^2   \rho^{- \breve \be } . 
 \end{align}
 Here $c$ and $\breve \be$ depend on  the   data and  $ \al$ if  $ 1 \leq p  \leq  n  - 1 $   
 while  these constants  depend only  on  the data if  $ p  >  n - 1$.
 Letting  $ l \to \infty, $ it follows  from   the above lemmas, and    Ascoli's theorem that  a  
 subsequence  of   $ (v_l), $ also denoted  $ (v_l), $  converges  uniformly to   $ u_2,  $   an  $ \mathcal{A}$-harmonic  function
     in  $  K  (\al) $  that  is  locally H\"{o}lder continuous in   $  \rn{n} \sem \{0\}$  with   $  u_2   \equiv 0 $  on  $ \rn{n} \sem (K  ( \al ) \cup \{0\}). $ 
      Moreover,  \eqref{2.16} holds with  $ v_l  $   replaced by 
$  u_2 $  and from    \eqref{2.13} we have   
\begin{align}
 \label{2.17} 
 - c \, \lan  \nabla u_2(x), x/|x| \ran  \geq  u_2 (x)\quad \mbox{whenever}\, \, x  \in  K ( \al ) .  
 \end{align}

\setcounter{equation}{0} 
 \setcounter{theorem}{0}

 \section{Boundary Harnack inequalities and  uniqueness in Theorem  \ref{thmA}} 
 \label{section3} 
 To prove  that   $u_1$ and $u_2 $  are unique   and  satisfy   \eqref{1.1} in  Theorem  \ref{thmA}   we use  a  variety of   boundary  Harnack inequalities, mostly in Lipschitz domains.    To set the stage for these inequalities, let $ K \subset \rn{n - 1}, n \geq 2,  $  be a non-empty compact set  and   recall  that    $ \ph :   K  \to \mathbb R $ is said to be Lipschitz on $ K  $ provided there
exists $  \hat b,  0 < \hat b  < \infty,  $ such that
\begin{align} 
\label{3.1}   
| \ph ( z' ) - \ph ( w' ) |  \,  \leq \, \hat b   \, | z'  - w' | \quad  \mbox{whenever}\, \, z', w' \in  K.  
\end{align}
The infimum of all  $ \hat b  $ such that  \eqref{3.1} holds is called the
Lipschitz norm of $ \ph $ on $ K, $ denoted by $ \| \ph  \hat  \|_{K}$.   
It is well-known that if $ K  \subset \mathbb R^{n-1} $ is compact, then  $ \ph $  has  an  extension to    
$ \rn{n - 1} $   (also denoted by $ \ph$)  which is differentiable  almost everywhere in $ \rn{n-1}$
and     
\[  
\| \ph  \hat \|_{\mathbb R^{n-1}} = \| \, | \nabla \ph | \,  \|_\infty  \leq c   \| \ph  \hat \|_{K}. 
\]    
In fact, one can take $c=1$ (see \cite[Section 2.10.43]{Federer}). Now suppose that  $ D $ is  an open set, $ w \in \ar D,  \hat r > 0, $  and  
\begin{align}   
\label{3.2}   
\begin{split}
    \partial D\cap B(w, 4  \hat r )&=\{y=(y',y_n)\in\mathbb R^{n} : y_n=
  \phi ( y')\}\cap B(w, 4 \hat r), \\
     D\cap B(w, 4  \hat r )&=\{y=(y',y_n)\in\mathbb R^{n} : y_n > 
  \phi ( y')\}\cap B(w, 4 \hat r)
  \end{split}
\end{align}   
  in an appropriate coordinate system  for some   Lipschitz function $\phi$ on $ \mathbb{R}^{n-1}$ with  
$ \ph ( w' ) = w_n. $ 
 Note from elementary geometry  that   if   $ \ze   \in\partial D  \cap  B ( w, 2\hat r) $  and  $0<s<\hat r  $,  we can find points    \[ 
 a_s(\ze )\in D\cap B(\ze ,s)\quad \mbox{with} \quad d(a_s(\ze),\ar D)\geq c^{-1}s
 \]
for a constant $c$ depending on  $ \| \ph \hat \| $. In the following,    
we let $a_s(\ze )$ denote one such point.   Also  let  $ \De ( w, r ) =  \ar D  \cap B ( w, r), r > 0,   $ and   if  $ \ze  \in  \De  ( w, 2 \hat r )$ and $ t > 1 $ 
 let   
 \[   
\Ga ( \ze  ) = \Ga ( \ze, t ) = \{ y \in D \cap B( w, 4 \hat r)   :  | y - \ze | <  \, t  \, d ( y, \ar D ) \}. 
\]  
 Unless otherwise stated  we always assume that    $ t $  is fixed and   so large that   $ \Ga ( \ze ) $  contains the  inside of  a  truncated cone with vertex at $ \ze,  $  height  $  \hat r ,$  axis along the positive $ e_n $ axis, and of  angle opening $ \he  = \he ( t ) > 0. $  
We  note for $ D, \hat  r$, and $w$  as above that  $ \rn{n}  \sem  (D \cap  B ( w, \hat r )) $ is  uniformly $ (\hat r, p ) $-fat   for $ 1 < p  < \infty. $   Thus,  if  $ v   $   satisfies  the  same hypotheses as  
$ \ti u $ in      Lemmas   \ref{lem2.5} and \ref{lem2.6},   then  these  Lemmas    are valid with  $\ti u $  replaced by    $ v $  in  the above   $ D. $  
 It follows that   (see  \cite[Section 8]{ALSV} and \cite[section 10]{AGHLV}
 there exists $ \bar c \geq 1, $ depending only on the data and  $ \| \ph  \hat \|,$  
such that   if $ 0 < r \leq  \hat r $ and  $ \bar r  =  r/\bar c, $ then   
\begin{align}  
\label{3.3}  
\bar  r^{ p - n}   \int\limits_{B ( w, \bar  r)}   | \nabla v |^{ p }  dx  \leq \bar c    (v ( a_{\bar r} (w)))^{p}
   \end{align} 
   where $w$ is as in \eqref{3.2} and $a_{\bar r} (w)$ is a point associated to $w$ as in the display below \eqref{3.2}. Moreover, there exists $\hat \si \in(0,1),  $ depending only on the data and $\| \ph \hat \|$,  such that 
\begin{align}   
\label{3.4}  
| v ( x ) -  v ( y ) | \leq \bar c \left( \frac{ | x - y |}{\bar r}\right)^{\hat \si}     v   (  a_{\bar r}  ( w ) ) \quad \mbox{whenever}\, \, x, y \in B ( w, \bar r ).
 \end{align} 
  Finally, there  exists a unique finite positive
Borel measure  $ \nu $ on   $ \mathbb{R}^{n}$, with support contained in
$ \bar \Delta(w,r)$, such that 
\begin{align}
\label{3.5} 
\begin{split}
&(a)  \hs{.2in}    
{\ds \int  \lan  \nabla f   ( \nabla v  ),   \nabla \psi \ran dx  =  -   \int   \psi \,  d \nu} \quad \mbox{whenever} \, \, \psi  \in C_0^\infty (  B(w,r) ), \\
 &  (b)   \hs{.2in}  
\bar c^{ - 1} \,  \bar r^{ p - n}   \nu ( \Delta (  w,  \bar r ))\leq (v  ( a_{\bar r} ( w ) ))^{ p - 1}\leq \bar c  \, 
\bar r^{ p - n }   \nu (\Delta (  w, \bar  r )).  
\end{split}
\end{align}
Also in  \cite[section 10]{AGHLV}  for    $ 1 < p <  n $  and  in \cite[section 8]{ALSV}    for  $ p \geq n $  we  updated  to $ \mathcal{A}$-harmonic functions  the  following  Lemmas    
proved in  \cite{LN},  \cite{LN1},     for  $p$-harmonic functions when  $ 1 <  p  < \infty. $   
\begin{lemma}  
\label{lem3.1}  
Let   $D , \hat r,  w, \phi $  be as  in \eqref{3.2}, $p$ fixed, $ 1 < p < \infty, $   and $ 0 < r  \leq \hat r. $   
  Also  let $ v$   be a positive $  \nabla f = \mathcal{A}$-harmonic in  $  D  \cap  B ( w, r )$  and  continuous in $  B ( w, r ) $ with  $ v \equiv 0 $ on  $  B ( w,r ) \sem D. $      
There exists  $  c_{\star}   \geq 1,  $  depending only on  the data and   $\| \ph \hat \|$, 
such that if   $  4 \ti r  =   r/  c_{\star} $ and $ x  \in  B ( w,\ti  r )  \cap   D,   $   then 
\begin{align}
\label{3.6}   
\begin{split}
&(a)  \hs{.2in}  c_{\star}^{-1}  { \ds \frac{v (x)}{ d ( x, \ar D )}  \leq          \lan  \nabla v (x),   e_n \ran   \leq   |  \nabla v (x)  |  \leq    
    c_{\star} \frac{v (x)}{ d (x, \ar D )} } \,   ,\\
&(b) \lim_{\substack{x\to y \\ x \in  \Ga ( y)\cap B(w,2r)}}  \nabla v ( x ) \stackrel{def}{=} 
\nabla v ( y ) \, \, \mbox{exists}  \quad \mbox{for}\, \,  \mathcal{H}^{n-1}\mbox{-almost every}\, \, y \in\Delta( w, \ti r ).
\end{split}
\end{align}
Moreover, $ \De ( w,  \ti r ) $ has a tangent plane
  for  $ \mathcal{H}^{n-1}$-almost every $ y  \in   \De ( w, \ti r )  $.   If
 $ \mathbf{n}( y ) $ denotes the unit normal to this tangent plane pointing
 into
 $ D \cap B ( w, 2 \ti r ), $  then
\begin{align}
\label{3.7}  
\nabla v  ( y )
=  | \nabla v ( y )  | \, \mathbf{n} ( y ) \quad \mbox{for $\mathcal{H}^{n-1}$-almost every} \,\,  y\in\De ( w, 2 \ti  r ) 
\end{align}
and 
\begin{align}   
\label{3.8}  
 \frac{ d \nu }{ d  \mathcal{H}^{n-1}} (y) =   p  \frac{ f ( \nabla v(y))}{|\nabla v (y)|} \quad \mbox{for}\, \, \mathcal{H}^{n-1}\mbox{-almost every}\, \, y\in \Delta(w,2\ti r). 
 \end{align}
\noindent Finally,   there exists    $  q >  p/(p-1)$ and $c_{\star \star}  $  with the same dependence  as  $ c_{\star}  $   such that  
\begin{align}
  \label{3.9} 
  \begin{split}
  & 
 \hs{.2in}  {\ds   \int_{ \De ( w,     \ti  r )  } \,  
 \left(\frac{ f ( \nabla v )}{|\nabla v |}  \right) ^q \, d\mathcal{H}^{ n - 1 } 
\,   \leq \, c_{\star \star}  \,    r^{ (n -  1)(1 - q) }    \left( \int_{\De ( w,   \ti  r )  
 } \, \frac{ f ( \nabla v )}{ |\nabla v|} \, d\mathcal{H}^{ n - 1} \, \right)^{q}.  
 }
\end{split}
\end{align}
  
 \end{lemma}       
To prove   uniqueness  for $u_1$  in  Theorem \ref{thmA}  we need  the following  boundary Harnack inequality.           
    \begin{lemma} 
\label{lem3.2} 
  Let  $D, \hat r, w, \ph, p,  $  be as in  Lemma \ref{lem3.1}  and $ 0 < r  \leq \hat r. $   
  Also  let $ v_i$, for $i = 1, 2$  be positive $   \nabla f = \mathcal{A}$-harmonic functions in  $   D  \cap  B ( w, r )$  and  continuous 
  in  $  B ( w, r ) $  with   $ v_1 \equiv v_2 \equiv 0 $ on  $ B ( w, r )  \sem D. $     Then
  there exist  $  \be_+ \in (0, 1)$ and $c_{+}  \geq  1$, depending only on  the data and 
 $ \| \ph  \hat \|$,  such that if $r^+  =   r/c^+$ then   
\begin{align}   
\label{3.10} 
\left|   \frac{ v_1 ( x )}{ v_2 (x)}  -  \frac{ v_1 (y )}{ v_2 (y)} 
  \right|  \, \leq \,   c_+    \left( \frac{ | x - y | }{ r^+ }   \right)^{\be_+ } \, 
 \frac{ v_1 (x) }{v_2 ( x )} 
  \end{align}  
whenever $x, y   \in  D  \cap B (w, r^+)$.
 \end{lemma}

\subsection{Uniqueness  in Theorem  \ref{thmA}  for  $ 0 < \al  < \pi $}  
To prove uniqueness for  $ u_1 $  when   $ p,  1 <  p < \infty,  $  and  $ \al  \in ( 0, \pi ) $  are fixed,  suppose  $  \hat u  > 0 $ in $ K ( \al ) $ and is  also $ \mathcal{A}$-harmonic  as well as continuous in  
$ \rn{n}$ with  $ \hat u  \equiv 0 $  on $ \rn{n} \sem K ( \al ) $ and $ \hat u ( e_1 ) = 1. $  Using  Lemma 
\ref{lem3.2}  with $ D =  B ( 0, R)  \cap  K  ( \al ),     v_1 = u_1,  v_2  = \hat u,  $  and $ w = 0, \hat  r  =  R/ 2,   $  we find that    
\begin{align}
 \label{3.11} 
\left|   \frac{ u_1 ( x )}{ \hat u (x)}  -  \frac{ u_1 (y )}{ \hat u (y)}   \right|  \, \leq \,   c_+    \left( \frac{ | x - y | }{ R }   \right)^{\be_+ } \,  \frac{u_1 (x)}{\hat u (x)} 
  \end{align}
in   $  B ( 0, \frac{R}{2 c_+}) $  for some $ c_+$ and $\be_+$ depending only on the data and  the Lipschitz constant of $\partial K(\alpha)$.  Fixing $ x, y, $  and letting  $  R \to \infty $  it follows that  $ u_1 =  \hat u. $  To show 
that   $u_1$ has the form \eqref{1.2}    observe    that  for  fixed $  t > 0,  $  the  function $ x  \mapsto  u_1 ( t  x )$ for  $x  \in  K ( \al  )$   is positive,  $\mathcal{A}$-harmonic,  and  has  boundary value 0  on  $  \ar  K  (\al ), $  so by uniqueness of   $ u_1,  $  we have   
\begin{align}   
\label{3.12}  u_1  (t x ) =  u_1 ( t e_1 )  u_1 ( x ) \quad \mbox{whenever}\, \, x \in K ( \al ).  
\end{align}  
Differentiating  \eqref{3.12} with respect to $ t $ (permissible by  Lemma  \ref{lem2.2})   and evaluating at $ t =  1 $  we see that   
      \[ 
        \lan x,  \nabla u_1 ( x ) \ran =   \lan e_1,  \nabla u_1 ( e_1 ) \ran  u_1 ( x )\quad \mbox{whenever} \quad x  \in K (\al). 
        \]    
If we put $ \rho =  | x |,   x / | x | = \om \in    \mathbb{S}^{n-1}, $   in this identity  we obtain that  
\[  \rho \,   (u_1)_{\rho}  ( \rho \om ) =  \lan e_1,  \nabla u_1 ( e_1 )\ran u_1 ( \rho  \om ).  
\]
Dividing this equality by  $ \rho  u_1  ( \rho  \om ) , $  integrating with respect to $ \rho, $  and  exponentiating,   we find  that   
        $  u_1 (  r \om )  =  r^{\la_1 } u_1 ( \om ) $  whenever $ \om \in \mathbb{S}^{n-1}$  where  
$ \la_1  =  \lan e_1,  \nabla u_1  ( e_1 ) \ran$.  

To prove    uniqueness  for $u_2 $   in $ K ( \al )$   with   $   p$ and $\al $  fixed with  $0 < \al < \pi,  1  < p <   \infty, $ we let   $ 0 < \hat  u  $   be $  \mathcal{A}$-harmonic in 
$ K   ( \al) $   with  continuous boundary value 0 on  $  \ar K ( \al )  \sem \{0\}, \hat u (e_1) = 1,  $   and 
\begin{align}
  \label  {3.13}  \lim_{|x| \to \infty}  \hat u ( x ) = 0. 
  \end{align}
From   Lemma   \ref{lem3.2} 
 we  see  that if   $   w \in  \ar K ( \al ) \sem \{0\}$,  $r  = |w|/4$,  $v_1  \not = v_2$ with $v_2 = \hat u$ or $v_2= u_2$,   
 then   \eqref{3.10} in  Lemma \ref{lem3.2} is valid for both $\hat{u}/u_2$ and $u_2/\hat{u}$.  Now  \eqref{3.10} for $ \hat u, u_2$,  \eqref{3.13}, \eqref{2.16} for $u_2$,    Harnack's  inequality  and   the maximum principle  for  $ \mathcal{A}$-harmonic functions  yield  that  
\begin{align} 
 \label{3.14} 
 c^{-1}  \leq  \frac{ u_2 (x)}{ \hat u ( x ) } \leq   c \quad  \mbox{for}\, \, x\in K ( \al )
 \end{align} 
 where $ c \geq 1 $  depends only on the data.  Indeed, if for example   
\[  
\liminf_{\substack{x\to 0 \\ x \in K (\al )}}   \frac{ u_2 (x)}{ \hat u ( x ) } =  0 
\] 
then  the  above  program first  gives   $ u_2(x)/\hat u (x)  \to  0 $  as $ x \to 0 $ in $ K ( \al ) $  and second that  $ u_2 \equiv 0, $   clearly  a contradiction.    

 Now \eqref{3.14}, \eqref{3.6} $(a)$  for $ \hat u$ and $u_2$  when  $ w  \in  \ar K (\al) \sem \{0\}$ and $r = |w|/4$,  and \eqref{2.17} imply that  there exist  $ c_*  \geq  1$ and $\hat  \be \in ( 0, 1)$,  depending only on  the data and $\al$,    such that      
\begin{align}
  \label{3.15} 
   \left|   \frac{u' (x)}{u'' (x)}  -  \frac{u' (y)}{u'' (y)}  \right| \, \leq \,    c_*  \,  \frac{u' (x)}{u'' (x)}  \left( \frac{\rho}{ \min \{ | x |, | y | \}}  \right)^{\hat \be} \quad \mbox{for}\,\,  x,  y  \in  \rn{n}  \sem  B ( 0,   c_* \,  \rho )
\end{align}  
whenever $ 0 < \rho < 1/c_*  $   and  $ u'  \not = u''  \in  \{ \hat u,  u_2 \}. $   Fixing 
$ x, y, $   and  letting $ \rho  \to 0 $   we conclude that   $ \hat u = u_2$.
The proof  of   \eqref{3.15}  is  quite similar to the proof of  \eqref{3.10} (given the above assumptions) only arguments are made  in   $ \rn{n} \sem  B (0,\rho) $  rather  than  $  B (0, r). $    For the  proof  of  a  somewhat stronger inequality than \eqref{3.15}  when  $ \hat u$ and $u_2 $  are  $p$-harmonic functions,  see  the proof of Theorem 3 and Corollary 5.25  in 
  \cite{LN1}.  The proof  of  \eqref{3.15}  when $ \hat u$ and $u_2 $   are  $ \mathcal{A}$-harmonic is  essentially unchanged,  so we omit the details.  
Homogeneity of  $ u_2, $ i.e.,  \eqref{1.1},  assuming   uniqueness,  is proved in the same way as for $u_1$ when   $ \al \in  ( 0,  \pi ). $

\subsection{Existence and uniqueness in  Theorem \ref{thmA} for $ \al = \pi$}

It  remains  to  show existence and  uniqueness in  Theorem  \ref{thmA} when $ \al  =  \pi $ and $ p  >  n - 1$.     To do this, for $i=1,2,$ we  temporarily  write  
\[    
u_i  ( t x ,  \al ) =   t^{\la_i (\al)}   u_i   ( x,  \al) \quad  \mbox{for}\, \, x \in K ( \al)\, \, \mbox{and}\, \,  \al \in (0, \pi)
\]   
for the  functions  in  Theorem \ref{thmA}  corresponding  to  $  K  ( \al ). $  
 From  the  maximum principle for  $\mathcal{A}$-harmonic functions it  follows  that  if $ 0 < \al_1  < \al_2 <  \pi, $  then  
$ u_1  ( \cdot, \al_1 )  \leq  \bar c  \,    u_1  ( \cdot, \al_2) $ in   $ K ( \al_1 )  \cap  B ( 0, 1) $
 so necessarily 
 \[
 0  < \la_1 ( \al_2)  \leq \la_1 ( \al_1). 
 \] 
 Also strict inequality must hold since otherwise from  \eqref{1.1}   it would follow that 
$  u_1  ( \cdot, \al_1)/u_1 ( \cdot, \al_2 ) $ has an absolute maximum in  $ K ( \al_1 )$   which  again  leads to a contradiction by way of  the maximum principle for  $ \mathcal{A}$-harmonic functions.      
Similarly ,  if $ 0 < \al_1  < \al_2 <  \pi, $  then  
$ u_2  ( \cdot, \al_1 )  \leq  \bar c   \,  u_2  ( \cdot, \al_2) $ in   $ K ( \al_1 )  \sem   B ( 0, 1) $ and   $  \la_2 ( \al )  <  0  $    for   $  \al \in (0, \pi ), $   thanks  to    \eqref{2.16} for 
$ u_2  ( \cdot, \al ). $   
Thus  
\[  
 0  <  - \la_2 (\al_2 ) \leq - \la_2 ( \al_1).   
  \] 
Moreover,    strict inequality holds in  this  equation since  otherwise we could get  a contradiction by the same argument as above.    
We  conclude from our considerations for $ i = 1, 2,$ that
\begin{align}  
\label{3.16} 
 |\la_i  ( \al ) | \mbox{ is decreasing on }  (0, \pi ).   
\end{align}
For $ i  = 1, 2$, let
\[ 
\la_i  ( \pi )  =  \lim_{\al \to \pi} \la_i (\al). 
\]    
 We note that  if      $   \al \in (0, \pi]$ and $n-1 <   p  < \infty$   then  Lemmas   \ref{lem2.1}, \ref{lem2.2}, \ref{lem2.5}, and   \eqref{2.8}   
 are valid for  $ u_1  $  in $ K (\al) \cap  B ( 0, \rho) $ with constants depending  only on the data as  follows from   
     uniform  $ (\rho, p ) $-fatness of  $ (\rn{n} \sem K (\al)) \cap B (0, \rho) $   when 
$  n - 1 <  p < \infty. $  Using these facts  and Ascoli's  theorem  we find  that  as $ m   \to  \infty$, 
a  subsequence of  $ \{ u_1  ( \cdot,  \pi - 1/m  ) \}, $     converges
uniformly on  compact subsets of   $  \rn{n}  $   to  $ u_1 ( \cdot, \pi ), $ 
a  H\"{o}lder  continuous function on  $ \rn{n}  $ which is   $ \mathcal{A}$-harmonic  in   $ K ( \pi )$   
with $ u_1  \equiv 0 $ on  $ \ar K ( \pi ) .$ 
   Similarly,    Lemmas   \ref{lem2.1}, \ref{lem2.2}, \ref{lem2.5},   \eqref{2.8}, \eqref{2.17}, 
   and  \eqref{2.16} (with $v_l$ replaced by $u_2$)     are valid for  
   $ u_2 ( \cdot, \al)  $  in $ K (\al) \cap  B ( w, \rho) $ whenever  $ w  \in  \ar  K ( \al ) \sem \{0\}$ and $\rho  < |w|/4. $  
   All constants depend  only on the data for   $ n - 1 < p < \infty. $  Using these facts  as above,  
   we  obtain  $ u_2 ( \cdot,  \pi ),  $  a  uniform limit on compact subsets of $ \rn{n} \sem \{0\}, $ of  a  subsequence  of   $ 
( u_2 ( \cdot,  \pi -   1/m )) $ as $ m  \to \infty.$     Also  $ u_2 ( \cdot, \pi ) $  is $ \mathcal{A}$-harmonic
   in  $ K ( \al) \sem \{0\}$  and locally H\"{o}lder  continuous  on  
$ \rn{n} \sem \{0\} $  .   Moreover,   \eqref{2.16}, \eqref{2.17}  hold   with  $ v_l, u_2, $ replaced  by
     $ u_2 ( \cdot, \pi ). $    From  \eqref{1.1} for $ \al \in ( 0,  \pi ) $  and   \eqref{3.16}  we  deduce for   $ i  = 
1, 2, $  that 
\begin{align} 
\label{3.17}  
u_i  (t x,  \pi  )  =   t^{\la_i ( \pi ) }  u_i  ( x, \pi ) \quad \mbox{whenever}\,\, x \in  \rn{n} \sem \{0\}. 
\end{align}
 
To  prove  uniqueness of $ u_i ( \cdot, \pi )$ for $i   = 1,  2,$ we  need  several Lemmas  
  analogous  to   Lemmas  \ref{lem3.1} and \ref{lem3.2}  for  Lipschitz domains.         
\begin{lemma}  
\label{lem3.3} 
 Fix  $ p$ with  $n - 1 < p  <  \infty$,  $n  > 2$, $t   > 0$, and let $   \breve I  $ be  the line segment with endpoints  $  - 3t e_1/2$ and $-  t e_1/2 $. 
 Let  $ 0<  v  $  be $ \mathcal{A} =  \nabla f$-harmonic  in  $ B ( - t e_1, t/2 ) \sem   \breve I $ 
with continuous boundary value 0  on  $ \breve I .$     Then there  exists $ c \geq 4, $  depending only on the data,  such that  
\begin{align}    
\label{3.18}    
c^{-1}    \, \frac{v (x)}{d(x, \breve I)}   \leq       | \nabla v (x) | \,  \leq c \, \frac{v (x)}{d(x, \breve I)} 
\end{align}
 for  $ x \in  B ( - t e_1,  t/c ) \sem \breve I. $  
 \end{lemma} 
   \begin{proof} 
   See  Lemma 7.1  in  \cite{LN2}. 
    \end{proof} 

\begin{lemma}   
\label{lem3.4}   
Let   $ p,  n, f, t, \breve I,  $  be as in Lemma  \ref{lem3.3}. For fixed  $  \rho,  0 < \rho  < t/ 2,  $ let $ 0 < v_i, \, i = 1, 2, $ be  $ \mathcal{A} =  \nabla f$-harmonic in  $ B ( - t e_1, \rho) \sem \breve I. $ There exist
$ c_*   \geq 1$ and $\be_*  \in ( 0, 1 ), $ depending only on the data, such that
\begin{align}
\label{3.19}   \,  \left| \frac{v_1 (x)}{v_2 (x)}  -    \frac{v_1 (y)}{v_2 (y)}   \right|   \,  \leq \,   c_*  \,  \frac{v_1 (x)}{v_2 (x)} \,  \,  \left(  \frac{ | x - y|}{\rho} \right)^{\be_*} 
\end{align}
whenever       $  x, y   \in  B ( - te_1,  \rho/c_*) \sem  \breve I.  $ 
\end{lemma}     

\begin{proof} 
See  Lemma 6.2 in \cite{LN2}. 
\end{proof}

\begin{proof}[Proof of uniqueness of $u_1 ( \cdot,  \pi) $] 
We  now prove uniqueness of $u_1 ( \cdot,  \pi) $  when  $ p  >   n - 1 $  and  $ n  \geq 2. $        
Suppose  $ 0 \leq \hat u $ is also $ \mathcal{A}$-harmonic in $ K ( \pi )$ with  continuous 
boundary value 0  on  $ \ar K ( \pi  ) $ and $ \hat u ( e_1 ) = 1.  $      Then  from  \eqref{3.19}, 
Harnack's inequality, and the maximum principle for $ \mathcal{A}$-harmonic functions  we deduce for $ n \geq 3  $  as  in \eqref{3.14}   that 
\begin{align}
 \label{3.20}   
 \ti c^{-1}  \leq  \frac{u_1 (x, \pi )}{\hat u (x)}  \leq  \ti c\quad    \mbox{whenever}\, \,  x \in \rn{n} \sem \ar  K ( \pi )
 \end{align}
where $ \ti c $  depends only on the data.    To prove \eqref{3.20}  for $ n = 2 $  we note that both components  of    $ B ( - t e_1, \rho )\sem  \breve I $  are Lipschitz domains so we can use the boundary  Harnack  inequality  for Lipschitz domains (Lemma  \ref{lem3.2})  to  estimate  the  ratio of  $ u_1 ( \cdot, \pi) / \hat u $ in $  B ( - t e_1, t/c ) \sem \breve I.$    Doing this and using  Harnack's  inequality,  the maximum principle for $ \mathcal{A}$-harmonic functions, once again,  it follows that Lemma  \ref{lem3.4}  and  \eqref{3.20}  are also  valid when $ n = 2. $  
 Next  observe from homogeneity of  $ u_1,  u_1 \geq 0,  $  and Lemmas \ref{2.1}, \ref{2.2},  that     
given $0 <  \de < \pi $ there exists $ c( \de )  \geq 1 $, depending only on the data and $ \de $,  such that 
\begin{align}
\label{3.21} 
c (\de)^{-1}   \frac{u_1(x, \pi)}{ d ( x,  \ar K ( \pi ))}  \leq     |   \nabla u_1 ( x, \pi  ) |  \leq  c (\de)   \frac{u_1(x, \pi)}{ d ( x,  \ar K ( \pi ))}
\end{align}
for $ x \in K ( \pi - \de)$.     Using    Lemma  \ref{3.3} for $ u_1 ( \cdot, \pi)$  when $ n  > 2$   and \eqref{3.1} $(a)$ on  both sides of  $ \ar K (\pi) $  when $ n = 2, $  we deduce  for fixed  $ \de =  \de_0 $ near enough $\pi$  that    \eqref{3.21} is  valid  when    $ x  \in   K ( \pi )   $ for some $ c (\de_0), $ depending only on the data.    
Finally  \eqref{3.20}, \eqref{3.21}, and    Lemmas    \ref{lem3.3}, \ref{lem3.4},  can be used for  $ n \geq 3 $  as in  \cite[subsection 4.2,  Assumption 1]{LN2}  and for $ n = 2 $  as in \cite{LLN}  to show first  that 
  \eqref{3.21} with $ u_1 $ replaced by $ \hat u $  holds  when  $ x \in  K ( \pi ) $ for  some $0<  \de = \de_1 <  \de_0$. Second   that    there exists,   $ c_{**}  \geq 1, \be_{**} \in (0, 1), $  depending only on the data with 
  \begin{align}
\label{3.21a}   \,  \left| \frac{u_1 (x)}{\hat u  (x)}  -    \frac{ u_1 (y)}{\hat u  (y)}   \right|   \,  \leq \,   c_{**}  \,  \frac{u_1 (x)}{\hat u  (x)} \,  \,  \left(  \frac{ | x - y|}{\rho} \right)^{\be_{**} } 
\end{align}
whenever    $ \rho > 0 $   and  $  x, y $  in  $ K ( \pi )  \cap B ( 0, \rho).  $    Letting $ \rho  \to \infty $ in this inequality it follows that $ \hat u = u_1 ( \cdot, \pi) $ so  $u_1 $ is unique. 
 \end{proof}
 
\begin{proof}[Sketch of Proof of  \eqref{3.21a}]
To briefly outline the strategy in the proof of  \eqref{3.21a}, assuming  \eqref{3.21}  for  $ \hat u, u_1, $   in   $ \rn{n} \sem \ar K (\pi), $ when  $ n \geq 3, $   suppose   $ a, b \in ( 0, \infty ). $  Then using   Lemmas 2.1, 2.2, and  \eqref{3.21},   one can show that  $  \chi (x)   =     ( a \, | \nabla \hat u (x)  |  + b \, | \nabla   u_1 (x) | )^{p - 2} $ is an $ A_2 $  weight on $ \rn{n} $  with  $ A_2 $ constant  $\leq c $  where  $ c $  depends only on the data.  That is,  
\[  \left(   \int_{B( y, r ) } \chi dx  \right) \cdot \left(   \int_{B( y, r ) } \chi^{-1}  dx  \right) \,  \leq  \,  c \, r^{2n}  \mbox{ whenever $ y  \in  \rn{n}$ and $ r > 0$} .  \]  
 Also $ \ze  = a  \, u_1  -  b \, \hat u$  is  a  weak  solution to  the degenerate elliptic divergence form PDE,  
\begin{align}  
\label{3.22}    
L  \ze = \sum_{i,j=1}^n \frac{ \ar (  b_{ij} (x) \ze_{x_j} )}{ \ar  x_i } \, =  \,   0
\end{align}
where    
\begin{align}  
\label{3.23}  
b_{ij}( x ) =   \int_0^1  f_{\eta_i \eta_j} ( t  a \,  u_1 (x, \pi) + (1-t) b \, \hat u (x) ) \, dt
\end{align}  
whenever  $ x \in  K (\pi).$   Moreover, for some $ c   \geq 1 $ depending only on the data, 
\begin{align}
\label{3.24}   
c^{-1}   \chi ( x ) | \xi |^2   \leq  \sum_{i, j = 1}^n      b_{ij}( x ) \xi_i  \xi_j  \,   \leq \,   c  | \xi |^2  \chi ( x) \quad \mbox{for}\, \,  \xi \in \rn{n} \sem \{0\}. 
\end{align}
Using   \eqref{3.22}-\eqref{3.24}, one can then use   the   boundary  Harnack  inequalities from divergence form linear degenerate elliptic PDE  whose degeneracy is given in terms of an  $A_2 $  weight to get  \eqref{3.21a} (see section 4 in \cite{LN2}) .  \eqref{3.21} for  $ \hat u $ is proven by a  perturbation type argument as in  (4.42)-(4.45)  of  \cite{LN2}. 
   \end{proof}
   
\begin{proof}[Uniqueness of  $ u_2 ( \cdot, \pi) $]
Uniqueness of  $ u_2 ( \cdot, \pi) $ is proved similarly.  Indeed suppose  $ \hat u  $ is also  $ \mathcal{A}$-harmonic  in $ K (\pi) $ with continuous boundary value 0  on $ \ar K ( \pi ) \sem \{0\}, $  and     $  \lim_{|x| \to \infty}  \hat u ( x ) = 0. $   
Then  \eqref{3.20} and  \eqref{3.21} in  $  K ( \pi  ) $  are  valid with $ u_1 ( \cdot, \pi) $ replaced by $ u_2 (  \cdot,  \pi) $  by the same argument as the one we gave for $ u_1 ( \cdot, \pi). $ 
These inequalities  can then be used as outlined above  to show  that for some  
$ \bar c^*   \geq 1$ and $\bar \be^*  \in ( 0, 1 ), $ depending only on the data,    that  
\begin{align}  
\label{3.25}  
 \left| \frac{u_2 (x, \pi )}{\hat u (x)}  -    \frac{u_2 ( y, \pi )}{\hat u  (y)}   \right|   \,  \leq \,   c^*  \,  
\frac{u_2  (x, \pi )}{\hat u  (x)} \,  \,  \left(  \frac{ \rho}{ \min \{ | x |, | y| \} }\right)^{\bar \be^*} 
\end{align}
whenever $ |x|, |y| \geq 2 \rho.$  Letting $ \rho  \to 0 $ we then get  $ u_2 ( \cdot, \pi ) = \hat u. $ This  completes the proof  of  uniqueness for $  u_1 ( \cdot, \pi)$ and $u_2 ( \cdot, \pi). $    
\end{proof}

\setcounter{equation}{0} 
 \setcounter{theorem}{0}

\section{Proof of  \eqref{1.7} in Theorem \ref{thmA}}  
\label{section4}  
To  show   $ \la ( \pi ) = 1 - (n-1)/p  $  for fixed  $ p  >  n - 1,  n   \geq 2, $  and  $ f  $  as in  
\eqref{1.5},    we let  $  0 < \de    < 10^{-100}$  be  a small  
but fixed number.   Also  $ \ep > 0,  $   $ 0 <  \ep  < <  \de^{1000}  $ is   allowed  to vary.  Put      
\[
 E      =   \{ x:  x_1   \geq  - 1  \}  \sem K (\pi - \ep ).
 \]
Given $  \eta   \in  \rn{n} \sem \{0\}, $  let   $  \hat  f   ( \eta )  = f (- \eta )$ when $ n - 1 < p < n.$ 
If  $ n - 1 <  p < n $  let $   U =  1 -  \hat U  $   where  $ \hat U $ is 
the  $   \hat{\mathcal{A}} =   \nabla  \hat f$-capacitary   function  
for  $  E  $  as  in  Theorem \ref{thm2},  so  $ U $ is  $ \mathcal{A} = \nabla f $-harmonic.    
If  $ n  \leq p <  \infty $  let  $   U   $  be the  $    \mathcal{A}  = \nabla    f $-harmonic Green's  function for   $E$ as in Theorem  \ref{thm1}.    We  also  write  $  u$ and $\la$   for $ u_1 ( \cdot, \pi - \ep )$ and $\la_1  ( \pi - \ep )$  in Theorem \ref{thmA}  when there is  no  chance of  confusion.   We shall  need the following  lemma (see 
Definition \ref{def2.8},   Remark \ref{rmk2.11} for notation).   
\begin{lemma} 
We have 
\label{lem4.1}       
\begin{align*}
p {\ds \int_{ \ar E }} &    \,  \lan  x  + e_1 ,    \mathbf{n}  \ran    f ( \nabla 
  U  )   d \mathcal{H}^{n-1}  
= \begin{cases}   \frac{p (n-p)}{p-1}  \mbox{Cap}_{\mathcal{\hat A}} (E) & \mbox{when } n - 1 < p < n \quad   \,   \\  \ga    &  \mbox{when}\, \, p  = n   \\
    \, \frac{p-n}{p-1}  \, \mathcal{C}_{\mathcal{A}}  (E)^{1/(p-1)}  
   &  \mbox{when}\, \, p  > n  \\
   \end{cases}\quad  \approx   c  
 \end{align*}  
where   $ \mathbf{n} ( x )  $  denotes  the  outer unit normal  at   $ x \in \ar  E $ and $ c $ depends only on the data.     
   \end{lemma}  
   \begin{proof} 
    Lemma \ref{lem4.1} is  proved   in  
 \cite{ALV}    (see  Remark  11.3)   for    $  p \geq  n$   and   in   \cite{AGHLV}  (see    
Remark 13.4)     for   $  n - 1 <  p  <  n$,   using  the  Hadamard  variational formula.   The integral in  these  remarks  is   
  defined  in terms of  a  measure on   $  \mathbb{S}^{n-1}$ obtained by way of the  Gauss map,  
  so  for example as in  $(c) $   of    Theorem  \ref{thm1}  for $ p  \geq n,  $ and 
  the  support  function of   a  convex  set  relative to   zero   rather than  $ - e_1. $  
  However,  using  \eqref{1.8}  $(ii)$ and  the definition  of    a  support  function it is easily seen that  both integrals are  equal.  
 \end{proof} 
 To obtain  estimates  on   $  U $  near    $  \ar E $   we note that in  \cite[Lemma 5.3]{LN2}, 
 it was shown that for  fixed $p$ with $n - 1 < p  < \infty$ and $n  \geq 3,  $ a   
 continuous function    $ w $  on $ \rn{n} $  exists with   
$ w  \equiv 0 $  on  $T$ where
\[
T := \{ x : \, \, x_k  = 0\, \, \mbox{for}\, \,  2 \leq k   \leq n\, \, \mbox{and}\, \, - \infty < x_1 <   \infty \}.  
\]
Also $ w  $  is   $  \mathcal{A}$-harmonic   in   $\rn{n} \sem T $ and    for $ x  \in \rn{n}, $
\begin{align} 
\label{4.2}   
w (x)  \approx  | x  - x_1  e_1 |^{\he}\quad  \mbox{where}\,\,   \he  =  \frac{ p + 1  -  n}{p-1}. 
\end{align}
Ratio  constants depend only on the  data.    We use  \eqref{4.2} to  show that there exists  $  \ti c_1   \geq  1   $   depending only on the data with  
\begin{align}  
\label{4.3}   
\ti  c_1  \,  U (x)   \geq  w (x) \quad \mbox{when}\, \,  x  \in B ( 0, 2 ) \cap \{y :  | y - y_1 e_1 |  \geq   \ti c_1  \ep    \}.
\end{align}  
To prove  \eqref{4.3}  observe from   Lemma \ref{lem4.1},  \eqref{2.9a} with $ \mathcal{A} $  
replaced by  $ \mathcal{\hat A}$  for $ n - 1 < p < n, $ and    \eqref{2.12a} when  $ p\geq  n $  that  
$  w \leq  c'    U  $  on  $ \ar B ( 0, 2). $  Using  \eqref{4.2}  and  the  boundary  maximum  principle  for $  \mathcal{A}$-harmonic functions it  follows  that   for  some $ c''  \geq 1,  $ 
\begin{align}
\label{4.4}   
w \leq   c'   U  + c''  \ep ^\he  \quad \mbox{in}\, \,   K ( \pi - \ep ) \cap  B ( 0, 2 )
\end{align}  
where constants  depend only on the data.  Using  \eqref{4.2} in   \eqref{4.4}   we see for $  \ep > 0, $  sufficiently small,   that  \eqref{4.3} is valid.    Next  we show for some  
$ \ti c_2  \geq 4, $ depending only on the data    that    
 \begin{align}   
 \label{4.5}  
  U/w  \leq   \ti c_2   \quad  \mbox{in}\, \,   {\ts  B ( - \frac{1}{2} e_1,  \frac{1}{ \ti c_2} )} \sem T.
\end{align} 
    To prove  \eqref{4.5}  let  $ v_1 $  be the  $ \mathcal{A}$-harmonic function in   
$ B ( - \frac{1}{2} e_1,   \frac{1}{4} )  \sem T $  with  continuous  boundary values  $ v_1 = u $ on     $ \ar  B ( - \frac{1}{2} e_1,   \frac{1}{4} ) $  and  $ v_1  \equiv 0  $  on  
$ T  \cap   B ( - \frac{1}{2} e_1,   \frac{1}{4} ) . $  Comparing boundary values of  $ u$ and $v_1 $,  we see from the maximum principle for $ \mathcal{A}$-harmonic functions  that                 
$ u  \leq v_1 $   in  $ B ( - \frac{1}{2} e_1,  \frac{1}{4} ) \sem T.$  This inequality, \eqref{4.2},    and   Lemma \ref{lem3.4} with  $ v_2  =  w,  $    give    \eqref{4.5} 
 since   $ u (- \frac{1}{2} e_1 +   \frac{1}{4} e_n ) \approx  w (- \frac{1}{2} e_1 +   \frac{1}{4} e_n ) .$    

Let  $  S   =   E  \cap   \{ y  :  y_1   \geq   - 1  +   4 \de  \}  $   and let   
$  V  $ be the  $ \mathcal{ A}$-harmonic Green's  function for  the  complement of 
$ S_1    =   E     \cap  \{ x :  x_1   \leq     -  1 + 4 \de \}$ (see Figure \ref{Kpiepsilon})  with a pole at infinity when  $ p  \geq n $ while  $  V  =   1 - \hat  V $ where $ \hat V $ is  the  $ \mathcal{\hat A}$-capacitary function for $ S_1$ if $ n - 1 < p< n.$

\begin{figure}[!ht]
\begin{center}
\begin{tikzpicture}[scale=2, font = \sansmath]

  \coordinate (O) at (0,0);


  \begin{scope}[rotate=90]

    \def\rx{.012}
    \def\ry{0.002}
    \def\z{1.6}

    \def\rs{.014}
    \def\rt{0.0028}
    \def\rl{1.98}

    \def\ru{.016}
    \def\rv{0.003}
    \def\rz{2.2}

    \def\ra{.018}
    \def\rb{0.0035}
    \def\rc{2.4}

    \path [name path = ellipse]    (0,\rc) ellipse ({\ra} and {\rb});
    \path [name path = horizontal] (-\ra,\rc-\rb*\rb/\rc) -- (\ra,\rc-\rb*\rb/\rc);
    \path [name intersections = {of = ellipse and horizontal}];


    \def\xo{2}
    \def\yo{.5}
    \def\zo{0}


  \shade[ball color = gray, opacity = 0.1] (0,0) circle [radius = 2cm];
      \draw[fill=white] (intersection-1) -- (0,0) -- (intersection-2) -- cycle;

  \draw[red, fill=red, opacity=.5] (O) -- (-0.012,1.6) -- (0.012,1.6) to [edge label = $S$] (O);
  \draw[blue, fill=blue, opacity=.5] (-0.012,1.6) -- (-.018, 2.4) -- (.018, 2.4) to  [edge label = $S_1$]  (0.012,1.6)-- (-0.012,1.6) ;

\path[opacity=.5, draw=red, decorate,decoration={brace,amplitude=8pt},xshift=0,yshift=0pt] (O) -- (-0.012,1.6) node [black,midway, yshift=-0.4cm]  {\footnotesize $1-4\delta$};
\path[opacity=.5, draw=blue, decorate,decoration={brace,amplitude=6pt},xshift=0,yshift=0pt] (-0.012,1.6) -- (-.014, 1.98)  node [black,midway, yshift=-0.3cm]  {\footnotesize $2\delta$};
\path[opacity=.5, draw=blue, decorate,decoration={brace,amplitude=4pt},xshift=0,yshift=0pt] (-.014, 1.98)--(-.016,2.2)   node [black,midway, xshift=.1cm, yshift=-0.3cm]  {\footnotesize $\delta$};
\path[opacity=.5, draw=blue, decorate,decoration={brace,amplitude=4pt},xshift=0,yshift=0pt] (-.016,2.2)--(-.018,2.4)   node [black,midway, xshift=.1cm, yshift=-0.3cm]  {\footnotesize $\delta$};


\draw[densely dashed] (0,\zo) ellipse ({\xo} and {\yo});

\draw[fill = gray!30, densely dashed] (0,\rz) ellipse ({\ru} and {\rv});

\draw[fill = gray!30, densely dashed] (0,\rl) ellipse ({\rs} and {\rt});

\draw[fill = gray!30, densely dashed] (0,\rc) ellipse ({\ra} and {\rb});

\draw[fill = gray!30, densely dashed] (0,\z) ellipse ({\rx} and {\ry});

\draw[fill] (0,\rz) ellipse ({.01} and {.01});
\draw[fill] (0,\rl) ellipse ({.01} and {.01});
\draw[fill] (0,\rc) ellipse ({.01} and {.01});
\draw[fill] (0,\z) ellipse ({.01} and {.01});

\draw[fill] (.2,\rz) ellipse ({.01} and {.01});

\node[above] at (.4, \rz) {\footnotesize $(-1+\delta)e_1+\delta e_n$};

\draw[->] (.4, \rz)--(.25,\rz);

  \end{scope}

 
 
  \filldraw (O) circle (.5pt); 
  \begin{scope}[rotate=90]
  


\node[below] at (0,-2.4) {$e_1$};
\draw[fill] (0,-2.4) circle (.5pt);
\node at (1,0) {$B(0,1-2\delta)\cap K(\pi-\epsilon)$};
  
  \end{scope}

\end{tikzpicture}
\end{center}
\caption{The set $E=S_1\cup S$ and $B(0,1-2\delta)\cap K(\pi-\epsilon)$.}
\label{Kpiepsilon}
\end{figure}
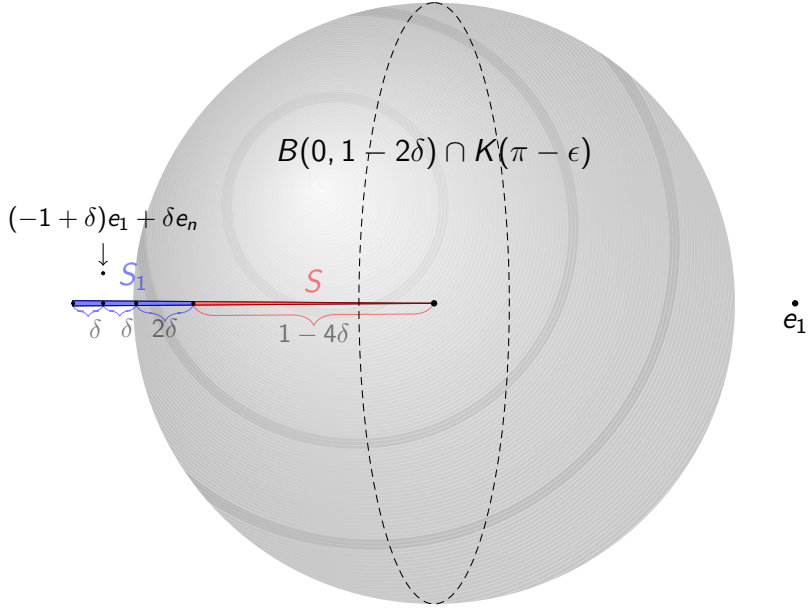

  We note  from  \eqref{2.12b} that  
$  V   \geq   U   $ in  $ \rn{n}  $  when $  p  \geq n$.    Using  this note,  \eqref{3.6} $(b)$,  \eqref{3.7},  and the  Hopf   boundary maximum  principle  we  deduce    for  $ n < p  <  \infty    $ that      
\begin{align}
\label{4.6} 
\begin{split}
   { \ds \int_{ \ar S_1  \cap \ar E}} & {\ds 
\lan y + e_1 ,\nabla  U  \ran   f ( \nabla  U  ( y ))   |  \nabla  U  ( y ) |^{-1}  d \mathcal{ H }^{n - 1}}  \\  
& \leq  
   { \ds \int_{\ar  S_1}  
\lan y + e_1 , \nabla  V   (y) \ran   f ( \nabla  V   ( y ))   
|  \nabla  V ( y ) |^{-1}  d \mathcal{ H }^{n - 1}} \\
&  \leq c  \de^{\frac{|p-n|}{p-1}}  
\end{split}
\end{align}   
   thanks   to  Lemma \ref{lem4.1} with $ E $ replaced by  $ S_1$  and \eqref{2.12},     
   where  $ c $  depends only on the data. If $ 1 < p  < n $  we see  from   \eqref{2.9} $(b)$   that  
 $ U ( x), V (x)  \rar  1$  \mbox{ as }  $ |x| \rar \infty $  so  $ U \leq V $  in   $ \rn{n}, $ 
 by  the  maximum principle for $ \mathcal{A}$-harmonic functions.  In view  of   
 this fact and  \eqref{2.4} we conclude that  \eqref{4.6} remains valid when 
$ 1 <  p  < n$  if   $   \de^{\frac{|p-n|}{p-1}} $  is  replaced by $ \de^{|p-n|}$. 
   If  $ p  =  n $  it  follows   from  \eqref{2.5} $ (ii),$   \eqref{2.8},  for   $  U$  with  $ w  = - e_1, $  and 
\eqref{2.12a},  \eqref{2.8},   dilation invariance and Harnack's  inequality for  
$  \mathcal{A}$-harmonic functions,  as applied to    $ V, $   that   for  some $ c \geq 1, \hat \si \in (0,1),    $  depending only on the data,    
  \begin{align*}
\max_ {\ar B ( -   e_1,  8 \de ) }    U    \leq   c   \de^{\hat \si} \leq c^2   
  \de^{\hat \si}\min_ { \ar  B ( - e_1,  8  \de ) }     V.
\end{align*} 
Then by the  boundary maximum principle for  $  \mathcal{A}$-harmonic functions, 
\begin{align} 
 \label{4.7}  
  U   \leq     c^{2}   \de^{\hat \si}  V  \quad \mbox{in}\, \,  B ( -  e_1,  8 \de ) \sem S_1.  
 \end{align}    
Using  \eqref{4.7}  and arguing  as  above  it  follows  for   some  $  c  \geq 1 $ that 
\begin{align}
\label{4.8} 
 \int_{  \ar S_1 \cap \ar E  }  
\lan y + e_1 ,\nabla  U  \ran  f ( \nabla  U  ( y ))   |  \nabla  U  ( y ) |^{-1}  d \mathcal{ H }^{n - 1}     \leq c  \de^{ n \hat  \si }.
\end{align}     
From \eqref{4.6}, \eqref{4.8},  Lemma  \ref{lem4.1},      we  see for  $ \de > 0  $  sufficiently small that 
\begin{align}
\label{4.9} 
\begin{split}
 \int_{ \ar E}  
\lan y + e_1, &  \nabla  U  \ran  f ( \nabla  U  ( y ))   |  \nabla  U  ( y ) |^{-1}  d \mathcal{ H }^{n - 1}\\   
&\quad \approx   \int_{  \ar S \cap \ar E }  \lan y + e_1 ,\nabla  U  \ran   f ( \nabla  U  ( y ))   |  \nabla U  ( y ) |^{-1}  d \mathcal{ H }^{n - 1}
\end{split}
\end{align}  
where constants depend only on the data.

Finally,    we claim   for  some  $ c ( \de )  \geq   1 ,  $ depending only  on the data and  $ \de $  that   
\begin{align}
\label{4.10}  
c( \de )^{-1}   \leq    \frac{ u }{  U }   \leq   c ( \de )  \quad  \mbox{in}\,\,     B ( 0, 1  -   2 \de  )  \cap  K ( \pi  -  \ep).  
\end{align}

Once \eqref{4.10} is proved we get    Theorem   \ref{thmA} as follows.   Note  that  $  S  \subset  B  ( 0,  1  - 2 \de ) $  and  in  Lemma \ref{lem4.1}, 
$  \lan  x + e_1,  \mathbf{n} ( x )   \ran  =  \sin  \ep $ when  $ x \not = 0$ and    $x\in \ar S \cap  \ar E. $     Using this  note,   
Lemma \ref{lem4.1},  \eqref{4.9}.  \eqref{4.10},  and the Hopf boundary  maximum principle
   we find    that   for some   $ \bar c (\de)  \geq 1, $  depending only on the data and  $ \de, $   
\begin{align}
\label{4.11} 
 \bar c (\de)^{-1}    \leq    \int_{ \ar S  \cap \ar E  }  \sin (\ep) 
     f ( \nabla u  ( y ))     d \mathcal{ H }^{n - 1}   
   \leq     \bar c  (  \de ).  
 \end{align}  
 We   also  note  that      $  \ar  E  \cap   B  (  - 1/2,  1/4 )$   is Lipschitz on a scale of $   \ep/100. $  
 That is, if $   z   \in   \ar  E  \cap   \bar B  (  - 1/2,  1/4 ),  $      there exists  $ \ph :  \rn{n-1}  \rar   \re $ 
 satisfying $\|  \ph   \hat \|  \leq  100$ such that after a  possible  rotation of  coordinates,   
\begin{align}  
\label{4.12}      
\begin{split}
&E  \cap  B  ( z,  \ep/100)   =   \{ x = ( x' , x_n ) :  x_n  > \ph ( x' )  \}  \cap  B  ( z,  \ep/100), \\
& \ar  E  \cap  B  ( z,  \ep/100)   =   \{ x = ( x' , x_n ) :  x_n  = \ph ( x' )  \}  \cap  B  ( z,  \ep/100). 
\end{split}
\end{align}        
 From    \eqref{4.12},   \eqref{3.8},   \eqref{1.5} $(a)$,  \eqref{3.9}  
with $q$  replaced by $p/(p-1) $ (permissible by  H\"{o}lder's inequality),   \eqref{3.5} $(b)$, and    Harnack's   inequality all   applied to    $  u$ and $U$   we  see    that   
\begin{align}  
\label{4.13}
  \int_{ \ar E  \cap  B ( z,  \frac{\ep}{1000}) }      f  ( \nabla v )  d  \mathcal{H}^{n-1} \approx    \ep^{ n   - p - 1}    v^p ( z  + 10 \ep  e_n  )  \quad \mbox{whenever}\, \, v = u\, \,  \mbox{or}\, \,  v=U 
\end{align}  
  where ratio constants depend only on the data.  Using this inequality,    \eqref{4.10},   \eqref{4.3}, \eqref{4.5},  \eqref{4.2},  and   the   Hopf boundary maximum principle once again,  we obtain     that   
\begin{align}  
\label{4.14}   
 \int_{ \ar E  \cap  B ( z,  \frac{\ep}{1000}) }      f  ( \nabla u )  d  \mathcal{H}^{n-1}  
    \approx    \ep^{ n   - p - 1}    w ( z  + 10 \ep  e_n  )^p 
\approx  \ep^{\frac{ (p+1 - n )}{p-1}}   
\end{align}
where ratio constants depend on the data and 
$  \de. $    Integrating  \eqref{4.14}  over  
$ z \in   \ar  E  \cap   B  (  - 1/2,  1/4 ),  $ and interchanging the order  of  integration or giving  a  covering argument,   we  conclude  after  some arithmetic that  
\begin{align}    
\label{4.15}    \int_{ \ar E  \cap  B ( -1/2, 1/4) }     f  ( \nabla u )  d  \mathcal{H}^{n-1}  
    \approx      \ep^{\frac{ 2-n }{p-1} }.   
    \end{align}
Using  \eqref{4.15},   $  ( \la  (  \pi - \ep )  -  1) p $-homogeneity  of  $  f  (  \nabla u ), $  and  $  0  <  \la ( \pi  -  \ep)  < 1$   in    \eqref{4.11}  we  arrive at    
\begin{align} 
 \label{4.16}  
 \begin{split} 
 1  &\approx       \int_{ \ar S    }  \sin (\ep) 
     f ( \nabla u  ( y ))     d \mathcal{ H }^{n - 1} \\
     & \approx    \ep^{\frac{ (p + 1 -n) }{p-1} }    \int_0^1  t ^{ p  ( \la - 1)  + n - 2 }  dt  \\
&\approx    \frac{1}{ p  ( \la - 1)  + n - 1}  \ep^{\frac{ (p + 1 -n) }{p-1} }  
\end{split}
\end{align}
where for brevity we have written    $  \la  $ for $ \la ( \pi -  \ep ).$  Also ratio constants depend only on $ \delta $ and the data.  Clearly  \eqref{4.16} implies    that  
\[   
\la  -   1   +  (n-1)/p   \approx  \ep^{\frac{ (p + 1 -n) }{p-1} } .    
\]      
    So if $  \al  = \pi - \ep $  and we use the notation in  Theorem \ref{thmA} it follows  from 
this  inequality  that   there exist $ \de_0$ and $\ep_0$ with $0 <  \ep_0   < <  \de_0,  $  and   a  positive  
constant  $k  \geq 1$ depending on $ \de_0 $  and the data  such that   if   $ \pi - \ep_0  <  \al  <  \pi, $   then  
\begin{align}   
\label{4.17}    
k^{-1}    (  \pi  -  \al )^{\frac{ (p + 1 -n) }{p-1} }    \leq  \la_1 ( \al )  -  1  +  ( n - 1)/p    \leq   k  \,  (  \pi  -  \al )^{\frac{ (p + 1 -n) }{p-1} }.  
\end{align}
Thus   Theorem \ref{thmA} is true  once  we  prove  claim   \eqref{4.10}. 

Claim \eqref{4.10}  is  easily proved for  $ n = 2$ using  \eqref{3.10}   on both sides of  $  \ar K ( \pi - \ep ) $    
in each of  the  Lipschitz domains obtained from removing  the positive $ x_1 $  axis 
from   $ B ( 0, 1 - 2 \de)  \cap  \ar K ( \pi - \ep)$ (see Figure \ref{Kpiepsilon}) as well  as   Harnack' s  inequality and   $ u ( e_1) \approx  U (e_1)  \approx  1. $

\begin{figure}
\begin{center}
\begin{tikzpicture}[scale=2, font = \sansmath]

\begin{scope}[shift={(-3,0)}]

  \shade[ball color = gray, opacity = 0.5] (0,0) circle [radius = .5cm];

    \def\rx{.1}
    \def\ry{0.02}
    \def\z{1.5}

  \def\ra{.1}
    \def\rb{0.04}
    \def\rc{1.5}

\draw[blue, fill=blue, opacity=.5]  (1.5, -.05)--(-1.5,-0.1)--(-1.5,0.1) -- (1.5, .05) --cycle;

    \path [name path = ellipse]    (\rc,0) ellipse ({\ra} and {\rb});
    \path [name path = horizontal] (-\ra,\rc-\rb*\rb/\rc) -- (\ra,\rc-\rb*\rb/\rc);
    \path [name intersections = {of = ellipse and horizontal}];
\draw[fill = blue!50, densely dashed] (-1.5,0) ellipse ({\rb} and {\ra});

\draw[fill = blue!50, densely dashed] (1.5,0) ellipse ({.002} and {0.05});


\draw[red, fill=red, opacity=.5]  (1.5, -.05)--(4.5,0) -- (1.5, 0.05);

\node[above] at (4.5,0) {$O$};
  \filldraw (4.5,0) circle (.2pt);

\node[blue] at (-1,0) {$S_1$};
\node[above, red] at (3,.5) {$S$};
\draw[dashed, ->] (2.95,.53)--(2,0); 

  \filldraw (-1.5,0) circle (.2pt);
    \node[above] at  (-1.5,.1) {$-e_1$};
    
  \filldraw (1.5,0) circle (.2pt);
    \node[above] at  (1.5,.1) {$(-1+2\delta)e_1$};

  \filldraw (0,1) circle (.5pt);
  \node[above]at (0,1) {$(-1+\delta)e_1+\delta e_n$};
  
    \filldraw (0,0) circle (.5pt);
  \draw[dashed, ->] (0,-1)--(0,-.25);
\node[below]   at (0,-1) {$B((-1+\delta)e_1, \frac{\delta}{200})\setminus E$};
  
\end{scope}  
\end{tikzpicture}
\caption{The set $B((-1+\delta)e_1, \frac{\delta}{200})\setminus E$.}
\label{BKe}
\end{center}
\end{figure}
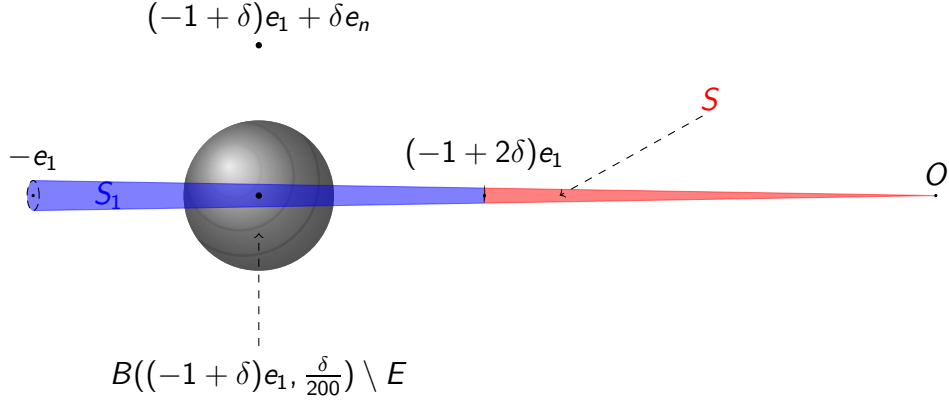

 Thus we assume  $n > 2.$   In  this case   we  give  an argument  which was  first used in  \cite[Lemma 2.16]{BL}  and  later in   \cite[section 6.1]{LN2}.    To  begin   
note that   \eqref{4.10}  on $  \ar B ( 0, 1  -   2 \de  )  \cap  K ( \pi  -  \ep) $ follows  from    
\begin{align}  
\label{4.18} 
c^{-1}  \frac{u ( (\de - 1) e_1  +  \de e_n ) }{ U ( (\de - 1) e_1  +  \de  e_n )} 
\, \leq \,   \max_{   B   ( (\de - 1) e_1,  \de/200) \sem E }    \frac{u}{ U} \, \leq  \, c   
\frac{u ( (\de - 1) e_1  +  \de e_n ) }{ U ( (\de - 1) e_1  +  \de  e_n )} 
\end{align} 
for some $ c \geq 1, $ depending only on the data,  as we see from 
  $   0  <  \ep  < <  \de  < < 1, $   $   U ( e_1)  \approx  u ( e_1 )  \approx 1, $   
   and   Harnack's  inequality for  $ \mathcal{A}$-harmonic functions.  Then \eqref{4.10}  in  $   B ( 0, 1  -   2 \de  )  \cap  K ( \pi  -  \ep) $ follows from the boundary  
   maximum principle for $ \mathcal{A}$-harmonic functions.

   To prove the  right-hand  inequality in   \eqref{4.18} 
 let    
\[   
C_t  :=    \{ x = (x_1, x') \in \rn{n} :  ( \de -   t - 1 ) < x_1 <  (  \de +  t -  1 ),\,   |x'| < \de/1000 \} 
\] 
 when $  t \in       [ \de/200  ,    \de/10 ] $ and suppose that 
$ u / U     >     \zeta   $ at some point in  $ \ar  C_{\de/200} \sem E . $   Given    $   t  \in  (\de/200,  \de/100 ) $   observe  from  Harnack's   inequality  and the maximum principle  for 
$ \mathcal{A}$-harmonic functions  that  either we have   $    u(y)/ U (y)    >   \zeta $  at some $y$ in  $  \ar  C_t \sem E   $  with  $  y_1 =  \de  \pm t  - 1 $ or   the  right-hand inequality in  \eqref{4.18} holds.   
 If $ \ze $  is large enough this  observation implies  that there exists    $ I =   [\de/200, \de/100] $ or $ I = [- \de/100, - \de/200]$  such that  
for all $ t \in I $   there is       $y''= y'' ( t )$  with $ 0  < |y''|  <  \de/1000 $  and  
\begin{align}
\label{4.18a}   
 \frac{u ( \de  + t -  1, y'' )}{ U ( \de  + t  - 1, y''  )}  >  \ze.
\end{align}
  If for example  there exists  $ t' \in [- \de/100, - \de/200]$ 
 such that for all    $  y  =  (\de   +  t'  -  1 , y'') $ in   $  \ar C_{- t'} \sem E $   we have  $ u (y)/ U (y)   \leq  \ze, $  then we can apply the above analysis  in  
\[ 
 \{ x = (x_1, x') \in \rn{n} :  ( \de +   t' - 1 ) < x_1 <  (  \de +  t -  1 ) ,   |x'| < \de/1000\} 
 \] 
whenever  $ t   \in   [\de/200, \de/100] $  to conclude the existence of  $ I  
=  [\de/200, \de/100]. $ 
 
 Let  $ \nu$ and $\tau $  denote the measures  associated  with  $ u$ and $U $  restricted to  
  $  C_{\de/10}.  $  We observe from   \eqref{3.5} $ ( b)$  that  $ \nu$ and $\tau $  are doubling measures  in the  sense that if  $  z  \in    C_{\de/100} \cap \ar E  $  and  $ 0  <  s  <  \de/200, $  then  
\begin{align} 
\label{4.19}      
 \he  (  B  ( z, 10s ) )   \leq  c\,   \he ( B  ( z, s ) ) 
\end{align} 
for $ \he \in  \{ \nu, \tau \} $ and some  $ c \geq 1 $  depending only on the data.

Given   $ t  \in I,$  choose  $ y'' ( t ) $  as  in \eqref{4.18a}.   If $ |y'' (t)| > 4 \ep $ we put  $  \rho(t)  =   |y'' (t) |.$
Otherwise since as noted earlier  $   \ar E  \cap  B ( - 1,   1/4) $  is Lipschitz on a scale of  $  \ep/100, $ we deduce from \eqref{3.10} of  Lemma \ref{lem3.2} that  there exists   $  \hat y''  = 
\hat y'' (t)  $  with  $  | \hat y'' | = 4 \ep $  and  
\[     \ze   <   \frac{ u (\de +  t - 1, y'' )}{ U (\de +  t - 1, y'' ) }     \leq  c   \frac{ u (\de +  t - 1, \hat y'' )}{ U ( \de + t - 1, \hat y'' ) }. \] 
In this case we put  $  \rho ( t )  =  4 \ep. $ Set    $ \ti y'' ( t ) =  y'' ( t) $ when  $ |y'' ( t ) |  > 4 \ep $  while   $ \ti y'' ( t ) = \hat y'' ( t ),  $ otherwise.   Using  \eqref{4.19}  and  \eqref{3.5} $(b)$ once again it follows that 
\begin{align}
 \label{4.20}   
    \zeta^{p-1}   \leq    c \left(  \frac{ u (\de + t - 1, \ti y''(t) )}{ U  (\de + t - 1, \ti y'' ( t ) )}  \right)^{p-1}   \,  \leq \,  c^2 \,  
\frac{ \nu ( B ( (\de + t - 1)e_1 ,  \rho (t) ) )}{ \tau ( B (  (\de  + t - 1)e_1 , \rho (t) ) )} \, .  
\end{align}
 Next using  a  standard covering lemma  we see there exists $  \{ t_j \},    t_j  \in I ,  $  for which   \eqref{4.20} holds with  $ t, \ti y'' ( t ), \rho ( t  ), $  replaced  by  $ t_j, \ti y'' ( t_j ),\rho (t_j). $
 Also if  $  \kappa  ( t_j)  =   (\de + t_j - 1)e_1,  $ then   
\begin{align} 
  \label{4.21}   
  \begin{split}
   L  :=  \{  y :  y_1 =   \de + t  - 1 ,   t \in I  \}  \cap  \ar  E     \subset  \bigcup_j  B ( \kappa ( t_j), \rho(t_j)), \\
 B ( \kappa (t_k),  \rho(t_k)/5) \cap  B ( \kappa(t_l) , \rho(t_l)/5 ) = \es \quad  \mbox{when}\, \, l  \not = k.
 \end{split}
  \end{align}
  From \eqref{4.19}, \eqref{4.20}, \eqref{4.21},   \eqref{3.5},  and  Harnack's  inequality    it follows,  for some $  c   \geq 1, $  depending only on the data, that
\begin{align} 
 \label{4.22} 
 \begin{split}  
    \ze^{p-1}   \tau (L )    & \leq   \ze^{p-1}   \tau \left( \bigcup_j  B  ( \kappa(t_j), \rho ( t_j ) ) \right) \\
    &   \leq  \ze^{p-1}   \sum_j  \tau ( B  ( \kappa(t_j), \rho ( t_j )) )  \\   
&\leq c    \sum_j  \nu ( B ( \kappa(t_j), \rho ( t_j) /5)  )   \leq c^2 \nu ( L ).  
\end{split}
\end{align}

Also  from  \eqref{2.8}  and Harnack's   inequality   we see that 
\[ 
\de^{p-n}  \tau ( L ) \approx  U  (  ( \de - 1) e_1  + \de e_n )^{p-1}   \quad  \mbox{and}\quad  \de^{p-n}   \nu ( L )  \approx  u  (  ( \de - 1) e_1  + \de e_n )^{p-1}
\]  
where ratio constants depend only on the data.  
Using these inequalities  in    \eqref{4.22}  we find that  
\[   
\ze      \leq  c  \frac{ u ( ( \de - 1) e_1 + \de e_n)}{\hat U  ( ( \de - 1) e_1 + \de e_n)}.
 \]  
 The right-hand  inequality in \eqref{4.18}  follows from this  display and Harnack's inequality for 
$  \mathcal{A}$-harmonic functions  with  
$  2 \ze =  \max (u/U)  $  on  $  \ar  C_{\de/200} \sem  E. $   Interchanging the roles of  $  u$ and $U $  in  this  argument we get the left-hand  inequality  in   
\eqref{4.18}.  This  completes the  proof  of  claim \eqref{4.10}  and  so  also of  Theorem  \ref{thmA}.   

\setcounter{equation}{0} 
 \setcounter{theorem}{0}

 \section{Proof of  Theorem \ref{thmB} } 
 \label{section5}  
We  begin this section with a  discussion of some familiar concepts from convex geometry which 
were used in \cite{J,CNSXYZ} to prove  analogues of  Theorem  \ref{thmB}.   
  Let  $ E \subset  \rn{n} $  be  a  compact convex set with non-empty interior. Translating and dilating  $ E $  
  if necessary  we may assume that   $  \bar B (0, 1)  \subset  E$  is  a ball with largest  radius contained
   in  $ E$  while  $ \bar B ( 0,   \ti  R_0) $ is  the   ball with  smallest radius  and center at the origin  
   containing $ E$ for some $\ti R_0>0$.   Then  $   \ti e   =  1 / \ti R_0  $ is  called the  eccentricity of  $ E. $  
     From basic geometry one sees that  if $ w  \in   \ar E $   there exists  $ \hat c =  \hat c(n) \geq 1, $  depending only on $n$  
   such that  $ \ar  E $  can be covered by at most  $ N = \hat c^2  (\ti  e)^{1-n}   $ balls,  $  B (  w  ,    \ti  r  ),$    with   
$  w  \in \ar E,  \ti  r  \geq  1/8 ,   $  and  the property that  after a possible change of  coordinates  there exists
  a  real valued  convex function $ \ph $   on   
  \[
   \bar B'  ( w' ,   \ti  r ) := \{  x'  = ( x_1, \dots, x_{n-1} )  : |x' - w' | \leq   \ti  r \}
   \]
 which  extends to  a  Lipschitz function on
  $ \rn{n-1} $  with  $  \| \ph  \hat \|  \leq  \hat c/\ti e$. Moreover, if we let  $w = ( w', w_n)$ and 
$\ph ( w' )  = w_n, $ after a possible change of  coordinates,     we also have
\begin{align} 
\label{5.1}
\begin{split}  
& \{ ( x', x_n ) :  x_n =  \ph ( x' ) \, \, \mbox{and}\, \, x'  \in  \bar B' ( w', \ti r) \}  \subset   \ar E,\\  
&\{ x =  ( x', x_n ) :  x_n  >  \ph ( x' )    \} \cap B ( w,  \ti r )  \subset  E \sem   B  ( 0 ,  1/2).
\end{split}
\end{align}

\begin{definition} 
\label{def5.1}  
Let  $ \psi  $  be  a  real valued convex function on a  
bounded  convex open  set  $  \Om  \subset \rn{n-1}. $  If  $ x'  \in \Om $  we  write    $   \he =  ( \he_1, \dots,  \he_{n-1} ) \in \ar \psi ( x' )  $  provided  
$   \psi ( y') \geq  \psi ( x' )  +    \lan \he,  y'  -  x' \ran  $  whenever  $  y'   \in  \Om.   $      
 If  $  \tau  $  is   a  finite  positive Borel measure on  $ \Om$   then    $ \psi $ is said to be  a  solution to the  Monge-Amp{\`e}re equation
\begin{align}
 \label{5.2}  
\text{det}(\nabla^2 \psi ) = \tau \quad \mbox{on}\, \, \Omega
 \end{align}
 in the sense of  Alexandrov  provided that 
\begin{align}
  \label{5.3} 
   \mathcal{ H}^{ n - 1}  \left(  \bigcup_{x' \in K}   \ar \psi (x')  \right)   =  \mathcal{H}^{n-1} \left( \ar \psi ( K )  \right)  =  \tau (K) \quad \mbox{for each  Borel set   $ K  \subset \Om$}.
   \end{align}
 \end{definition} 
Let  $ \mathbf{g}(\cdot, E)$  denote the  Gauss  function for  $ \ar  E, $     suppose  \eqref{5.1} is valid,    and  set $  \Om =   B' ( w' ,  \ti r )$ and $\ph =  \psi$.   If  $ x'  \in \Om$  then  one can define 
\begin{align}
 \label{5.4}  
  \mathbf{g} (( x',  \ph ( x' )),E)   =  \left \{ \frac{(   \he  , - 1   ) }{( 1 + |\he|^2)^{1/2} } :\, \,  \he   \in  \ar  \ph ( x'  ) \right\}.  
 \end{align}

 We note that  the  mapping   $ x'   \mapsto \frac{ \lan x', - 1 \ran  }{ (1+|x'|^2)^{1/2}} =  \xi   $  is  one to one  from $ \mathbb{R}^{n-1} $   onto  
$  \mathbb{S}^{n-1} \cap \{ \xi : \, \xi_n < 0 \}. $   Moreover, the inverse of this mapping    has   Jacobian $ |\xi_n|^{-n}$  at  $ \xi$ with $|\xi |  = 1$ and $\xi_n    < 0 $.  Using this fact, it follows  from 
\eqref{5.2}, \eqref{5.3}, and \eqref{5.4} that if  $ K \subset B ( w', \ti r ) $ is a Borel set  and 
$ \ti K : = \{ ( x', \ph ( x' )) : x' \in K \}$   then
\begin{align}
 \label{5.5}  
  \mathcal{ H}^{ n - 1} ( \ar  \ph (K) )   =   \tau ( K ) =  \int_{ \mathbf{g} (\ti K, E)} \, |\xi_n|^{-n}   d  \mathcal{H}^{n-1}
  \end{align}  
 in the sense of  Alexandrov.  Next  suppose for  fixed  $ p$ with  $1 < p <  \infty$   that  $ U$, $E$,  and  $\mu $  are as   in  Theorem  \ref{thm1} or  $U = 1 - \ti U $  where  $ \ti U $ is  as in   Theorem \ref{thm2}.  Then  for   $ \mathcal{H}^{n-1} $-almost  every  $ y  \in  \ar  E $ we  see from  Theorems  \ref{thm1} and  \ref{thm2}  that   
\begin{align}
\label{5.6}    
\mathbf{g} ( y ,E)  =     \frac{ \nabla  U ( y ) }{ |\nabla  U (y) |} \,  .
\end{align}
Thus   $ \mathbf{g}(\cdot, E)$  is  well  defined  by   \eqref{5.6} on a  Borel set  $ E_1  \subset  E $   with  $    \mathcal{H}^{n-1} ( E \sem E_1) = 0.  $   If  also  $ d \mu  =  \He    d \mathcal{H}^{n-1}|_{\mathbb{S}^{n-1}}  $   and  there exists  $  \ti a_3  \geq 1 $   such that    
\begin{align}
 \label{5.7} 
 0   <  \ti a_3^{-1}       \leq    \He   (\xi)  \leq \ti a_3  \quad  \mbox{for}\, \,  \mathcal{H}^{n-1}\mbox{-almost  every}\, \, \xi  \in \mathbb{S}^{n-1},
 \end{align}
then from    $  \|  |\nabla \ph|  \|_{\infty}  \leq c/ \ti e  < \infty , $  finiteness and positivity  of  $ \mu,  $   as well as  the  Radon-Nikodym theorem  we conclude for  $ \tau $ and $K$ as in    \eqref{5.5}
 that 
\begin{align}
\label{5.8}     
\begin{split}
 \tau ( K ) &=  \int_{ \mathbf{g} ( \ti K, E)}       |\xi_n|^{-n} \, d \mathcal{H}^{n-1} \xi  \\
 &=  \int_{K}  \frac{( 1 + |\nabla \ph |^2 (x') )^{(1+n)/2}  f ( \nabla  U ( x', \ph (x') ))} {\He( \mathbf{g} (( x', \ph ( x') ), E))}  dx'.
\end{split}
\end{align}
Thus to prove regularity of  $ \ar E, $  we study  the  Monge-Amp{\`e}re equation  in  domains of the form  $ \Omega=B' (w' ,  \ti r  )$   with 
measure as in  \eqref{5.8}.     
To outline  some of  the work of  previous authors on  the  Monge-Amp{\`e}re equation  we need several definitions.  
\begin{definition} 
\label{def5.2}  
Given  $ \psi$,  $\Om $ as in  Definition \ref{def5.1}    and  $ x'  \in  \Om$,  $t > 0$, $\he  \in  \ar \psi ( x' )$, we put 
\begin{align}  
\label{5.9}   
S ( x',  \he, t)  :=  \{ y'  \in  \Om  :\,    \psi ( y' )  -  \psi ( x' )   -  \lan \he,  y' - x' \ran  <  t  \}  
\end{align}
and  call     $  S  = S ( x',  \he, t ) $  a  cross  section of  $ \psi . $  Define the reduced distance $ \de ( \cdot, S ) $  on  $  S (  x', \he, t ) $  by   
\[
    \de ( z',  S ) =  \min \left\{   \frac{ | z' - \hat x |}{ | z' -  \hat y |} : \, \,   \hat x, \hat y \in \ar S  \mbox{ and  $ z ' $ lies on the line segment from  $ \hat x $ to  $\hat y $}  \right\}.  
\]  
\end{definition} 
Note  from convexity of  $ \psi $  that 
$ S (  x', \he, t ) $  is  a  convex  set.  Let   $ \bar x' $ denote the centroid of  $ S $  and for   $0 <  \la < 1$, set
\[
  S ( x', \he, t, \la) :=  \{ z'   :  z'  =  \la ( y' - \bar x' ) +  \bar x',  \, \, \mbox{such that}\, \, y'  \in S ( x', \he, t )  \}.
  \]
 For  ease of  writing,   for $y'  \in \Om$, we  put
\begin{align}  
\label{5.10}   
\begin{split}
S_\la&  :=  S  ( x', \he, t, \la),\\    
l_{x',\he} (y')   &:= \psi ( x' )  +  \lan  \he, y'  - x'  \ran,\\
  \ti \psi (y' )  &:=  \psi (y' )  -  l_{x', \he} \, ( y' ) - t  
  \end{split}
 \end{align} 
 when $ x'$, $\he$, $t $  are understood.       
Then from a theorem of  John  (see \cite[A.3.2]{F})  it follows  that there exists a  unique   ellipsoid,  $ \mathcal{E}$  of  
maximum volume  with  $ \mathcal{E}  \subset  S   \subset (n - 1) \mathcal{E}. $   Using this fact and  
basic  geometry  we  deduce the existence of  a  positive constant $  \be (n)$  and  an  affine  mapping of  
the  form   $ T z'  = A ( z'  -   \bar x') $ for $z' \in \rn{n-1}$   where $ A $ is an $ n-1 \times n - 1 $  nonsingular  matrix with  
$  T ( \bar x' ) = 0$   and   
\[       
B' ( 0,  \be(n) )   \subset  T (  S  )   \subset  B' ( 0, 1).  
 \]    
Here $ T ( S ) $ is said  to be  a normalization of $S$.     Note that 
$ \de ( z' ,  T ( S ) )  =  \de ( T^{ - 1} z'  ,   S ) $   and  if  
 $  \Psi ( z' )  =  \psi ( T^{-1}  z'  )$ for $z' \in  T (S)$  then  $ \Psi $ is   convex and   
\[    
 \ar \psi ( x' )    =   A^t   \, \ar \Psi (Tx') 
 \]  
 where $ A^t $ is the transpose of $ A. $   Also  
$ \Psi  $  is  a  solution  to the  Monge-Amp{\`e}re  equation in $ T (S) $  with measure $ \mathcal{T}  $  where  
 \[ 
  \mathcal{ T}  ( T (K)  )   =  (\mbox{det } A^{-1})  \, \, \tau ( K  )  
   \] 
    whenever  
$  K  \subset  S $ is a Borel  set.   Finally,  let  $  \ti \Psi ( T y' )  = \ti  \psi ( y' )$ for $y' \in\Om. $ 
  
 Using  the  above  normalizations it was shown in \cite[Lemma 7.3]{J}  that   
\begin{lemma}
 \label{lem5.3}   
 Let  $ \Om,   \psi,  \tau,  t,  x', \he,    $  be as  in   Definitions   \ref{5.1}, \ref{5.2},   and  suppose  $ \bar S ( x' ,  \he, t ) \subset \Om. $ 
 Then given  $  0 < \ep \leq 1, $   there  is   a  positive constant  $ C ( \ep, n )$ such that 
\begin{align}  
\label{5.11}  
|  \ti \Psi (z')  |^{n-1}  \leq   C ( \ep, n )   \,  
\de ( z' , T(S))^{\ep}          \int_{T(S)} \,  \de ( y' , T(S) )^{ 1 - \ep }   d \mathcal{T}( y') 
   \end{align}
   whenever $z' \in T(S) $. 
\end{lemma}  
\begin{proof}  
See Lemma  7.3  in    \cite{J}.  
\end{proof} 
Our  goal  is  to   show  for     $  \psi  =  \ph$,  $\Om =  B' ( w' ,  \ti r ), $  and    
$\He, \ti a_3 $  as in  \eqref{5.7}, \eqref{5.8},     that  there exists  $ \ep_0  \in ( 0, 1] $  for which     
\begin{align} 
\label{5.12}  
\int_{S}  \de ( y' , S )^{1-\ep_0} d \tau (y' )  \leq   \hat   C  \tau (  S_{ \frac{1}{2}}   ) 
\end{align} 
 whenever  $  \bar S  = \bar S ( x', \he, t ) \subset  \Om$  where  $ \ep_0$ and  $\hat  C $ depend on  the data,  $ \ti e,  $ and 
  $ \ti a_3 $ in  \eqref{5.7}.   
  
Before  proving \eqref{5.12}  we show as in  \cite{J} and  \cite{GH}, 
how \eqref{5.12} can be used to prove  Theorem  \ref{thmB}.  Indeed, normalizing this problem   we  deduce  first  that if  $  \ti \Psi ( T z'  )  =  \ti \psi  ( z'  ),    $  then    from 
   \eqref{5.12}   it follows  as in  Proposition 2.10 of  \cite{GH}  that       
\begin{align}  
\label{5.13}  \int_{T(S)} \,  \de ( z' , T(S) )^{ 1 - \ep_0 }   d \mathcal{T}( z') \leq  \hat C   \mathcal{ T } ( S_{\frac{1}{2}} )   \leq   C'  \, | \min_{T(S)} \ti \Psi |^{n-1}   = C'  t^{n-1}.
\end{align}
Using  this  inequality in   \eqref{5.11}  of  Lemma  \ref{lem5.3}  with $ \ep  = \ep_0 $   and  
$  z' = T ( x') ,  $  $  \ti \Psi (T ( x' ) ) = - t, $   we deduce      
\begin{align}  
\label{5.14}  
1   \approx   d ( T ( x' ),  \ar  T ( S ) ) 
\end{align} 
where ratio constants  have the  same dependence as  $  \hat  C $  in  \eqref{5.12}.   Using   \eqref{5.12}-\eqref{5.14}, it follows 
 that 
\begin{lemma}  
\label{lem5.4} 
 Let  $ \psi $ be a  real valued convex function on the convex open set $ \Om, $  and continuous on $ \bar \Om. $      
  If   $  \psi  \geq   l $ on $ \ar \Om $  where  $ l  $ is  an  affine   function and    $  \psi ( z' ) = l ( z' ) $  
  at some point  $ z'  \in \Om, $  then  either   $ \{ y'  :   \psi ( y' )  = l ( y' ) \} = \{ z' \} $  or  
  this set has no extremal  points in $ \Om. $  
  \end{lemma}    
  \begin{proof}  
  See  Theorem 1 in  \cite{C2}  or  Theorem 4.1 in  \cite{GH}.   The proof in either paper  is   by  contradiction  and   uses  invariance of  \eqref{5.11} and \eqref{5.12}
  under affine mappings as well as  the following  result.   Suppose  that   $ \hat \psi_j$  for  $j  = 1, 2,  \dots,  $  are 
   convex functions and   solutions to   the  Monge-Amp{\`e}re equation  with measures $  
\hat \tau_j $ in  an  open    set $ \hat \Om$. If   $  (\hat  \psi_j ) $  
converges uniformly  on compact subsets of  $ \hat \Om $ to  $  \hat \psi, $  a  
solution to the  Monge-Amp{\`e}re equation in  $ \hat \Om $  with measure $ \hat \tau,  $  
then  $ \hat  \tau_j  \rightharpoonup   \hat  \tau$  weakly in  $ \hat  \Om $ (see \cite[Lemma 1.2.2]{G}).
\end{proof} 

Applying Lemma \ref{lem5.4}  with $  \psi = \ph$ and $\Om = B' ( w', \ti r ) $ as in   \eqref{5.1}  we  see that  
  $ \ar  E $ is strictly convex   since  otherwise it would follow  from repeated application
   of   Lemma  \ref{lem5.4}  to balls  (as in  \eqref{5.1}),  with non-empty intersection,   that    
$  \ar E $  contains a  line segment of infinite length.   From this contradiction we  
conclude  that $ \ar E $ is  strictly convex.   Now  given  $  w = 
(w', \ph(w') ) \in \ar E,  $   with  $  x'  \in  \bar B (w', \ti r/4),$   where  $ \ph, \ti r $ 
 are as in  \eqref{5.1},  we  choose  $ t $ so that  for  $ S ( x',  \he, t ) $ 
 as in     \eqref{5.9}  we  have  
 \[
 S ( x',  \he, t  )    \subset   B' ( x', \ti r/2)  \quad \mbox{and}\quad   \bar S ( x',  \he, t )    \cap    \ar B' ( x', \ti r/2) \not = \es. 
 \]   
 Geometrically this means  there  is  a point  $ z  =  ( z',  z_n ) \in   \ar E $ 
 with  $  z'  \in   \bar S ( x', \he, t) \cap \ar B' ( x' ,  \ti r/ 2) $    
 which  lies at most  $  t $ distance from the support 
  plane $y_n =  l_{x', \he} (y') $ to  $  \ar  E $ at $ (x', \phi (x') ).$    
  We  claim that     $ t  \geq   t_0   >  0, $  where $ t_0 $ 
  has the same dependence as  $ \hat C $   in \eqref{5.12}. 
Indeed, otherwise  using  a  compactness argument,  the 
above convergence result,  and  Lemma \ref{lem5.4}  we could obtain  
a contradiction to the strict convexity of  $ \ar E. $   Finally,   
we observe from Lipschitzness of  $ \ph $ as in \eqref{5.1} that  there exists 
$   r_1  \geq  r_0 > 0 $  where $ r_0 $ has the same dependence as  $ t_0 $  
with  $  B' ( x' ,   r_1)  \subset S . $ Next  if $  \ti \phi $  is as in   \eqref{5.10} 
with  $ \phi =  \psi, $   we claim there is  a $ \si   > 2,  $ with  the same dependence as  $ t_0,  r_0, $  satisfying     
\begin{align}  
\label{5.15}   
0   \leq    t  +    \ti \ph ( y' )  \leq     \si^{-l} \,  t  \quad  \mbox{when}\, \,   l = 1, 2, \dots , \, \, \mbox{and}\,\,  y'  \in \hat S_{1/2^{l}}. 
\end{align}   Here   $ \hat S_\la $ is defined in the same way as $ S_\la $ in \eqref{5.10}  only with $ \bar x' $ replaced by $ x'.$   
 Indeed,  this inequality holds for $ l = 1 $  since otherwise we could 
  use Lemma \ref{lem5.4}  and  a  compactness  argument,  as above,  
  to contradict the  strict convexity of   $   \ar  E. $   Iterating  this inequality   
  we  obtain  \eqref{5.15}.    From  \eqref{5.15} and arbitrariness of  $ x'   \in  B' ( w',  \ti r/ 4) , $   we  get first that    
\begin{align}  
\label{5.16}   
\begin{split}
&\nabla \ph ( x' ) \, \,   \mbox{exists  for   $ x' \in B' ( w', \ti r/4) $}, \\
&|\ph (y') - \ph (x') -  \lan \nabla \ph(x'), y' - x' \ran     |  \leq \hat C_1  | y' - x' |^{ 1 + \al'}  
\end{split}
\end{align}
whenever  $  y'   \in  B' ( w', \ti r/4), $    where  $ \hat C_1 \geq 1,  \al' \in (0,1),  $   depend  on  the data,  $ \ti e,  $ and 
  $ \ti a_3  $ in  \eqref{5.7}.   Also from convexity and uniform  Lipschitzness of  $ \ph $  we deduce the existence of  $  \hat C_2 \geq 1,   $    
 having the same dependence as  $ \hat C_1 $   for which   
\begin{align}
\label{5.17}  
 \hat  C_2  \,  | \lan  \nabla  \ph  (y')   - \nabla   \ph (x') , y' - x' \ran    |  \geq    |  \nabla   \ph  (y')   -  \nabla  \ph (x') | | y' - x'   |  
\end{align}    
whenever $ x', y'  \in  B' ( w' , \ti r/4)$. Combining  \eqref{5.16},  \eqref{5.17}, and  using the triangle inequality,    we  find  that  
\begin{align} 
\label{5.18}    
\mbox{$ \ar E $ is locally   $ C^{1,  \al' } $ with   norm constants depending only on the data, $ \ti e,  $  and  $ \ti a_3 $.} 
\end{align}
Using  \eqref{5.18}   and   results  from  \cite{Li}   we see  that  $ \nabla  U $ when $ 1  <  p < n $  or   $ \nabla \ti U $  
when  $ n\leq p <\infty $   has  a    $ C^{1,  \be ' } $   extension  to  $  \ar E $  for some $   \be'  \in (0, 1) $  
having    the same dependence  as    $ \al'.$     Also from   \cite{Li} or  \eqref{5.23}   (to be proved)  we have   
   $ \min \{ |\nabla U |,  | \nabla \ti U |  \}  > 0 $   on  $  \ar E  $   where   constants depend 
   only on  the data and $  \ti e. $   In view of  this information and  \eqref{5.2},  \eqref{5.5},  \eqref{5.8}, 
  we   find  that  if   
$ 0 < \He   \in   C^{0, \hat \al} ( \mathbb{S}^{n-1} ) ,   $  then  for   some    $ 0 <  s_1,   s_2,  \al_*, $   having  the same dependence as $ \al' , $    
\begin{align} 
\label{5.19}    
s_1    <    \frac{d \tau}{d \mathcal{H}^{n-1}}  <  s_2   <   \infty  \quad \mbox{and}\quad   \frac{d \tau}{d \mathcal{H}^{n-1}}  \in C^{ 0,  \al_*  } ( \Om ).        
\end{align}
  From the above remarks,   \eqref{5.19},   and  \cite{C,C1,C3}, we conclude that  $  \ph    \in   C^{2, \hat \al} ( \Om ). $   
  Further applications of \cite{C,C1,C3}  also give the  higher order smoothness results in  Theorem \ref{thmB}.   

It  remains to  prove  \eqref{5.12} in order to complete the proof  of  Theorem \ref{thmB}.  
 Throughout the proof of  this inequality  we let $ C \geq 1 $  be a positive constant  
 which may depend only on  the data,  $ \ti e, $  and  
   $ \ti a_3,  $  not necessarily the same at each occurrence.  Also if $ A \approx B, $  
   proportionality constants may depend on  the data, $\ti e, $ and  $ \ti a_3 . $   
   Let  $ \hat f ( \eta ) = f ( - \eta ) $ when $ 1 < p < n $  and  $ \hat f ( \eta ) = f ( \eta ) $    when  $ n\leq p<\infty$.  
   Also set  $ \hat U =  1 - \ti U $  when  $ 1 < p <  n $  and  $ \hat U  =  U $  when $ n\leq p <\infty. $  
    Then  $ \hat U $ is  $ \mathcal{\hat A}    =   \nabla  \hat  f$-harmonic in  $ \mathbb{R}^n \sem E $ 
    with continuous boundary value 0  on  $ \ar E. $  Observe from   the  discussion  above  \eqref{5.1},   \eqref{5.7}, and  \eqref{5.8} that    
if   $  K\subset\Om  $  is  a compact   set  then    
\begin{align}
 \label{5.20}  
 \tau ( K  )  \approx   \int_{K}  | \nabla  \hat U ( x', \ph (x') )  |^p    dx'    =   \,  \,   \chi ( K  ). 
 \end{align}
  Thus  we  only  prove  \eqref{5.12}  for  $  \chi(\cdot)$.  Recall  that    
$  \bar S  = \bar S ( x',  \he, t ) \subset  \Om = B' ( w', \ti r ). $  We see that 
\[
F    =  \{ ( y', \ph ( y' ) ) :  y' \in  \bar S  \}
\]  
 is  the part  of  $  \ar  E $  that  lies  below or on the  plane  
\[  
\Si_1  =\{ ( y', y_n) :  y_n  -   \phi ( x' )  -   \lan \he, y' -  x' \ran   = t  \}  
\]  
and above or on  the support plane  $  \{ ( y', y_n) :  y_n  -   \phi ( x' )  -   \lan \he, y' -  x' \ran   = 0  \} $
 to  $ \ar E $  at  $ x =  ( x',  \ph (x') ). $ 
Then   $ S $  can be  viewed as  the projection of $ F $  onto  the plane $ y_n = 0 $  
by lines parallel to $ e_n $  or  the  $ y_n $  axis.    To  simplify   the  geometry in what follows  and for  use  in  adapting
  the work in \cite{J} to our situation   we  also project  $  F  $    onto   $  \Si_1 $  
  by  lines  parallel  to    $   e_n . $ 
    More specifically,  given  $ y  =  ( y',  \ph (y') )  \in F,  $     let  $   \pi (y) \in \Si_1  $  be that  point   with    
    \[  
     \lan \pi (y),  e_i  \ran  =  y_i'   \quad \mbox{for}\, \,  1 \leq i \leq n - 1.   
    \]    
Let $  \ti S = \pi ( F ), $ and note that 
$  \ti S  $ is convex.  Define the reduced distance  $ \de ( \cdot, \ti S ) $   
as in Definition \ref{def5.2}  with $ S$ and $\rn{n-1}, $ replaced by $ \ti S$ and $\Si_1$ respectively.     
From  \eqref{5.1} and the discussion above this display we deduce that  
\begin{align} 
\label{5.21}  
\de (  x', S )  \approx   \de ( \pi ( x ), \ti S )  \quad \mbox{whenever}\,\,  x  \in F 
\end{align}
 where ratio constants  depend only on  $ n, \ti e. $   Let    
$  \de ( x, F )  =   \de ( \pi ( x ),  \ti  S  ) $ when $ x \in F. $     
  Then  from    \eqref{5.20} and  \eqref{5.21}  
we conclude that  to prove   \eqref{5.12}  it suffices to  show,         
\begin{align}
 \label{5.22}  
\int_{F} \de^{1-\ep_0} ( x, F ) \,   |\nabla \hat U ( x ) |^p   d \mathcal{H}^{n-1}  \leq  C  \min \{ |\nabla  \hat U  |^p  :  y \in F \} \,  \,  \mathcal{H}^{n-1} (F). 
\end{align}
\begin{remark} 
\label{rmk5.5}  
We first remark that \eqref{5.22} is the exact counterpart of Theorem 6.5 in \cite{J}. We also note that in   \cite[section 6]{J},   the analogue of  
$ F $ is projected onto $ \Si_1 $   by  rays  through the origin. If  $ P ( y ) $  denotes this   
radial  projection of  $ y \in F $   onto   $ \Si_1, $   then    in   \cite{J} the  reduced 
distance of  $ y \in F $  is  defined to be equal  to  $  \de ( P (y), \ti S ) . $   Using  the  
definition of   reduced  distance and   \eqref{5.1} it is  easily verified as in  \eqref{5.21}
 that $  \de  ( P (y),  \ti S )  \approx    \de (\pi (y), \ti S) $   where  proportionality 
 constants depend  only on $ n$  and  $  \ti e.  $     Thus   \eqref{5.22} implies  the  corresponding inequality in 
\cite{J}  and  vice-versa. 
\end{remark}   

To prove \eqref{5.22}  we  shall require the following lemma.  
\begin{lemma}
 \label{lem5.6}  
 Let   $ w$, $E$, and  $\ti r$  be  as in    \eqref{5.1}.  There  exists  $ C \geq 1,  $  depending only on
  the data and $ \ti e $, such  that  if  $ 0 < r \leq \ti r/C$  then     
\begin{align}
 \label{5.23}  
 r^{ 1 - n}  \, \int_{\De ( w, r ) }   | \nabla \hat U  |^p  d \mathcal{H}^{n-1}   \approx    \min_{ \De (w, r) }  | \nabla \hat U|^p   \approx    r^{-p}   \, \hat U  (  a_r ( w ) )^p.
 \end{align}
\end{lemma}

\begin{proof}
  Let  $  H $   be  an open    half-space   with   $  H \cap  E  =  \emptyset  $  and  $ \ar H$   a 
  support plane for  $  \ar E  $   at $ w.$   Let  $ \xi  $  denote a unit normal pointing into 
  $  H  $    and   let     $  v $  be the    $  \mathcal{\hat A}$-harmonic  function in  
  $  \mathbb G  =   H \cap B ( w, r )  \sem  \bar B ( w + r \xi /2 ,   r/8 )  $   
  with  continuous boundary values,  $  v  \equiv  0 $  on  $   \ar ( H \cap  B ( w, r ) ) $   while   $  v  \equiv  
\hat U  (w + r \xi/2)  $  on $   \ar B ( w + r \xi/2, r/8). $    Comparing boundary values and using  
Harnack's inequality for $ \mathcal{\hat A}$-harmonic functions we see that  
$ v   \leq  c \hat  U $ in  $ \mathbb G $.  Also using the  boundary   Harnack inequality 
in  Lemma  \ref{lem3.2}   and comparing  $ v $  to  a  linear  function,  say  $ l,  $  
which   vanishes  on   $  \ar H $  with   $ l  ( w  + r\xi /2)  =  \hat U  ( w + r \xi/2 ) $  
we arrive at   
\[
\hat U ( w + r \xi/2 )/ r      \leq   C'  \hat U ( w + s \xi )/s    
\]   
whenever $  0 < s  \leq  r/C' $  where $ C' \geq 1 $ depends only on the data.  Letting $ s  \to 0 $       
 in this display  we  get  from   Lemma  \ref{lem3.1} for $  \mathcal{H}^{n-1}$-almost every $ w \in \ar E$ that    
\begin{align}  
 \label{5.24}    
  \hat U ( w + r \xi/2 )/ r      \leq   C'   | \nabla \hat U ( w ) |.
 \end{align}
 Next    observe from   \eqref{3.9} with $ q = (p-1)/p$, \eqref{3.8} of   Lemma  \ref{lem3.1},  and  
$ \eqref{3.5}(b)  $
 that   there exists  $ C \geq 1 $  depending only on the data and $ \ti e $ such that 
for $ 0 < r \leq  \ti r/ C, $  
\begin{align}
 \label{5.25}   
 r^{1-n} \,  \int_{\De ( w, r )}  | \nabla \hat U |^p  d \mathcal{H}^{n-1}   \approx  C  r^{ - p}  \,    \hat U ( w + r \xi/2)^p . 
 \end{align}
Combining   \eqref{5.25},   \eqref{5.24},  and  using  arbitrariness  of  $ w, $   Harnack' s  inequality for 
$ \mathcal{ \hat A }$-harmonic  functions,  we conclude the validity  of  
      Lemma \ref{lem5.6}.   
      \end{proof} 
      Note from Lemma \ref{lem5.6}  that 
\begin{align}
 \label{5.26}  
 \int_{ \ar E } | \nabla \hat U |^p   \leq C \quad   \mbox{and}\quad   \min_{\ar E } | \nabla  \hat U |  \geq C^{-1}.   
 \end{align}
Following \cite[Lemma 6.7]{J}  we first  note  from  \eqref{5.1}     that  if $ \ze \in F $  and  $ b $  
denotes the radius of  the largest  $  n - 1 $ dimensional  ball  contained in  $ \ti S $  (the  so  called  inradius of   $ \ti S $)  then    
\begin{align}
 \label{5.26a}    
  | \ze   -   \pi ( \ze  ) |    \, \leq \,  C_+  b 
 \end{align}
  for some $ C_+ \geq 1, $  depending only on the data and   $ \ti e. $  Second we state  
\begin{lemma} 
\label{lem5.7}  
If $  y, z \in F $ and  $ \de ( y, F ) \approx 1,  $ then   
\begin{align}
 \label{5.27}    
 \min_{  \De ( y, b) }  | \nabla \hat U | \leq  C  \min_{ \De ( z, b) } | \nabla \hat U |.
 \end{align}
 \end{lemma}  
\begin{proof} 
The  analogue of   Lemma  \ref{lem5.7}   in \cite{J} is  Lemma 6.8.   Given  Lemma  \ref{lem5.6}  and  \eqref{5.26a} 
we  can essentially copy the  clever geometric argument in \cite{J}, so we  refer to this paper for details.  
\end{proof}   
\begin{lemma} 
\label{lem5.8}  
There exists  $  \ep_1 \in  (0, 1]$ and $C \geq 2$  depending only on the data, $ \ti e, $  such that $ \ep_1  \geq C^{-1} $   when $ 1  < p \leq n - 1 $   while   
$ \ep_1  \geq  1  + (1 - n)/p  + C^{-1} $  when   $ p  > n -1. $  Moreover, if    $ \hat x \in F, $  then  
\begin{align}  
\label{5.28}   
b^{ 1 - n }   \,  \int_{\De ( \hat x, b ) }  |\nabla \hat U |^p   d \mathcal{H}^{n-1}  \leq C  \de ( \hat x, F )^{ - p  + p \, \ep_1}  \, \min_{F }  | \nabla  \hat U |^p.
\end{align}      
\end{lemma}  
\begin{proof} 
   As  in  Lemma  6.13  of    \cite{J}  we  note    that  if  $  \de ( \hat x, F )  \approx 1 $  then       \  \eqref{5.28}   
 follows from Lemmas \ref{lem5.6} and \ref{lem5.7}.   Thus we assume that 
$ \de (\hat  x,  F ) < <  1 $  and choose  $ y, z  \in F $  so  that  $  \de ( y, F ) \approx 1 $ 
 and  $ \pi (z ) = z   \in F \cap  \ti S $  with $ \pi (\hat x ) $ lying on the line segment  
 from $ \pi ( z ) $ to  $ \pi ( y ). $  
Let $ \rho =  | \pi ( z ) - \pi ( y ) | . $   We  note  that  if   $  \rho < 100\, b ,  $  
then  \eqref{5.28} follows from   Lemmas \ref{5.7},  \ref{5.6}  with  $ w = \hat x, y, $  \eqref{5.26a}, 
 and   Harnack's inequality for $ \mathcal{\hat A} $-harmonic   functions  with  $ \ep_1 = 1. $  
 Thus we assume $ \rho  \geq   100 \, b.$       
Then from the  John  ellipsoid  theorem mentioned  below \eqref{5.10}   we  deduce     
\[
  \bar \de  :=  | \pi (\hat  x ) - \pi ( z) | / \rho   \approx   \de ( \hat x, F ) 
  \]    
  so  we assume, as we may,  that  $ | \pi ( \hat x )  -  \pi ( z ) |  <  \rho/100. $    Next  we  define   the cone:  
\[  
\Ga  =  \{    \pi (z)  + s ( \ze - \pi ( z ) ) :  \ze  \in  E  \cap \bar B ( y,  \rho/2 )  \mbox{ and }  0 < s < \infty \}. 
 \]   
From  convexity of  $ E $    we see that  $ E \cap  B  ( y,  \rho/2), $  contains  a  ball of  radius  $  \rho/C'',  $  where  $ C'' $ depends only on the data and $ \ti e.$    
        Let   $   X  $  denote a  point that lies $ \rho  $ distance from $ \Ga \cup E $  and   at most  
$ 2 \rho $  from  $ y . $    As  in   the proof  of  Theorem \ref{thmA}  we first  construct    $   V  $   
a   positive  $  \mathcal{\hat A}$-harmonic function in  $ \rn{n} \sem  \Ga $  which is continuous in $ \rn{n} $  with 
$ V \equiv 0 $ on  $ \Ga $ and $ V ( X ) = \hat U ( X ). $   Second we use the fact that  $  \pi ( z )  = z $ and the boundary 
 Harnack inequality in   Lemma \ref{lem3.2} as in the proof of  Theorem \ref{thmA}  
 to  deduce  that   $ V $ is unique, homogeneous, and 
\begin{align}
  \label{5.28a}   
  V (z +  s  (\hat w - z)  )  = s^{\ep_1}    V (  \hat w  )  \mbox{   for some  $ \ep_1 > 0 $  whenever  $ \hat w  \in \rn{n} \sem \{z\},  s > 0$.} 
  \end{align}   
From the discussion  below the  definition of  $ \Ga $  we observe  that  $ \rn{n} \sem \Ga $  is contained  in  a   
translation  and  rotation of  $ K (  \al ) $  for some $ \al  \in ( 0, \pi) $  with  $ \pi  - \al \geq  \bar C^{-1}$.
  Using this fact,   Lemma \ref{lem3.2},  and  Theorem \ref{thmA}    we  see that  
  if   $ p > n -1, $ $ \ep_1 -   1 +  (n -  1  )/p  \geq C^{-1}. $   If  $ 1 <  p \leq n - 1, $ 
  one can use a  compactness argument or  an argument as in  \cite{KM}  to  show that  
$ \ep_1 \geq C^{ - 1}. $  Let $ \Ga_1 $  be the convex hull of   $   E  \cap \bar B  ( y,\rho/2  )  $ and $ z. $  
Also let  $   V^* $ be the  $ \mathcal{A}$-capacitary function for $\rn{n} \sem  \Ga_1 $  when $ 1 <  p <  n $  while $ V^* $ is the 
$ \mathcal{A}$-harmonic Green's function for  $  \rn{n} \sem \Ga_1 $  when $ p \geq n. $  
Define  $  \hat f$ and $\mathcal{\hat A}  $ as above \eqref{5.20} and observe that 
$  \hat V^* = 1 -  V^* $ is  $ \mathcal{\hat A}$-harmonic when $ 1 < p < n $  while  
$ \hat V^*  = V^* $ is  $  \mathcal{\hat A}$-harmonic when $ p \geq n  $  with  
continuous boundary value $ 0 $  on $ \Ga_1. $  We first let    
\[
 V'  = \frac{ \hat U ( X )}{ \hat V^* ( X ) } \,  \hat V^* 
 \]
 and claim that
\begin{align}
\label{5.29}  
\begin{split}
&(a)  \hs{.2in}    \hat U   \leq  C  V' \quad \mbox{in}\, \,     B ( y, 4\rho ) \sem  E .  \\  
&(b)  \hs{.2in}   V \approx  V'  \quad  \mbox{in}\, \,    B ( z,  \rho/4 ) \sem   \Ga .  \\
&(c) \hs{.2in}  \hat U  \approx  V' \approx V   \quad  \mbox{in}\, \,   B ( y,  \rho/8 )  \sem E \,.
\end{split}
\end{align}  
To  prove  \eqref{5.29} $(a)$  observe    from \eqref{2.8} that     
\[  
\max_{ \bar B ( y, 4 \rho ) }    \hat U   \leq   c  \, \hat  U ( X ) = c \,  V' ( X ) \leq c^2  \max_{\bar B (y, 4 \rho )}  V'   . 
\] 
 This inequality,  $  \Ga_1  \subset E, $ and the boundary maximum principle for 
  $  \mathcal{\hat A}$-harmonic functions give  \eqref{5.29} $(a)$.   
On the other hand,    \eqref{5.29} $(b)$   follows from   \eqref{2.8},  Harnack's  inequality for   $ \mathcal{\hat A}$-harmonic   functions,    
    Lemma  \ref{lem3.2},   and the fact that  
    $ \Ga_1 \cap \bar B ( z, \rho/ 2) =  \Ga \cap \bar B( z, \rho /2)$. Finally,
\eqref{5.29} $(c)$   follows from  these  inequalities and the fact that 
\[        \Ga  \cap  \bar B ( y, \rho/ 2)  =    \Ga_1 \cap \bar B ( y, \rho/ 2) =  E  \cap \bar B( y, \rho /2). \]
We  conclude from  \eqref{5.29}  that  
\begin{align}
\label{5.30}  \hat U  \leq  C  V   \mbox{ in }   B ( z,  \rho/4)  \sem E \quad  \mbox{while}\quad   \hat U  \approx  V  \, \,  \mbox{in}\, \,    B ( y,  \rho/8 ) \sem E.
\end{align}   
 If  $ \ti C  \geq  1 $  is large enough depending on $ \ti e $ and the data,    then from 
 \eqref{5.30},  the fact that $ \rho  \geq 100 \,  b, $    \eqref{5.26a},     Harnack's  inequality for 
 $ \mathcal{\hat A}$-harmonic functions,      and  Lemma  5.6    we deduce that    
\begin{align}
 \label{5.31} b^{-1}  \,   V  (    \pi ( y ) -   \ti C b e_n   )  \approx  b^{-1}  \,   \hat U (  \pi ( y ) -   \ti C b e_n    ) \approx  \min_{F}  | \nabla \hat U | 
 \end{align}
and    
\begin{align}
\label{5.32}  
b^{-1}  \,   \hat  U    (    \pi ( \hat x )  -   \ti C b e_n   )       \, \leq \,    C \, b^{-1}  V   ( \pi (\hat x)   -   \ti C b e_n    ).
\end{align} 
Next we  draw  the  line segment  $\hat l $  from  $ z  $  to    $   \pi ( y )  -    \ti C  b e_n. $     
From   similar  triangles and the definition of   $  \bar  \de  $ below  \eqref{5.28},  we see that    
$      \pi(\hat x )  -    \bar \de  \ti C  b  e_n     \mbox{ lies on $ \hat l $.  }  $
 From this  observation and   homogeneity of  $ V $ we   get      
\begin{align}
 \label{5.34}    
 V ( \pi(\hat x )     -  \bar \de  \ti C b e_n  )     = \bar\de^{\ep_1}  V (  \pi (y) -   \ti C  b \,e_n ).
 \end{align}     
Now   since $ \Ga $    is convex we can   repeat the argument  given  in    Lemma  \ref{5.6}    
with  $  \hat U  $    replaced   by   $V$  to get      \eqref{5.23}   with    $ \hat U $   replaced by     
$ V , $ $  w  $ by $  \pi (\hat x )  $, and   $  r $  by   $ \bar \de \ti C  b,   \ti C  b,  $ provided $ \ti C$ is large enough.         We obtain    
\begin{align}
 \label{5.35} 
\begin{split}
  \min_{ B ( \pi ( \hat x ) ,  \ti C b) \cap \partial\Ga} |\nabla  V |^p      & \approx   b^{-p}  V  (    \pi(\hat x )  -    \ti C  b  e_n   )^p, \\
\min_{ B  ( \pi ( \hat x ) , \bar \de \ti C    b)  \cap  \partial\Ga } |\nabla  V |^p      & \approx   (\bar \de  b)^{-p}  V  (    \pi(\hat x )   -      \bar \de  \ti C  b  e_n   )^p  \, .
\end{split}
\end{align}
From  Lemma \ref{5.6},  \eqref{5.31}-\eqref{5.35},   Harnack's  inequality for $ \mathcal{\hat A}$-harmonic functions,    we see that  
\begin{align}
 \label{5.36}  
 \begin{split}
 b^{1-n}  \,    \int_{  \De ( \hat x, b ) } |\nabla  \hat U |^p  d \mathcal{H}^{n-1} & \approx     \min_{ \De (  \hat x  ,   b)} |\nabla  \hat U |^p   \approx b^{-p}  \hat U  ( \pi (\hat x )  -    \ti C  b  e_n   )^p   \\  
& \leq    C \,   ( \bar \de  b)^{-p} \,  V ( \pi(\hat x )   -    \bar \de  \ti C  b  e_n   ) ^p  \\
& \leq  C'   \, \bar  \de^{( \ep_1 - 1) p} \,  \,   \min_{F}  | \nabla \hat U |^p
\end{split}
\end{align} 
where $  C, C' $  depend only on the  data and $ \ti e. $  Thus   Lemma \ref{5.8} is valid. 
\end{proof}  

To complete the proof  of  Theorem  \ref{thmB}  we need   Lemma  6.16  from  \cite {J}  which in our  situation can be  stated as following lemma.     
\begin{lemma} 
\label{lem5.9}  
With the same  notation as  in  Lemma   \ref{lem5.8}   choose   a coordinate system  with  axes  parallel to  the axes of  an  optimal inscribed ellipsoid  
contained  in  $ \ti S. $   Let   $  \mathcal{Q} $  be a  tiling of  $ \ti S $  by  closed   cubes $ \subset \sum_1 $    and  of   side-length  $ s   \leq b $  with  sides  parallel to  the  coordinate  axes.  If 
$ Q  \in  \mathcal{ Q}, $    let   $ Q^* $  be  the  cube  concentric to  $ Q $  with    side-length    $ 10 ( (n-1) ! )^2  s $    and  let 
\[  
\de^* ( Q ) =  \max_{y \in Q^* \cap \ti S}    \de ( y,  \ti S ).  
\]  
There exists  $ c (n) \geq 1$  such that   
\begin{align}
  \label{5.37}    
   \sum_{ \{ Q : \, \de^* (Q)  <   \si \} }   \mathcal{H}^{n-1} ( Q )     \leq   c ( n )  \si  \mathcal{H}^{n- 1} ( \ti S )
  \end{align}
where $ C  $  depends only  on the data and  $ \ti e. $     
\end{lemma}

Let  $ \ep_1 $   be as   in  Lemma     \ref{lem5.8} and put $ \ep_0  =  \ep_1  $  if  
 $  1 < p  \leq  n - 1 $   
while   $  \ep_0   =  \ep_1  - 1 + (n-1)/p $    if  $  p  >  n - 1  $.       To prove   \eqref{5.22}   
and thus complete the proof of Theorem \ref{thmB}  we   first  note  
 from   Lemma  \ref{lem5.9} that   if  $ \ep \in (0, 1), $    
\begin{align}
  \label{5.38}      
   \sum_{ \{ Q \in \mathcal{Q} \} }   \de^* (Q)^{- 1 + \ep}     \mathcal{H}^{n-1} ( Q )   \leq   C (\ep)    \mathcal{H}^{n-1} ( \ti  S  ),
  \end{align}   
as follows   from   summing  separately  over  cubes  $ Q \in \mathcal{Q} $  with  $ \de^* (Q)   \leq  2^{-k} s,  k = 0,  1, 2,  \dots  $  
Second  from  Lemmas \ref{5.6}, \ref{lem5.8},  we   deduce    that if  $ \hat y, \hat z  \in F $  and $ \pi ( \hat y ), \pi(\hat z)  \in  Q^* $,  then    
\[    
\max_{\hat \ze \in \{ \hat x,  \hat y \}}  b^{1-n}      \int _{\De ( \hat \ze , b) } |\nabla \hat U|^p   d \mathcal{H}^{n-1}  \leq C \min_{\hat \ze \in \{ \hat y, \hat z  \}}    
\de^{ - p + \ep_1 p } ( \hat \ze , F )   \min_F  |\nabla \hat U |^p  .   
\]   
Hence,     
\[   
b^{1-n}      \int _{\pi^{-1}(Q)\cap F}      |\nabla \hat U|^p  d\mathcal{H}^{n-1}  \leq  C   \de^* (Q)^{- p + p \ep_1}   \min_F  |\nabla \hat U |^p.  
\]            

 Using     \eqref{5.38}   with   $    s    =  b/2,  $  and the above   inequality   we conclude that   
\begin{align}
 \label{5.39}  
 \begin{split}
  \int_{F}  \de^{1 -  \ep_0}  ( \cdot, F)   \,   |  \nabla \hat U |^p   \,  d \mathcal{H}^{n-1}&  \leq  C      
  \sum_{Q \in \mathcal{Q}}  \de^* (Q)^{1-\ep_0}  \int_{\pi^{-1}(Q)\cap F}    |   \nabla  \hat U|^p  d\mathcal{H}^{n-1}    \\
&  \leq    C^2     \sum_{Q \in \mathcal{Q}}  \de^* (Q)^{1 - \ep_0   - p  + p \ep_1}  \mathcal{H}^{n-1} (Q)    \min_F   |\nabla \hat U |^p      \\
& \leq   C^3   \mathcal{H}^{n-1} ( F )    \min_F | \nabla \hat U |^p 
    \end{split}
    \end{align}
as we obtain  from  \eqref{5.38} if    $   1 < p \leq 2  $  or  $ n = 2, 3,  $ and $ p  >  2. $   Indeed ,   
    $ \de^*(Q)^{ 1 -  \ep_0   -  p  +  p \ep_1}   \leq \de^*(Q)^{  - 1  + (p-1) \ep_0}  $ if $ 1 < p  \leq 2 $  while  
 $ \de^*(Q)^{ 1 -  \ep_0   -  p  +  p \ep_1}  =  \de^*(Q)^{ 2 - n   + ( p - 1 ) \ep_0 }  \leq \de^*(Q)^{  - 1  + (p-1) \ep_0}$  
if  $ n = 2, 3$ and $ p > 2.$  Thus   \eqref{5.22} is valid  and the proof of  Theorem \ref{thmB}  is  now complete.

 \section{Closing Remarks}  
 \label{section6}     
 Here we discuss possible  generalizations  of  Theorems  \ref{thmA} and \ref{thmB}.   First can  any of the  hypotheses on $ f $  in  \eqref{1.5} $(a), (b), (c) $ be weakened or  even removed ?  For  example can $ (c)  $ be replaced by  the assumption that $ f $  is  $ C^2 $  in  $ \rn{n} \sem \{0\}? $  Does one  need  $  {\mathcal A}  =  \nabla f $ or is it enough to assume 
$  {\mathcal A} :  \rn{n}\sem \{0\}  \rar  \rn{n} $  is a homogeneous  $ p - 1 $ vector field  with  continuous first partials satisfying structure conditions similar to  1.5 $(b)$?  Does one really need uniform ellipticity in  $(b)?$ 
    We cannot   give a  quick answer to any of these questions,   still    we  note that existence and  uniqueness for $ u$  as   in Theorem  \ref{thmB} made important use of  boundary Harnack  inequalities  from  \cite{AGHLV} and   \cite{LN2}.   In  both references,  theorems are stated  for  an  $ f $  satisfying  \eqref{1.5}.  However \cite{AGHLV} is  concerned  with proving  boundary  Harnack  inequalities for  much more general   Lipschitz   domains.    Using  smoothness of  $ \ar  K ( \al),  $  it appears likely that at least for $ 0 <  \al < \pi,  $  it  would be enough to assume $ f $ has continuous second partials,   rather than Lipschitz second partials  in \eqref{1.5} $(c)$.    Also  in   \cite{LN2}  the emphasis was on  domains, a portion of    whose boundaries,  are $ k $ Reifenberg flat,    $  1 \leq  k < n - 1. $    If  $  2 \leq k  <  n - 1, $  the authors  of this paper needed    an assumption  similar to  \eqref{1.5}  in order to construct a   lower dimensional  barrier, which ultimately  provided a lower bound  for  a certain boundary  Harnack  inequality.  if  $ k = 1 $ though  these considerations can be avoided and one can use the same argument as in the  proof of  \eqref{4.18}  to   show  for example that   \eqref{3.20} holds.  Moreover,  this argument is  valid for  more general $ \mathcal{A} $   vector fields and  corresponding  $ \mathcal{A}$-harmonic  functions as outlined above.       The  proof  that   $  \la_1 (  \pi  )  =   1  -  \frac{(n-1)}{p}$  when  $  p  >  n - 1,$  made important  use of  Lemma \ref{lem4.1}.   This Lemma  was proved in  \cite{AGHLV}  for  $ n - 1  < p <   n $  and in  \cite{ALSV} for  $ n \leq p < \infty .$   In  both papers  it was assumed that   \eqref{1.5}  held for $ f $  primarily  in order to prove uniqueness in certain  Brunn-Minkowski type inequalities for  $\mathcal{A}$-capacity   and in the  proof  of  Theorems   \ref{thm1} and \ref{thm2}.     The proof of  Lemma \ref{lem4.1} in either  paper  follows  from a   Rellich type inequality,  which  could easily hold  if  $ f $ has continuous second partials  and  perhaps also is true   for more general  $ \mathcal{A}$ than when $\mathcal {A} = \nabla f$.

 In the  proof  of  Theorem \ref{thmB}   we first  show in \eqref{5.18}    that  $ \ar E $  is  locally  $ C^{1, \al'}$  when $   \He $ is  bounded  above  and  below  on   $\mathbb S^{n-1}$.    We then  assumed  $  \He \in  C^{0, \hat \al}    ( { \mathbb S^{n-1} })   $ in order  to  complete the proof  of  Theorem \ref{thmB} when $ k = 0. $  If instead  of 
   $  \He \in  C^{0, \hat \al}    ( { \mathbb S^{n-1} }),$    one  assumes  that  $ \He $ is only  continuous on    $ { \mathbb S^{n-1} },    $   then one can use results from   \cite{C3}  to conclude that  $  \ar  E $ is locally  $ W^{2, q} $  for $ 1 < q  < \infty.$ As for possible generalizations of  Theorem \ref{thmB},  recall that  Theorem  \ref{thmA}  was  used   only to prove Lemma \ref{lem5.8}, which was then used  in  \eqref{5.39}.  In \eqref{5.39}  one needs    to   estimate   the second  line of  this display   from  above by the left-hand  side of  \eqref{5.38},   which  is only  possible for the values of  $ p $ in  Theorem \ref{thmB},  as explained below  \eqref{5.39}.      
Moreover,  \eqref{5.39}  was the  final  step in the proof of   \eqref{5.22}   which was the key  inequality needed  to  eventually  conclude  Theorem  \ref{thmB}.     Examples  from  Section 8  of   \cite{J},   show  that for  $ p = 2$,   the exponent $ \ep_0  $    in \eqref{5.22}  is dependent on the  eccentricity of  $ E.$   Using  Theorem \ref{thmA} and proceeding operationally  it  appears likely   that  the  same example   implies   when   $   n \geq 4$ and $p \geq   1  -  
(n-1)/p$ that 
\begin{align}
\label{6.1}  
\lim_{t \to 0}  \frac{  { \ds \int_{F}} \de ( x, F)   d { \mathcal H}^{n-1} }{\,  {\mathcal H}^{n-1} ( F )    \,  { \ds \min_F |\nabla  \hat U |^p}}  =  \infty
\end{align}  
where $ 0 \in \ar E $  and  $ F = \{x \in \ar{E} :  x_n < t \}. $   To briefly outline this example,   
 let  $ x'  =  ( x_1, \dots, x_{n-1} ) $  and let 
 \[    
 Q  =  \{x :\,  | x_i| \leq s, 1\leq i \leq n - 1,  x_n = t_1  \}\quad \mbox{and}\quad  I   =  \{ ( x_1, 0, \dots, 0) :\,  - 1 \leq x_1 \leq 1 \},  
 \] 
where  $0 <  t_1  <  < s < < 1.$    Let  $ E $  be  the  convex hull  of  $ Q \cup  I. $     Repeating the  argument in 
  \cite{J}    through display (8.3)  with  $ h  =  |\nabla \hat U | $   one uses Theorem  \ref{thmA}  to get
\begin{align}
\label{6.2}   
\min_{B(x_1, t)} |\nabla \hat U | \geq   c^{-1}  ( 1 - |x_1| )^{ - (n-1)/p \, +  \, \eta  } \quad \mbox{and}\quad       \min_{B(x_1, t) \cap  F }   \de ( x, F )  \geq     c^{-1}  ( 1 -  |x_1| )     
\end{align}
for  $x_1 \in  I$   where  $  \eta > 0 $  can be arbitrarily small and  $ c,$ may depend on $ s$, $t_1, $  and the data.  Now 
  \[  
  \pi ( F )  =  \{ x :   |x_1| < 1 -  \frac{t}{t_1}, \, \,   |x_i| <\frac{ts}{t_1}, \,   2 \leq i  \leq n -1, \,  x_n  =  t \}. 
  \]      
  Using    this  observation and  \eqref{6.2}  it follows  that 
\begin{align}
\label{6.3}   
\int_{F} \de ( x, F )  | \nabla \hat U |^p  \,  d {\mathcal H}^{n-1} \geq c^{-1} t^{n-2}   \int_0^{1 -  2t/t_1}  ( 1 - x_1 )^{ 2 - n + p  \eta } dx_1 =  c_+^{-1} t^{ 1  + p \eta}. 
\end{align}      
  Moreover, the  argument in  \cite{J}  can be used  to  get 
\begin{align}
\label{6.4}   {\mathcal H}^{n-1} (F) \,   \min_{F}  | \nabla \hat U |   \,  \leq \,  \ti c   \, t^{n -  2 - \eta_1}
\end{align}
  where $  \eta_1 > 0  $  can also be  arbitrarily small  and  $ \ti c,  c_+,$  have the same dependence as $ c $ in  \eqref{6.2}.     
 Choosing $ \eta, \eta_1,  $   small enough and then fixing  $ s, t_1,  $   we deduce   \eqref{6.1} from  \eqref{6.3}, \eqref{6.4}.

 We note,  however  that  $ \mu,  $  as defined in   Theorems  \ref{thm1}, \ref{thm2},  corresponding to  $ E$ in the example above,   does  not satisfy the  hypotheses of  Theorem  \ref{thmB}  since for example   $  \mu ( \{e_n\} )  >  0. $   Thus it is an interesting open question whether  Theorem  \ref{thmB}  remains true when    $ 2  < p < \infty $  and   $ n  \geq 4. $ Perhaps  one  should  first  try  to  answer  this question  under the  additional assumption   that   $  \|  \He  - 1 \|_{\infty}   \leq  \ep ,  \ep >0,$  small,  since   Theorems \ref{thm1},  \ref{thm2},  give  $ E  = $ ball when   $ \He  \equiv 1.$  Somewhat similar questions have recently been considered in  generalizations of  the work of  Caffarelli in \cite{C3} on the  Monge-Amp{\`e}re equation (see Theorems 3.13 and 3.14 and Corollary 3.15 in \cite{Figalli18}).

\section*{Acknowledgement} 
Part  of  this  research was done while the  second author was visiting TIFR  in Bangalore India.  The second author thanks  TIFR for their gracious hospitality.

\bibliographystyle{amsbeta}

\bibliography{refALV}  

\end{document}